\documentclass{amsart}

\usepackage[hidelinks]{hyperref}
\usepackage{mathrsfs}
\usepackage{enumitem}
\usepackage{graphicx,color}
\usepackage{amssymb}
\usepackage{bm}
\usepackage{adjustbox}	
\usepackage{scalerel}	
\usepackage{microtype}
\usepackage{lmodern}

\newtheorem{defi}{Definition}[section]
\newtheorem{prop}[defi]{Proposition}
\newtheorem{thm}[defi]{Theorem}
\newtheorem{lem}[defi]{Lemma}
\newtheorem{cor}[defi]{Corollary}

\newtheorem{qu}{Question}

\numberwithin{equation}{section}

\newcommand{\N}{\mathbf{N}}
\newcommand{\Z}{\mathbf{Z}}
\newcommand{\R}{\mathbf{R}}
\newcommand{\C}{\mathbf{C}}
\newcommand{\bS}{\mathbf{S}}

\newcommand{\cD}{\mathscr D}

\newcommand{\cH}{\mathscr H}		

\newcommand{\cL}{\mathscr L}		

\newcommand{\cR}{\mathscr R}

\newcommand{\bM}{\operatorname{\mathbf M}}

\newcommand{\Lip}{\operatorname{Lip}}

\newcommand{\spt}{\operatorname{spt}}

\newcommand{\dist}{\operatorname{dist}}
\newcommand{\defl}{\mathrel{\mathop:}=}
\newcommand{\defr}{=\mathrel{\mathop:}}

\newcommand{\B}{\textbf B}		
\newcommand{\oB}{\textbf U}		

\newcommand{\curr}[1]{[\![{#1}]\!]}
\newcommand{\id}{\mathrm{id}}

\newcommand{\res}{\scaleobj{1.7}{\llcorner}}
\newcommand{\BigWedge}{\mathord{\adjustbox{valign=B,totalheight=.6\baselineskip}{$\bigwedge$}}}

\makeatletter
\newcommand*{\cone}{%
	{%
		\mathpalette\@coneOf{\times}%
	}%
}
\newcommand*{\@coneOf}[2]{%
	\sbox0{$\m@th#1\mathsf{#2}$}%
	\mathsf{#2}%
	\kern-\wd0 %
	\mkern2.00mu\relax
	\nonscript\mkern0.25mu\relax
	\mathsf{#2}%
}
\makeatother

\title[The hemisphere is almost calibrated]{The Riemannian hemisphere is almost calibrated in the injective hull of its boundary}

\author{Roger Z\"{u}st}
\email{roger\_zuest@hotmail.com}
\thanks{This work was carried out while the author was at the University of Bern.}

\begin{document}

\begin{abstract}
An exact differential two-form is constructed in the injective hull of the Riemannian circle, whose comass norm, defined via the inscribed Riemannian area on normed planes, is stationary at every point of the open hemisphere spanned by the circle. As a consequence, in any metric space, the induced Finsler mass of a two-dimensional Ambrosio-Kirchheim rectifiable current with boundary a Riemannian circle of length $2\pi$ admits a lower bound of $2\pi$ plus a second-order term in the Hausdorff distance to an isometric copy of the hemisphere. This estimate applies to all oriented Lipschitz surfaces spanning the circle, regardless of their topology, and thus offers positive evidence for Gromov's filling area conjecture.
\end{abstract}

\maketitle
\tableofcontents


\newpage

\section{Introduction}

Let $\bS^1$ be the Riemannian circle of length $2\pi$ equipped with the intrinsic geodesic distance $d$. A filling of $\bS^1$ is a compact, oriented Riemannian surface $M$ with intrinsic distance $d_M$ such that the restriction $(\partial M, d_M|_{\partial M \times \partial M})$ is isometric to $\bS^1$. Whether $\operatorname{Area}(M) \geq 2\pi$ for such $M$ is an open question posed by Gromov \cite[\S2.2]{G}. Equality does hold for the Riemannian hemisphere of constant curvature $1$, from now on denoted by $\bS^2_+$. There are some partial answers available in the literature. As a consequence of Pu's systolic inequality \cite{P}, it is true that $\operatorname{Area}(M) \geq 2\pi$ whenever $M$ is a Riemannian disk that fills $\bS^1$. There is also a generalization due to Ivanov \cite{Iv2} for Finsler disks $M$ in case $\operatorname{Area}(M)$ is interpreted as the Holmes-Thompson or Busemann-Hausdorff definition of area. It is shown in \cite{BCIK} that this lower bound holds for any Riemannian surface $M$ of genus 1 that fills $\bS^1$. The question is widely open for surfaces of higher genus. In this paper, we propose a more general approach to the problem using the theory of metric currents developed by Ambrosio and Kirchheim \cite{AK}. Two-dimensional real rectifiable currents $\cR_2(X)$ in a metric space $X$ are generalizations of compact, oriented Lipschitz surfaces. In this setting, the question strongly depends on the particular notion of Finsler area used. As we will indicate in Subsection~\ref{sec:finslermass}, a Finsler area induces a corresponding Finsler mass on real rectifiable currents. More generally, this applies to rectifiable sets and to rectifiable chains in metric spaces with coefficients in a normed abelian group as introduced by De Pauw and Hardt \cite{DPH}. For example, the mass of Ambrosio-Kirchheim currents corresponds to the Gromov-mass$\ast$ (or Benson) area. The inscribed Riemannian area $\mu^{\rm ir}$, introduced by Ivanov \cite{Iv}, is the largest possible choice, see Lemma~\ref{lem:masslem}. So, if Gromov's filling area conjecture is true with respect to some Finsler area, it is also true with respect to the inscribed Riemannian area. The mass on real rectifiable currents corresponding to $\mu^{\rm ir}$ is denoted by $\bM_{\rm ir}$. A possible formulation of Gromov's conjecture is thus:

\begin{qu}
	\label{qu:mainquestion}
Let $X$ be a complete metric space and $T \in \cR_2(X)$ be a real rectifiable current with boundary $\partial T$ isometric to $\curr {\bS^1}$. Is it true that $\bM_{\rm ir}(T) \geq 2\pi$, with equality if and only if $T$ is isometric to $\curr{\bS^2_+}$?
\end{qu}

Note that in case $\curr M$ is induced by a compact, oriented Riemannian surface $M$, then $\bM_{\rm ir}(\curr M)$ agrees with the usual area of $M$. Instead of working with an arbitrary metric space $X$, it is sufficient to consider the injective hull $E(\bS^1)$ of $\bS^1$. Roughly speaking, this is the smallest injective metric space that contains $\bS^1$ isometrically. It is interesting to note that $E(\bS^1)$ contains a unique isometric copy of $\bS^2_+$, see Lemma~\ref{lem:isometricembedding}. Injective hulls were introduced independently by Isbell \cite{I} and Dress \cite{D}. An injective metric space $Y$ has the defining property that, whenever $\varphi : A \to Y$ is a $1$-Lipschitz map defined on a subset $A \subset X$ of a metric space $X$, there exists a $1$-Lipschitz extension $\bar \varphi : X \to Y$. So whenever $T \in \cR_2(X)$ has boundary $\partial T$ isometric to $\curr{\bS^1}$ in a metric space $X$, there exists a $1$-Lipschitz map $\bar \varphi : X \to E(\bS^1)$ such that $\bar \varphi|_{\spt(\partial T)}$ is an isometry. The pushforward $\bar \varphi_\# T \in \cR_2(E(\bS^1))$ is also a filling of $\curr{\bS^1}$ with $\bM_{\rm ir}(\bar \varphi_\# T) \leq \bM_{\rm ir}(T)$. The injective hull of the sphere $\bS^n$ can be characterized explicitly as those $1$-Lipschitz functions $f : \bS^n \to \R$ with $f(x) + f(-x) = \pi$, see Proposition~\ref{prop:injectivehull}. By fixing a base point and an orientation of $\bS^1$ we can identify $E(\bS^1)$ with the space of $1$-Lipschitz functions $f : \R \to \R$ that satisfy $f_{\alpha + \pi} + f_\alpha = \pi$ for all $\alpha \in \R$. This space is contained isometrically as a compact and convex subset of $L^\infty([0,\pi))$. Although $E(\bS^1)$ spans an infinite dimensional subspace, it may be possible to find a calibration for the isometric copy of $\bS^2_+$ that sits in $E(\bS^1)$ by employing a notion of differential form in an infinite dimensional setting. Such a calibration would answer Gromov's question in the positive. In this direction, we study in detail the differential two-form defined by
\begin{equation}
	\label{eq:defomega}
\tilde\omega_f \defl \frac{1}{\pi}\int_0^{\pi}\int_\alpha^{\pi} p_{\alpha,\beta}(f)\, d\pi_\alpha \wedge d\pi_\beta\,d\beta\,d\alpha
\end{equation}
with coefficients
\[
p_{\alpha,\beta}(f) = \frac{1 - \cos(\beta-\alpha)^2 - \cos(f_\alpha)^2 - \cos(f_\beta)^2 + 2 \cos(\beta - \alpha)\cos(f_\alpha)\cos(f_\beta)}{\sin(\beta-\alpha)^2\sin(f_\alpha)^2\sin(f_\beta)^2} .
\]
This definition is motivated by the differential form in \cite{BI} used to show that planes contained in a normed space are calibrated with respect to the Hausdorff measure. Further justifications for this particular form are given following the statement of the main theorems. First we clarify the notation used in the definition of $\tilde \omega$. If $E(\bS^1)$ is realized as subset of $L^\infty([0,\pi))$, we adopt $\pi_\alpha : L^\infty([0,\pi)) \to \R$ to denote the coordinate projections $\pi_\alpha(g) = g_\alpha$ for almost every $\alpha$. Although ill-defined as proper linear functionals for fixed $\alpha$, in contrast to say $\pi_\alpha : C([0,\pi]) \to \R$, the definition of $\tilde \omega_f$ as an integral is meaningful.

The coefficients $p_{\alpha,\beta}(f)$ are well-defined because any $f \in E(\bS^1)\setminus \bS^1$ takes values in $(0,\pi)$. If $T \in \bM_2(E(\bS^1))$ is a metric current with finite mass and support away from $\bS^1$, the action $T(\tilde\omega)$ is defined by integrating $T(p_{\alpha,\beta}(f)\,d\pi_\alpha\wedge d\pi_\beta)$ with respect to $\alpha$ and $\beta$, see Subsection~\ref{sec:actioncurrents}. This integral makes sense because the coefficients are nonnegative and uniformly bounded on $\spt(T)$, due to an interpretation of $p_{\alpha,\beta}(f)$ in spherical geometry, see Lemma~\ref{lem:geometriccoeff}. Furthermore, the comass $\|\tilde\omega_f\|_{\rm ir}$ of $\tilde \omega_f$ is given by the infimum over all $M \geq 0$ such that
\[
|\tilde\omega_f(v \wedge w)| \leq M \mu^{\rm ir}(v \wedge w)
\]
for all $v,w \in L^\infty([0,\pi))$. Here $\mu^{\rm ir}(v \wedge w)$ is the inscribed Riemannian area of the parallelogram spanned by $v$ and $w$. Calibrations, as defined by Harvey and Lawson \cite{HL}, are special differential forms on Riemannian manifolds. They are exact—or closed, depending on the setting—and have comass equal to 1. Here are the essential reasons for the particular definition of $\tilde \omega$ in \eqref{eq:omegadef}:
\begin{itemize}
	\item $\tilde \omega$ calibrates $\bS^2_+$ in the sense that $|\tilde \omega_f(v \wedge w)| \leq \mu_f(v \wedge w)$ for arbitrary $f \in \bS^2_+ \setminus \bS^1$ and $v,w \in L^\infty([0,\pi))$ with equality if and only if $v$ and $w$ are in the tangent space of $\bS^2_+$ at $f$ (or are linearly dependent). Because we use the inscribed Riemannian area, instead of possible other definitions of Finsler area, this statement is a consequence of the isoperimetric inequality for plane paths.
	\item $\tilde \omega$ is closed, i.e., $d\tilde \omega = 0$, because $p_{\alpha,\beta}(f)$ depends only on $f_\alpha$ and $f_\beta$, and not on $f_\gamma$ for any $\gamma \in [0,\pi) \setminus \{\alpha,\beta\}$. Since $E(\bS^1)$ is convex, $\tilde \omega$ is exact.
\end{itemize}

We obtain that $\|\tilde \omega_f\|_{\rm ir}$ is close to $1$ in case $f$ is close to $\bS^2_+$. This is a consequence of the stability of the isoperimetric inequality. The precise statement we obtain is the following:

\begin{thm}
	\label{thm:mainthm1}
For every $r > 0$ and $\xi \in (1,2)$, there exists $C > 0$ such that for all $f \in E(\bS^1)$ with $\dist(f,\bS^1) \geq r$, the following estimate holds:
\[
\left|\|\tilde\omega_f\|_{\rm ir} - 1\right| \leq C \dist(f,\bS^2_+)^{\xi} .
\]
\end{thm}

In this sense, even though $\tilde\omega$ seems not be a global calibration, it is almost a calibration near the hemisphere $\bS^2_+$. This allows to estimate the filling area of $\bS^1$ among surfaces that are close to $\bS^2_+$ with respect to the Gromov-Hausdorff distance.

\begin{thm}
	\label{thm:mainthm2}
For every $r \in (0,\frac{\pi}{2})$ and $\xi \in (1,2)$, there exists $C > 0$ such that the following holds. Let $X$ be a metric space, and let $S,T \in \cR_2(X)$ be real rectifiable currents with compact support. Assume that:
\begin{enumerate}
	\item $S$ is isometric to $\curr{\bS^2_+}$.
	\item $\partial T = \partial S$ (which is isometric to $\curr{\bS^1}$ by (1)).
	\item $T \res N_r = S \res N_r$ for the $r$-neighborhood $N_r$ of $\bS^1$ inside $\spt(S)$.
\end{enumerate}
Then
\[
\bM_{\rm ir}(T) \geq 2\pi - C h(\spt(T),\spt(S))^\xi .
\]
\end{thm}

Here $h(A,B)$ denotes the directed Hausdorff distance from $A$ to $B$, defined by
\[
h(A,B) \defl \mathop{\sup}_{a \in A} \inf_{\raisebox{-0.56ex}{$\scriptstyle b \in B$}} d(a,b).
\]
The key features of the theorem are that $\xi > 1$ and that competing surfaces may have arbitrary topological type. To emphasize this, take $X = E(\bS^1)$ and $0 < r < \frac{\pi}{2}$ and consider the open set
\[
U \defl \{x \in X : d(\bS^1,x) > r > d(\bS_+^2,x)\} \subset E(\bS^1).
\]
Let $Z \in \cR_2(X)$ be an arbitrary cycle in $U$, that is, $\spt(Z) \subset U$ and $\partial Z = 0$. Then the mass bound of the theorem applies to $T = \curr{\bS_+^2} + Z$. In particular, any metric surface obtained by attaching handles to $\bS^2_+$ inside $U$ is of this type. Consequently, the hemisphere $\bS^2_+$ is stationary under variations by surfaces of arbitrary topological type, provided a collar neighborhood of the boundary is fixed.

\begin{cor}
	\label{cor:maincor}
Assume that $X$ is a metric space and that $r \in \left(0, \frac{\pi}{2}\right)$ is fixed. Let $T_n, S \in \cR_2(X)$ for $n \in \N$, and suppose that:
\begin{enumerate}
	\item $S$ is isometric to $\curr{\bS^2_+}$.
	\item $\partial T_n = \partial S$.
	\item $T_n \res N_r = S \res N_r$, where $N_r$ denotes the $r$-neighborhood of $\spt(\partial S)$ inside $\spt(S)$.
	\item $\lim_{n \to \infty} h(\spt(T_n), \spt(S)) = 0$ and $T_n \neq S$ for all $n$.
\end{enumerate}
Then
\[
\limsup_{n \to \infty} \frac{\bM_{\rm ir}(S) - \bM_{\rm ir}(T_n)}{h(\spt(T_n), \spt(S))} \leq 0.
\]
In particular, $\liminf_{n \to \infty} \bM_{\rm ir}(T_n) \geq 2\pi$.
\end{cor}

We now provide a brief overview of the proof of Theorem~\ref{thm:mainthm1}. The hemisphere
\[
\bS^2_+ \defl \left\{(x,y,z)\in \R^3 : x^2 + y^2 + z^2 = 1, z \geq 0\right\}
\]
with its intrinsic length metric induced by the standard Euclidean distance of $\R^3$ is represented uniquely in $E(\bS^1)$ as those functions $h : \R \to \R$ with $h_\alpha = \arccos(\cos(d)\cos(\alpha - \tau))$ for parameters $\tau \in (-\pi,\pi]$ and $d \in [0,\frac{\pi}{2}]$, Lemma~\ref{lem:isometricembedding}. Its boundary, the representation of $\bS^1$ inside $E(\bS^1)$, corresponds to those functions with $d = 0$. The coefficient function $p_{\alpha,\beta}$ of $\tilde \omega$ at a point $h \in \bS^2_+\setminus \bS^1$, with parameters $\tau$ and $d$, possess the product structure
\[
p_{\alpha,\beta}(h) = \frac{\sin(d)}{1-\cos(d)^2\cos(\alpha - \tau)^2}\frac{\sin(d)}{1-\cos(d)^2\cos(\beta - \tau)^2} = p_\alpha(h)p_\beta(h) ,
\]
with
\begin{equation}
\label{eq:intro1}
\int_0^{2\pi}p_{\alpha}(h)\,d\alpha = 2\pi ,
\end{equation}
see Lemma~\ref{lem:tangentsphere} and Lemma~\ref{lem:geometriccoeff}. The inscribed Riemannian comass of $\tilde\omega$ at $f \in E(\bS^1)\setminus \bS^1$ can be expressed as
\[
\|\tilde\omega_f\|_{\rm ir} = \sup \left\{\tilde\omega_f(v\wedge w) : (v,w) \in L^\infty([0,\pi))^2, \|v^2 + w^2\|_\infty \leq 1\right\} ,
\]
see Proposition~\ref{prop:massbound}. We will also write $\tilde \omega_f(\gamma)$ for paths $\gamma = (v,w) : [0,\pi) \to \R^2$ as above. They are extended to $[0,2\pi)$ by $\gamma(t+\pi) = -\gamma(\pi)$, reflecting a symmetry inherited from the structure of $E(\bS^1)$. Because of the product structure of $p_{\alpha,\beta}(h)$ for $h \in \bS^2_+\setminus \bS^1$, the maximization problem for $\|\tilde\omega_h\|_{\rm ir}$ reduces to the classical isoperimetric inequality in the plane, as captured in Lemma~\ref{lem:prodtype}. Up to rotations of $\R^2$, there exists a unique maximizer $\gamma$ attaining $\|\tilde\omega_h\|_{\rm ir} = 1$; moreover, $\gamma$ parametrizes the unit circle. Consequently, there exists a unique bi-Lipschitz function $\nu_h : [0,\pi) \to [0,\pi)$ such that $\tilde\omega_h(e^{i\nu_h}) = \|\omega_{h}\|_{\rm ir} = 1$. 

For arbitrary $f \in E(\bS^1)\setminus \bS^1$, the existence of a maximizing path $(v,w)$ follows from the Banach-Alaoglu theorem, using the weak$\ast$ compactness of the unit ball in $L^\infty([0,\pi)) = L^1([0,\pi))^\ast$, see Lemma~\ref{lem:compactness}. Under mild assumptions on $f$, any such maximizer additionally satisfies  $v_\alpha^2 + w_\alpha^2 = 1$ for almost every $\alpha$, Lemma~\ref{lem:gammaboundary}. Given $h \in \bS^2_+$, a maximizing path for $\tilde\omega_f$ can thus be written in the form $e^{i(\nu_h + \eta)}$ for some function $\eta$ in
\[
L_0^2 \defl \left\{\eta \in L^2([0,\pi)) : \int_0^\pi \eta = 0 \right\}.
\]
The value $\|\omega_f\|_{\rm ir}$ is thus attained as the maximum of $\eta \mapsto \Psi_h(f,\eta)$, where
\[
\Psi_h : E(\bS^1)\setminus \bS^1 \times L_0^2 \to \R
\]
is defined by
\[
\Psi_h(f,\eta) \defl \frac{1}{\pi}\int_0^\pi\int_\alpha^\pi p_{\alpha,\beta}(f) \sin(\nu_h(\beta) - \nu_h(\alpha) + \eta(\beta) - \eta(\alpha))\,d\beta\,d\alpha .
\]
The space $L_0^2$ is natural here because second variations of $\Psi_h$ with respect to $\eta$ are naturally controlled by the $L^2$-norm, and the zero-mean condition eliminates the rotational invariance by fixing the phase.

To conclude Theorem~\ref{thm:mainthm1}, we require an implicit function theorem. Specifically, we need a map $f \mapsto \eta_f$ such that $\eta_f \in L^2_0$ is a maximizer of $\eta \mapsto \Psi_h(f,\eta)$ and $\|\eta_f\|_\infty$ depends continuously on $\|f-h\|_\infty$. This is established in Lemma~\ref{lem:optimallinfinity}. The crucial ingredient is the stability of the isoperimetric inequality for planar paths, in the form established by Fuglede \cite{Fu}. This yields the estimate $\|\eta_f\|_\infty \leq C(h,\xi) \|f - h\|_\infty^{\xi/2}$ for some maximizer $\eta_f$ in case $\|f - h\|_\infty$ is small enough. Moreover, if $\varepsilon$ is sufficiently small and $\|f-h\|_\infty \leq \varepsilon$, then $\eta \mapsto \Psi_h(f,\eta)$ is strictly concave in the sense that
\[
\frac{d^2}{dt^2}\Psi_h(f,(1-t)\eta_0 + t\eta_1) \leq -c\|\eta_1 - \eta_0\|_2^2,
\]
for some $c(h) > 0$ and all $\eta_0,\eta_1 \in L^2_0 \cap \B^\infty(0,\varepsilon)$, $t \in [0,1]$, Lemma~\ref{lem:strictconcavity}. Consequently, $\eta_f$ is the unique maximizer in $L^2_0 \cap \B^\infty(0,\varepsilon)$. The remaining task is to reconcile the roles of the $L^\infty$ and $L^2$ norms appearing above, in order to derive the refined bound $\|\eta_f\|_\infty \leq C(h,\xi) \|f - h\|_\infty^\xi$ for any $\xi \in (0,1)$. This is established in Proposition~\ref{prop:strictconcavity2} and allows us to perform actual variations of the function $f \mapsto \Psi_h(f,\eta_f)$ at $f = h$, thereby completing the proof of Theorem~\ref{thm:mainthm1}. The second main Theorem~\ref{thm:mainthm2} is a direct consequence of the first one and the results about rectifiable currents and Finsler mass in Subsection~\ref{sec:finslermass}.

\section{Setting}

\subsection{Metric currents and Finsler mass}
\label{sec:finslermass}

Metric currents, as introduced by Ambrosio and Kirchheim \cite{AK}, are functionals acting on tuples of Lipschitz functions and generalize the classical Euclidean currents originally developed by Federer and Fleming \cite{FF,F} and de Rham \cite{DR}. Since we only need currents with compact support, an equivalent definition is due to Lang \cite{L}. See \cite[Definition~2.2]{Z} for the same set of axioms used below. 

\begin{defi}[Metric currents with compact support]
	\label{metric_current_def}
	Let $X$ be a metric space and $n \geq 0$. A multilinear functional $T : \Lip(X)^{n+1} \to \R$ is a current in $\cD_n(X)$ if the following axioms hold:
	\begin{enumerate}
		\item $T(f,g_1,\dots,g_n) = 0$ if some $g_i$ is constant in a neighborhood of $\spt(f)$.
		\item $\lim_{k\to\infty} T(f_k,g_{1,k},\dots,g_{n,k}) = T(f,g_1,\dots,g_n)$ if $f_k \to f$, $g_{i,k} \to g_i$ uniformly for all $i$ and $\sup_{i,k}\{\Lip(f_k),\Lip(g_{i,k})\} < \infty$.
		\item There exists a compact set $K \subset X$ such that $T(f,g_1,\dots,g_n) = 0$ whenever $\spt(f) \cap K = \emptyset$.
	\end{enumerate}
\end{defi}

The support $\spt(T)$ of $T$ is the intersection of all closed sets $A \subset X$ with the property that $T(f,g_1,\dots,g_n) = 0$ whenever $\spt(f) \cap A = \emptyset$. See \cite[Lemma~2.3]{Z} for more details on the support related to the axioms above. Assuming $n \geq 1$, the boundary $\partial T \in \cD_{n-1}(X)$ of $T \in \cD_n(X)$ is defined by
\[
\partial T(f,g_1,\dots,g_{n-1}) \defl T(1,f,g_1,\dots,g_{n-1}) .
\]
If $\varphi : X \to Y$ is a Lipschitz map between metric spaces, then the pushforward $\varphi_\# : \cD_n(X) \to \cD_n(Y)$ is defined by
\[
(\varphi_\# T)(f,g_1,\dots,g_{n}) \defl T(f\circ \varphi ,g_1 \circ \varphi,\dots,g_{n}\circ \varphi) .
\]
The mass of a current $T \in \cD_n(X)$ is defined by
\[
\bM(T) \defl \sup \sum_{\lambda \in \Lambda} T\bigl(f_\lambda,g_{1,\lambda},\dots,g_{n,\lambda}\bigr) ,
\]
where the supremum is taken over all finite collections $\Lambda$ such that $(f_\lambda,g_{1,\lambda},\dots,g_{n,\lambda})$ is in $\Lip(X)^{n+1}$, each $g_{i,\lambda}$ is $1$-Lipschitz and $\sum_{\lambda \in \Lambda} |f_\lambda| \leq 1$.

For example if $\theta \in L^1(\R^n)$ has (essentially) compact support, then $\curr{\theta} \in \cD_n(\R^n)$ is defined by integration
\[
\curr{\theta}(f,g_1,\dots,g_n) \defl \int_{\R^n} \theta(x)f(x)\det(D(g_1,\dots,g_n)_x)\,dx , 
\]
and satisfies $\bM(\curr{\theta}) = \int_{\R^n}|\theta(x)|\,dy$. This is justified by \cite[Example~3.2]{AK}.

By combining \cite[Theorem~4.5]{AK} with \cite[Lemma~4]{K}, rectifiable currents can be characterized as follows.

\begin{defi}
	\label{def:rectifiable}
$T \in \cD_n(X)$ is an $n$-dimensional real rectifiable current in $\mathcal R_n(X)$ if it has finite mass and for any $\lambda > 1$ there exist a sequence $K_i$ of compact sets in $\R^n$, functions $\theta_i \in L^1(K_i)$, norms $\|\cdot\|_i$ on $\R^n$ and maps $\varphi_i : K_i \to K$ into some compact set $K \subset X$ such that the sets $\varphi_i(K_i)$ are pairwise disjoint,
\[
\lambda^{-1}\|x-y\|_i \leq d(\varphi_i(x),\varphi_i(y)) \leq \lambda\|x-y\|_i ,
\]
\[
T = \sum_{i=0}^\infty \varphi_{i\#} \curr{\theta_i} \qquad \text{and} \qquad \bM(T) = \sum_{i=0}^\infty \bM(\varphi_{i\#} \curr{\theta_i}) .
\]
\end{defi}

Next we want to define a notion of Finsler mass on rectifiable currents that depends on a specific definition of volume. As we will see, the Ambrosio-Kirchheim mass is induced by the Gromov-mass$\ast$ (or Benson) volume.

\begin{defi}
	\label{def:area}
	Given $n \in \N$, a Finsler volume assigns to every $n$-dimensional normed space $V$ a Haar measure $\mu_V$ with the properties:
	\begin{enumerate}
		\item If $V$ and $W$ are $n$-dimensional normed spaces and $A : V\to W$ is a linear map with $\|A\|\leq 1$, then $A$ is volume decreasing, that is, $\mu_W(A(B)) \leq \mu_V(B)$ for all Borel sets $B \subset V$.
		\item If $V$ is Euclidean, then $\mu_V$ is the standard Lebesgue measure.
	\end{enumerate}
\end{defi}

This is equivalent to the definition given in \cite[\S3]{APT}, where, instead of a Haar measure, a norm, also denoted by $\mu_V$, is assigned to the one-dimensional space $\bigwedge_n V$. The equivalence is induced by the identity
\[
\mu_V(P(v_1,\dots,v_n)) = \mu_V(v_1 \wedge \cdots \wedge v_n) ,
\]
where $P(v_1,\dots,v_n)$ is the parallelepiped spanned by the vectors $v_1,\dots,v_n \in V$.

If $s$ is a seminorm on $\R^n$ with standard basis $e_1,\dots,e_n$, the Jacobian of $s$ is
\begin{equation}
	\label{eq:defjacobian0}
\mathbf{J}_{\mu}(s) \defl
\left\{
\begin{array}{ll}
\mu_s(e_1 \wedge \cdots \wedge e_n)  & \text{if $s$ is a norm} , \\
0 & \text{otherwise} .
\end{array}
\right.
\end{equation}
Or equivalently, in case $s$ is a norm,
\begin{equation}
\label{eq:defjacobian}
\mathbf{J}_{\mu}(s) = \frac{\mu_s(B)}{\cL^n(B)}
\end{equation}
for every Borel set $B \subset \R^n$ with positive and finite Lebesgue measure $\mathcal{L}^n(B)$.

Let $(V,\|\cdot\|)$ be a normed space of dimension $n$ with unit ball $\B_V$ and dual space $(V^\ast,\|\cdot\|^\ast)$. $\bm\alpha(n)$ denotes the Lebesgue measure of the Euclidean unit ball and $\mathcal E_V \subset \B_V$ is the inscribed L\"owner-John ellipsoid. This is the unique ellipsoid of largest volume contained in $\B_V$. Below is a list of defining properties for those definitions of volume we need, see for example \cite[\S3]{APT}:
\begin{itemize}
	\item (Gromov-mass$\ast$ or Benson)
	\[
	\mu_V^{\rm m\ast}(\nu) = \sup\{|\langle \xi_1\wedge\cdots\wedge\xi_n, \nu \rangle| : \|\xi_i\|^\ast \leq 1\} .
	\]
	\item (Busemann-Hausdorff)
	\[
	\mu_V^{\rm bh}(\B_V) = \bm\alpha(n) .
	\]
	\item(Inscribed Riemannian)
	\[
	\mu_V^\text{\rm ir}(\mathcal E_V) = \bm\alpha(n) .
	\]
\end{itemize}
The inscribed Riemannian volume, introduced by Ivanov \cite{Iv}, serves as the main volume definition in the present work. Properly normalized, the $n$-dimensional Hausdorff measure $\cH^n$ coincides with $\mu^{\rm bh}$; that is, $\cH^n(B) = \mu_V^{\rm bh}(B)$ for all Borel sets $B$ in an $n$-dimensional normed space $V$, see, for example, \cite[Lemma~6]{K} and the references therein.

\begin{defi}
	\label{def:finslermass}
Any Finsler volume $\mu$ on $n$-dimensional normed spaces gives rise to a Finsler mass $\bM_\mu$ for rectifiable currents $T \in \cR_n(X)$ as follows. Given bi-Lipschitz parametrizations $\varphi_i : K_i \to X$ and densities $\theta_i$ for $T \in \cR_n(X)$ as in Definition~\ref{def:rectifiable}, the $\mu$-mass of $T$ is defined by
\[
\bM_\mu(T) \defl \sum_i \int_{K_i} |\theta_i(x)| \mathbf{J}_{\mu}(\operatorname{md}(\varphi_{i})_x) \, d\cL^n(x) ,
\]
where $\operatorname{md}(\varphi_{i})_x$ is the (approximate) metric derivative of $\varphi_i$ at $x$ as defined in \cite{K}.
\end{defi}

Note that in case $X$ is a Banach space, the maps $\varphi_i : K_i \to X$ can be extended to Lipschitz maps $\bar \varphi_i : \R^n \to X$ due to \cite[Theorem~2]{JLS}. Since $X$ has an isometric embedding into $\ell_\infty(X)$ via the Kuratowski embedding, we can always assume that the maps $\varphi_i$ are defined on all of $\R^n$ and the metric derivatives exist almost everywhere.

We leave it to the reader to show that this definition does not depend on the particular parametrization. With a decomposition argument, it boils down to an application of the area formula \cite[Theorem~3.2.3]{F} and the following chain rule:
\[
\operatorname{md} (\varphi\circ\psi)_{x}(v) = \operatorname{md} \psi_{\varphi(x)}(D\psi_x(v))
\]
for almost every $x \in K_1$ and all $v \in \R^n$, whenever $\psi : K_1 \to K_2$ is bi-Lipschitz, $K_1,K_2 \subset \R^n$ are compact and $\varphi : K_2 \to K$ is Lipschitz. The chain rule follows quite directly from the definition of the metric derivative in \cite{K}. The definition of the Jacobian \eqref{eq:defjacobian0} then implies
\[
\mathbf{J}_{\mu}(\operatorname{md} (\varphi\circ\psi)_{x}) = \mathbf{J}_{\mu}(\operatorname{md} \varphi_{\psi(x)}) |\det D \psi_x|
\]
for almost every $x \in K_1$. Note that the same definition extends to rectifiable sets in metric spaces, and more generally to rectifiable chains with coefficients in a normed abelian group, as developed by De Pauw and Hardt \cite{DPH}.

\begin{lem}
	\label{lem:masslem}
Assume that $\mu$ is a Finsler volume on $n$-dimensional normed spaces, $V$ is an oriented $n$-dimensional normed space, and $X$ is a metric space. The following properties hold for currents in $\cR_n(V)$ and $\cR_n(X)$, respectively:
\begin{enumerate}
	\item If $\theta \in L^1(V)$ has compact support, then
	\[
	\bM_\mu(\curr\theta) = \int_{V} |\theta(x)| \,d\mu_V(x) .
	\]
	\item $\bM = \bM_{\rm m\ast}$.
	\item $C_n^{-1}\bM \leq \bM_\mu \leq C_n \bM$ for some $C > 1$ that depends only on $n$.
	\item $\bM_\text{\rm ir} \geq \bM_\mu$.
	\item $\bM_{\mu}(\psi_\#T) \leq \Lip(\psi)^n\bM_{\mu}(T)$ if $\psi \in \Lip(X,Y)$ and $T \in \cR_n(X)$.
\end{enumerate}
\end{lem}

\begin{proof}
(1): We first assume that $\theta = \chi_B$ for some bounded Borel set $B \subset V$ of positive measure. If a coordinate system on $V$ is fixed via an isomorphism $I : \R^n \to V$, then $\operatorname{md}I_{x}$ is the pull-back norm on $\R^n$ denoted by $s$. Let $B' \defl I^{-1}(B)$ and rewrite the definition of $\bM_\mu(\curr B)$ using \eqref{eq:defjacobian} as
\[
\bM_\mu(\curr B) = \int_{B'}\frac{\mu_s(B')}{\cL^n(B')}\, d\cL^n(x) = \mu_s(B') = \mu_V(B) .
\]
The last equality uses the fact that $I : (\R^n,s) \to V$ is a linear isometry. The result for a general weight function $\theta$ follows by approximation with step functions.
	
(2): Fix a basis $\xi_1,\cdots,\xi_n$ of the dual space $V^\ast$ with $\|\xi_i\|^\ast = 1$ and the property that $|\langle\xi_1\wedge\cdots\wedge\xi_n,\nu\rangle| = \mu_V^{\rm m\ast}(\nu)$ for one (and hence all) $\nu \in \BigWedge_n V \setminus \{0\}$. Let $v_1,\dots,v_n \in V$ be the predual basis and $P \defl P(v_1,\dots,v_n)$ be the parallelepiped spanned by it. The set
\[
\{v \in V : |\xi_i(v)| \leq 1 \text{ for all } i\} = \biggl\{\sum_{i}x_iv_i : |x_i| \leq 1 \text{ for all } i\biggr\}
\]
contains the unit ball $\B_V$ and is a homothetic copy of $P$. By (1) and the definitions,
\[
\bM_{\rm m\ast}(\curr P) = \mu_V^{\rm m\ast}(P) = \mu_V^{\rm m\ast}(v_1\wedge\cdots\wedge v_n) = 1 .
\]
$P$ is parametrized by $[0,1]^n \ni (x_1,\dots,x_n)\mapsto x_1v_1 + \cdots + x_nv_n$. With a standard linearization argument, the Ambrosio-Kirchheim mass of $\curr P$ can be expressed as
\begin{align*}
\bM(\curr P) & = \sup_g \curr{P}(1,g) = \sup_g \int_{[0,1]^n} \langle g_1\wedge \cdots \wedge g_n, v_1 \wedge \cdots \wedge v_n\rangle \, d\cL^n = 1 ,
\end{align*}
where the supremum is taken over all linear $g = (g_1,\dots,g_n) : V \to \R^n$ with $\|g_i\|^\ast \leq 1$ for all $i$. This implies that $\bM_{\rm m\ast}(\curr P) = \bM(\curr P)$ and thus $\bM(\curr \theta) = \int_{V} |\theta(x)| \,d\mu_V^{\rm m\ast}(x)$ for all $\theta$ by approximation with step functions. For the general statement, let $T \in \cR_n(X)$ and for $\lambda > 1$ choose a parametrization $(\varphi_i,K_i,\theta_i)$ of $T$ as in Definition~\ref{def:rectifiable}. By the definition of the metric derivative in \cite{K}, it holds
\[
\lambda^{-1}\|v\|_i \leq \operatorname{md}(\varphi_{i})_{x}(v) \leq \lambda\|v\|_i
\]
for all $i$, $v \in \R^n$ and almost every $x \in K_i$. As a consequence of the first property in Definition~\ref{def:area} and the scaling property of a Haar measure, the estimate above implies $\lambda^{-n}\mu^{\rm m\ast}_{\|\cdot\|_i} \leq \mu^{\rm m\ast}_{\operatorname{md}(\varphi_{i})_{x}} \leq \lambda^n\mu^{\rm m\ast}_{\|\cdot\|_i}$ and with \eqref{eq:defjacobian} also $\lambda^{-n}\mathbf{J}_{\rm m\ast}(\|\cdot\|_i) \leq \mathbf{J}_{\rm m\ast}(\operatorname{md}(\varphi_{i})_{x}) \leq \lambda^n\mathbf{J}_{\rm m\ast}(\|\cdot\|_i)$. Since
\[
\int_{K_i}|\theta_i(x)| \mathbf{J}_{\rm m\ast}(\|\cdot\|_i)\, d\cL^n(x) = \int_{K_i}|\theta_i(x)|\,d\mu_{\|\cdot\|_i}^{\rm m\ast}(x) = \bM(\curr{\theta_i}) ,
\]
we conclude
\begin{align*}
\bM(T) & \leq \lambda^n \sum_i \bM(\curr{\theta_i}) \leq \lambda^{2n}\sum_i \int_{K_i}|\theta_i(x)| \mathbf{J}_{\rm m\ast}(\operatorname{md}(\varphi_{i})_{x})\, d\cL^n(x) \\
 & = \lambda^{2n}\bM_{\rm m\ast}(T).
\end{align*}
The lower bound $\lambda^{-2n}\bM_{\rm m\ast}(T)$ is obtained by a similar argument. Since this holds for all $\lambda > 1$, we conclude (2).
	
(3): By a result of John \cite{J}, we have the inclusions
\[
\mathcal E_V \subset \B_V \subset n^\frac{1}{2}\mathcal E_V .
\]
Let $e$ be the Euclidean norm on $V$ whose unit ball is $\B_e = \mathcal E_V$. These inclusions imply
\[
e(v) \geq \|v\| \geq n^{-\frac{1}{2}}e(v)
\]
for all $v \in V$. The two properties of volumes in Definition~\ref{def:area} justify
\begin{align*}
\bm\alpha(n) & = \mu_e(\mathcal E_V) \leq n^{\frac{n}{2}} \mu_V(\mathcal E_V) \leq n^\frac{n}{2}\mu_e(\mathcal E_V) = n^\frac{n}{2}\bm\alpha(n) .
\end{align*}
Hence $\mu_1 \leq n^\frac{n}{2}\mu_2$ for any two definitions of volume, and by \eqref{eq:defjacobian}, the statement follows for $C_n = n^\frac{n}{2}$.
	
(4): If $e \geq \|\cdot\|$ is the Euclidean norm on $V$ as above, then
\[
\mu_V^{\rm ir}(\mathcal E_V) = \mu_e(\mathcal E_V) \geq \mu_V(\mathcal E_V) .
\]
Consequently, $\mu^{\rm ir} \geq \mu$ and also $\bM_{\rm ir} \geq \bM_\mu$.

(5): This is a consequence of Definition~\ref{def:area}(1). More precisely, with a decomposition of a parametrization into smaller compact sets it boils down to the following chain rule argument. Assume that $K_X \subset V_X$ and $K_Y \subset V_Y$ are compact subsets of $n$-dimensional normed spaces, $\varphi_X : K_X \to X$ and $\varphi_Y : K_Y \to Y$ are bi-Lipschitz embeddings with bi-Lipschitz constants bounded by $\lambda > 1$ and assume that $\psi : \varphi_X(K_X) \to \varphi_Y(K_Y)$ is bi-Lipschitz too. Then $\varphi \defl \varphi_Y^{-1}\circ \psi \circ \varphi_X : K_X \to K_Y$ satisfies $\Lip(\varphi) \leq \lambda^2 \Lip(\psi)$. Assume that $x \in K_X$ is a point of approximate differentiability of $\varphi$, then for any Borel set $B \subset V_X$ of positive and finite measure
\begin{align*}
\frac{\mu_{V_Y}(D\varphi_x(B))}{\mu_{V_X}(B)} \leq \|D\varphi_x\|^n \leq \Lip(\varphi)^n \leq \lambda^{2n}\Lip(\psi)^n .
\end{align*}
Since, for parametrizations as in Definition~\ref{def:rectifiable}, we can choose $\lambda > 1$ arbitrary close to $1$, the result follows from (1). The details for this decomposition argument into bi-Lipschitz pieces $\psi : \varphi_X(K_X) \to \varphi_Y(K_Y)$ is given in \cite[\S3.5]{DPH} in a more general setting and builds on \cite[Lemma~3.1.1]{DPH} applied to $\psi \circ \varphi_i$, where $\varphi_i$ is part of a parametrization for $T$ as in Definition~\ref{def:rectifiable}.
\end{proof}

To estimate the action of differential forms on rectifiable currents, a suitable notion of tangent spaces is needed. Such tangent spaces exist for $L^\infty([0,\pi))$ because it is the dual of the separable Banach space $L^1([0,\pi))$. The result we use here is \cite[Theorem~3.5]{AK2}, which states that If $f : \R^n \to Y$ is Lipschitz, where $Y = X^\ast$ is tue dual of a separable Banach space $X$, then $f$ is weak$\ast$ differentiable at almost every point $x \in \R^n$. More precisely, there exists a linear map $\operatorname{wd}f_x : \R^n \to Y$ such that
\[
\text{w}^\ast\text{-}\lim_{y\to x} \frac{f(y) - f(x) - \operatorname{wd}f_x(y-x)}{|y-x|} = 0
\]
and $\|\operatorname{wd}f_x(v)\| = \operatorname{md}f_x(v)$ for all $v \in \R^n$. By \cite[Theorem~9.1]{AK}, any $T \in \cR_n(Y)$ admits a representation $\curr{S,\theta,\tau}$, where $S \subset Y$ is a countably $\cH^n$-rectifiable Borel set, $\theta : S \to (0,\infty)$ is a Borel function with $\int_S\theta\,d\cH^n < \infty$ and $\tau : S \to \BigWedge_n Y$ is an orientation such that
\begin{equation}
	\label{eq:currentrepres}
T(f,g_1,\dots,g_n) = \int_S \theta(x)f(x)\left\langle \BigWedge_n d^S_x g, \tau(x) \right\rangle\,d\cH^n(x) .
\end{equation}
Here are some details. By the weak$\ast$ differentiability of Lipschitz maps, the set $S$ has an $n$-dimensional approximate tangent space $\operatorname{Tan}^{(n)}(S,x)$ at $\cH^n$-almost every $x \in S$. An orientation is a simple $n$-vector filed $\tau = \tau_1 \wedge \cdots \wedge \tau_n$ on $S$, where $\tau_1,\dots,\tau_n : S \to Y$ are Borel maps such that for $\cH^n$-almost every $x \in S$ and all $j$:
\begin{enumerate}
	\item $\tau_j(x) \in \operatorname{Tan}^{(n)}(S,x)$.
	\item $|\tau_j(x)| \leq C_n$ for some $C_n \geq 1$ depending only on $n$.
	\item $\mu^{\rm bh}(\tau) = 1$.
\end{enumerate}
These measurable vector fields can be constructed as follows. Assume that $\varphi_i : K_i \to \varphi_i(K_i)$ are bi-Lipschitz parametrizations of $S$ as in Definition~\ref{def:rectifiable}, with bi-Lipschitz constant $\lambda \leq 2$. Let $L_{i,x} : \R^n \to \operatorname{Tan}^{(n)}(S,\varphi_i(x))$ be the weak$\ast$ derivative of $\varphi_i$ at $x \in K_i$, if it exists. Let $e_1,\dots,e_n$ of $\R^n$ be an oriented orthonormal basis with respect to the inscribed Riemannian inner product associated with $(\R^n,\|\cdot\|_i)$. Then, appropriate measurable vector fields $\tau_1,\dots,\tau_n$ on $\varphi_i(K_i)$ can be defined by
\[
\tau_j(\varphi_i(x)) \defl \mu^{\rm bh}(L_{i,x}(e_1) \wedge \cdots \wedge L_{i,x}(e_n))^{-n} L_{i,x}(e_j).
\]

The linear maps $d_x^Sg : \operatorname{Tan}^{(n)}(S,x) \to \R^n$ are characterized by the property that $\operatorname{wd}(g \circ f)_y = d^S_{f(y)}g \circ \operatorname{wd}f_y$ holds for almost every $y \in f^{-1}(S)$, whenever $f : \R^n \to Y$ is a Lipschitz map.

As a consequence of \cite[Theorem~9.5]{AK}, we have
\[
\bM(T) = \int_S \theta(x) \lambda(x) \, d\cH^n(x) ,
\]
where $\lambda(x) = \lambda_{\operatorname{Tan}^{(n)}(S,x)}$, and $\lambda_V$ for any $n$-dimensional subspace $V \subset Y$ is defined by
\[
\lambda_V \defl \sup_P\frac{2^n}{\cH^n(P)} ,
\]
with the supremum taken over all parallelepipeds $P$ that contain the unit ball of $V$. Equivalently,
\[
\lambda_V = \frac{\mu_V^{\rm m\ast}(B)}{\mu_V^{\rm bh}(B)}
\]
for any Borel set $B \subset V$ of finite and positive measure. $\lambda_V$ is the factor on normed spaces used to transform from the Haar measure $\mu_V^{\rm bh}$ to the Haar measure $\mu_V^{\rm m\ast}$. The reason for basing this on the Busemann-Hausdorff definition of Finsler volume is that it is induced by the $n$-dimensional Hausdorff measure of the ambient space. A corresponding density can be computed for any Finsler volume, and together with Lemma~\ref{lem:masslem}, this yields the following characterization of the Finsler mass.

\begin{lem}
	\label{lem:masscor}
Assume $Y = X^\ast$ for a separable Banach space $X$, and let $T \in \cR_n(Y)$ be represented by $\curr{S,\theta,\tau}$. Then for every Finsler volume $\mu$, the $\mu$-mass of $T$ is given by
\begin{equation*}
\bM_\mu(T) = \int_S \theta(x) \lambda^\mu_{\operatorname{Tan}^{(n)}(S,x)} \, d\cH^n(x),
\end{equation*}
where $\lambda_V^\mu$ is defined by
\[
\lambda^\mu_V \defl \frac{\mu_V(B)}{\mu_V^{\rm bh}(B)}
\]
for any Borel set $B \subset V$ of finite and positive measure.
\end{lem}

\subsection{Injective hull of spheres}
\label{subsec:inj_hull}

Although we will only work with the injective hull of the Riemannian circle $\bS^1$, the main result of this subsection remains valid for the Riemannian sphere $\bS^n$ of arbitrary dimension $n$. This is the standard Euclidean unit sphere  $\bS^n = \{x \in \R^{n+1} : |x| = 1\}$ endowed with the intrinsic geodesic distance $d$. As a subset of the Banach space $\ell_\infty(\bS^n)$ of bounded functions $\bS^n \to \R$, the injective hull $E(\bS^n)$ can be identified with the set of $1$-Lipschitz functions $f : \bS^n \to \R$, denoted $\Lip_1(\bS^n)$, satisfying the following conditions:
\begin{enumerate}
	\item $d(x,y) \leq f(x) + f(y)$ for all $x,y \in \bS^n$.
	\item For all $x \in \bS^n$ there exists $y \in \bS^n$ with $f(x) + f(y) = d(x,y)$.
\end{enumerate}
Injective hulls were introduced independently by Isbell \cite{I} and Dress \cite{D}. For further details and consequences of the definition, see for instance \cite[\S3]{L2}. It is shown there that $E(\bS^n)$, as defined above, is indeed an injective metric space. That is, for every $1$-Lipschitz map $\varphi : A \to E(\bS^n)$ defined on a subset $A$ of a metric space $X$, there exists a $1$-Lipschitz extension $\bar \varphi: X \to E(\bS^n)$. The map $\iota : \bS^n \to E(\bS^n)$, defined by $\iota(x) \defl d_x$ with $d_x(y) \defl d(x,y)$ for all $y \in \bS^n$, is an isometric embedding. This is known as the Kuratowski embedding. We identify $\bS^n$ with the image of $\iota$. As in the proof of Lemma~\ref{lem:isometricembedding} below, $\iota(\bS^n)$ is the only isometric copy of $\bS^n$ in $E(\bS^n)$. With this identification, properties (1) and (2) directly imply
\begin{equation}
\label{eq:disttoboundary}
\|d_x - f\|_\infty = f(x) ,
\end{equation}
for every $x \in \bS^n$. This means that the distance of $f$ to a point in $\bS^n$ is given by the evaluations of $f$ at this point.

The key observation for the characterization of $E(\bS^n)$ below is that every point $x \in \bS^n$ lies on a geodesic connecting any point $y \in \bS^n$ to its antipodal point $-y$.

\begin{prop}
	\label{prop:injectivehull}
The following properties hold:
\begin{enumerate}
	\item For $f \in \Lip_1(\bS^n)$, one has $f \in E(\bS^n)$ if and only if $f(x) + f(-x) = \pi$ for all $x \in \bS^n$.
	\item $E(\bS^n)$ is a compact and convex subset of $\ell_\infty(\bS^n)$.
	\item Functions in $E(\bS^n)$ take values in $[0,\pi]$.
	\item $f \in E(\bS^n)\setminus \bS^n$ if and only if $f \in E(\bS^n)$ and $f$ takes values in $(0,\pi)$.
\end{enumerate}
\end{prop}

\begin{proof}
(1): Let $f \in E(\bS^n)$, by definition, for any $x \in \bS^n$, there exists $y \in X$ with $f(x) + f(y) = d(x,y)$. Using this, we obtain
\begin{align*}
d(x,-x) & \leq f(x) + f(-x) = f(x) + f(y) + f(-x) - f(y) \\
& \leq d(x,y) + d(-x,y) = d(x,-x) .
\end{align*}
Hence, equality holds throughout, and we conclude that $f(x) + f(-x) = d(x,-x) = \pi$.

On the other hand, suppose that $f \in \Lip_1(\bS^n)$ satisfies $f(x) + f(-x) = \pi$ for some $x \in \bS^n$, and assume for contradiction that there exists $y \in \bS^n$ such that $d(x,y) > f(x) + f(y)$. Since $f$ is $1$-Lipschitz, we have
\begin{align*}
d(x,-x) - d(-x,y) & = d(x,y) > f(x) + f(y) \\
& = f(x) + f(-x) + f(y) - f(-x) \\
& \geq d(x,-x) - d(y,-x) .
\end{align*}
This is not possible. Hence $f(x) + f(y) \geq d(x,y)$ for all $y \in \bS^n$. If $f(x) + f(-x) = \pi$ holds for all $x \in \bS^n$, then $f(x) + f(y) \geq d(x,y)$ for all $x,y \in \bS^n$ with equality for $y = -x$. This shows that $f \in E(\bS^n)$ and establishes (1).
	
(2): Since all the functions in $E(\bS^n)$ are $1$-Lipschitz and $\bS^n$ is compact, the Arzel\`a-Ascoli theorem implies that $E(\bS^n)$ is compact. If $f,g \in E(\bS^n)$ and $t \in [0,1]$, then $tf + (1-t)g$ is $1$-Lipschitz and moreover
\[
(tf(x) + (1-t)g(x)) + (tf(-x) + (1-t)g(-x)) = t\pi + (1-t)\pi = \pi ,
\]
for all $x \in \bS^n$. Hence $tf + (1-t)g \in E(\bS^n)$ by (1). This proves (2).
	
(3): Observe that $f(x) + f(-x) = \pi$ by (1) and since $0 \leq \frac{1}{2}d(x,x) \leq f(x)$ by the definition of $E(\bS^n)$, it follows that $f(x) \in [0,\pi]$ for all $x \in \bS^n$.
	
(4): If $f \in \bS^n \subset E(\bS^n)$, then $f = d_x$ for some $x$ and hence $f(x) = d(x,x) = 0$. On the other hand, if $f(x) = 0$ for $f \in E(\bS^n)$ and $x \in \bS^n$, then $\|f - d_x\|_{\infty} = f(x) = 0$ by \eqref{eq:disttoboundary}. Hence $f = d_x \in \bS^n$. Similarly, if $f(x) = \pi$, then $f(-x) = 0$ by (1) and therefore $f = d_{-x}$ by the same argument.
\end{proof}

For our main applications, we fix an orientation of $\bS^1$ and a base point $p_0 \in \bS^1$. Let $\gamma : \R \to \bS^1$ be the $2\pi$-periodic covering map with $\gamma(0) = p_0$ that preserves both length and orientation. Any function $f \in E(\bS^1)$ then admits a unique lift $\bar f : \R \to \R$ such that $f \circ \gamma = \bar f$. Working with these lifts, it follows from Proposition~\ref{prop:injectivehull} that $E(\bS^1)$ can be identified isometrically with the space of functions $f : \R \to \R$ satisfying:
\begin{enumerate}
	\item $f$ is $1$-Lipschitz.
	\item $f_{\alpha + \pi} + f_\alpha = \pi$ for all $\alpha \in \R$.
\end{enumerate}
As a consequence of (2), any such function is $2\pi$-periodic. We henceforth fix the identification of $E(\bS^1)$ with its lifted representatives on $\R$ via $\gamma$. Although the differential form $\tilde \omega$ introduced earlier may a priori depend on the choice of a point (and certainly on the orientation), we will see in \eqref{eq:indepbasepoint} that it is in fact independent of the base point. In this notation, points in $\bS^1$ are identified with functions of the form $\alpha \mapsto \arccos(\cos(\alpha - \tau))$ for some parameter $\tau \in \R$, these are piecewise linear "zigzag" functions.

Since the coefficients of $\tilde \omega$ are not bounded in a neighborhood of $\bS^1$, we will also make use of the truncated injective hulls
\begin{equation}
\label{eq:truncated}
E_\varepsilon(\bS^1) \defl \{f \in E(\bS^1) : f_\alpha \in [\varepsilon, \pi-\varepsilon] \text{ for all } \alpha\}
\end{equation}
for $\varepsilon \in (0,\frac{\pi}{2})$. The following observations are easy to check and left to the reader.

\begin{lem}
	\label{lem:truncated}
$ $
\begin{enumerate}
	\item $E_\varepsilon(\bS^1) = \{f \in E(\bS^1) : \dist(f,\bS^1) \geq \varepsilon\}$.
	\item $E(\bS^1)\setminus \bS^1 = \bigcup_{n \in \N} E_{\frac{1}{n}}(\bS^1)$.
	\item $E_\varepsilon(\bS^1)$ is a compact and convex subset of $\ell^\infty(\R)$.
	\item $E_\varepsilon(\bS^1)$ is a $1$-Lipschitz retract of $E(\bS^1)$.
\end{enumerate}
\end{lem}

\subsection{Representation of the hemisphere}
\label{sec:hemisphere}

The hemisphere
\[
\bS^2_+ \defl \{(x,y,z) \in \R^3 : x^2 + y^2 + z^2 = 1, \, z \geq 0 \}
\]
is equipped with the induced intrinsic metric denoted by $d$. For $p \in \bS^2_+$ let $x_p \in \bS^1$ be a point with intrinsic distance $d(x_p,p) = \dist(\bS^1,p)$. This point is unique unless $p$ is the north pole $N \defl (0,0,1)$. For any $x \in \bS^1$, the spherical Pythagorean theorem states
\[
\cos(d(p,x)) = \cos(d(p,x_p))\cos(d(x_p,x)) .
\]
Thus, $p$ can be identified with the function $f_p : \bS^1 \to \R$ defined by
\[
f_p(x) \defl d(p,x) = \arccos(\cos(d(p,x_p))\cos(d(x_p,x))) .
\]
If $p = N$, we have
\[
f_N(x) = \tfrac{\pi}{2} = \arccos(0) = \arccos(\cos(\tfrac{\pi}{2})\cos(d(x_p,x)))
\]
for all $x \in \bS^1$.

\begin{lem}
	\label{lem:isometricembedding}
The map $\iota : p \mapsto f_p$ is an isometric embedding of $\bS^2_+$ into $E(\bS^1)$. Moreover, $\iota(\bS^2_+)$ is the only isometric copy of $\bS^2_+$ in $E(\bS^1)$.
\end{lem}

\begin{proof}
As a distance function, it is clear that $f_p$ is $1$-Lipschitz, since for all $x,y \in \bS^1$ we have
\[
|f_p(x) - f_p(y)| = |d(p,x) - d(p,y)| \leq d(x,y).
\]
Moreover, for all $x \in \bS^1$,
\[
f_p(x) + f_p(-x) = d(p,x) + d(p,-x) = \pi,
\]
because $p \in \bS^2_+$ lies on a minimizing geodesic connecting $x$ with $-x$. Thus $f_p$ is in $E(\bS^1)$ by Proposition~\ref{prop:injectivehull}.
	
It remains to show that the intrinsic distance $d(p,q)$ is given by
\begin{equation}
\label{eq:halfsphereinjective}
d(p,q) = \|f_p - f_q\|_\infty
\end{equation}
for all distinct points $p,q \in \bS^2_+\setminus \bS^1$. On the one hand, for every $x \in \bS^1$, we have
\[
d(p,q) \geq |d(p,x) - d(x,q)| = |f_p(x) - f_q(x)|.
\]
On the other hand, the unique geodesic from $p$ to $q$ in $\bS^2_+$ can be extended in $\bS^2_+$ until it meets the boundary $\bS^1$ at some point $x$. Since this extended geodesic is minimizing, it follows that
\[
f_p(x) - f_q(x) = d(p,x) - d(q,x) = d(p,q) .
\]
This establishes \eqref{eq:halfsphereinjective}.

For the second statement, assume that $X$ is an isometric copy of $\bS^2_+$ inside $E(\bS^1)$. Its (surface) boundary $\partial X$ is isometric to $\bS^1$, and we claim that $\partial X$ coincides with the natural isometric copy $S \defl \{d_x : x \in \bS^1\}$ of $\bS^1$ in $E(\bS^1)$. Indeed, any point $f \in \partial X$ has a corresponding point $g \in \partial X$ with $\|f - g\|_\infty = \pi$. However, as a consequence of Proposition~\ref{prop:injectivehull}, the only pairs of points in $E(\bS^1)$ that have distance $\pi$ are antipodal pairs $d_x, d_{-x}$ in $S$. Therefore $\partial X$ is contained in $S$. Since both $\partial X$ and $S$ are topological circles, it follows that $S = \partial X$. Any $f \in E(\bS^1)$ is uniquely determined by the distance functions $f(x) = \|f - d_x\|_\infty$ for $x \in \bS^1$, as follows from \eqref{eq:disttoboundary}. Similarly, any $f \in X$, being a point in an isometric copy of $\bS^2_+$, is uniquely determined by the distance functions $\|f-g\|_\infty$ to points $g$ of the boundary $\partial X$. From $S = \partial X$, it follows that $X \subset \iota(\bS^2_+)$. Because $X$ cannot be isometric to a proper subset of $\iota(\bS^2_+)$, we conclude that $X = \iota(\bS^2_+)$.
\end{proof}

This lemma establishes the existence of a unique subset of $E(\bS^1)$ that is isometric to $\bS^2_+$. Moreover, the proof shows that its boundary $\{d_x : x \in \bS^1\}$ is the only isometric copy of $\bS^1$ in $E(\bS^1)$. Consequently, both metric spaces $\bS^1$ and $\bS^2_+$ will be identified with these corresponding subsets of $E(\bS^1)$.

Any point $p \in \bS^2_+$ is represented by its lift $f : \R \to \R$,
\[
f_\alpha = \arccos(\cos(d)\cos(\alpha - \tau))
\]
for some parameters $\tau \in \R$ and $d \in [0,\frac{\pi}{2}]$. In fact, $(\frac{\pi}{2}-d,\tau) \mapsto f$ are polar normal coordinates centered at the north pole $N = \frac{\pi}{2}$.

For later use, we analyze the variations in $\tau$ and $d$. Consider the function
\[
\Gamma_\alpha(\tau,d) \defl \arccos(\cos(d)\cos(\alpha - \tau))
\]
for $\alpha,\tau \in \R$ and $d \in [0,\frac{\pi}{2}]$. Since $\arccos'(x) = \frac{-1}{\sqrt{1 - x^2}}$ for $x \in (-1,1)$, we obtain
\begin{align*}
\tfrac{\partial}{\partial \tau} \Gamma_\alpha(\tau,d) & = \frac{-\cos(d)\sin(\alpha - \tau)}{(1 - \cos(d)^2\cos(\alpha - \tau)^2)^\frac{1}{2}} , \\
\tfrac{\partial}{\partial d} \Gamma_\alpha(\tau,d) & = \frac{\sin(d)\cos(\alpha - \tau)}{(1 - \cos(d)^2\cos(\alpha - \tau)^2)^\frac{1}{2}} ,
\end{align*}
for $\alpha,\tau \in \R$ and $d \in (0,\frac{\pi}{2}]$. Thus
\begin{align*}
\left(\tfrac{\partial}{\partial \tau} \Gamma_\alpha(\tau,d)\right)^2 & + \cos(d)^2\left(\tfrac{\partial}{\partial d} \Gamma_\alpha(\tau,d)\right)^2 \\
& = \frac{\cos(d)^2\sin(\alpha - \tau)^2}{1 - \cos(d)^2\cos(\alpha - \tau)^2} + \frac{\cos(d)^2\sin(d)^2\cos(\alpha - \tau)^2}{1 - \cos(d)^2\cos(\alpha - \tau)^2} \\
& = \cos(d)^2\frac{(1 - \cos(\alpha - \tau)^2) + (1-\cos(d)^2)\cos(\alpha - \tau)^2}{1 - \cos(d)^2\cos(\alpha - \tau)^2} \\
& =  \cos(d)^2 .
\end{align*}
Hence, $\alpha \mapsto \nabla\Gamma_\alpha(d,\tau)$ traces an ellipse. Respectively, the path
\begin{align*}
	\alpha \mapsto \gamma_{\tau,d}(\alpha) & \defl \left(\frac{-\sin(\alpha - \tau)}{(1 - \cos(d)^2\cos(\alpha - \tau)^2)^\frac{1}{2}}, \frac{\sin(d)\cos(\alpha - \tau)}{(1 - \cos(d)^2\cos(\alpha - \tau)^2)^\frac{1}{2}}\right) \\
	& = \left(\frac{\tfrac{\partial}{\partial \tau} \Gamma_\alpha(\tau,d)}{\cos(d)}, \tfrac{\partial}{\partial d} \Gamma_\alpha(\tau,d)\right) \\
\end{align*}
lies on the unit circle. Relevant properties of $\gamma_{\tau,d}$ are collected in the next lemma.

\begin{lem}
	\label{lem:tangentsphere}
For fixed $d \in (0,\frac{\pi}{2}]$ and $\tau \in \R$, the tangent plane of $\bS^2_+ \subset L^\infty(\R)$ at $f = \arccos(\cos(d)\cos(\cdot - \tau))$ is spanned by the coordinate functions of the plane path $\gamma = \gamma_{\tau,d} : \R \to \R^2$, which gives a counterclockwise parametrization of the unit circle. Further,
\[
\gamma(\alpha) \times \gamma(\beta) = \frac{\sin(d)\sin(\beta-\alpha)}{(1 - \cos(d)^2\cos(\alpha - \tau)^2)^\frac{1}{2}(1 - \cos(d)^2\cos(\beta - \tau)^2)^\frac{1}{2}} ,
\]
where $v \times w \defl v_1w_2 - v_2w_1$ and
\[
|\gamma'(\alpha)| = \frac{\sin(d)}{1 - \cos(d)^2\cos(\alpha - \tau)^2} = \frac{\sin(\dist(f,\bS^1))}{\sin(f_\alpha)^2}
\]
with integral 
\begin{align*}
	2\pi & = \int_{0}^{2\pi} \frac{\sin(d)}{1 - \cos(d)^2\cos(\alpha - \tau)^2} \,d\alpha .
\end{align*}
\end{lem}

\begin{proof}
For $\alpha,\beta \in \R$,
\begin{align*}
\gamma(\alpha) \times \gamma(\beta) & = \frac{\sin(d)\cos(\alpha - \tau)\sin(\beta - \tau) - \sin(d)\cos(\beta - \tau)\sin(\alpha - \tau)}{(1 - \cos(d)^2\cos(\alpha - \tau)^2)^\frac{1}{2}(1 - \cos(d)^2\cos(\beta - \tau)^2)^\frac{1}{2}} \\
& = \frac{\sin(d)\sin(\beta-\alpha)}{(1 - \cos(d)^2\cos(\alpha - \tau)^2)^\frac{1}{2}(1 - \cos(d)^2\cos(\beta - \tau)^2)^\frac{1}{2}} .
\end{align*}
Here $v \times w = v_1w_2 - v_2w_1$ is the signed area spanned by the parallelogram of two vectors $v,w \in \R^2$.

Thus, $\gamma(\alpha) \times \gamma(\beta) > 0$ whenever $\alpha < \beta < \alpha + \pi$. Since $|\gamma(\alpha)| = 1$ for all $\alpha$, it follows that $\gamma$ is a smooth, counterclockwise parametrization of $\bS^1$. Its speed is given by
\begin{align*}
|\gamma'(\alpha)| & = \gamma(\alpha) \times \gamma'(\alpha) = \lim_{\varepsilon \to 0} \frac{\gamma(\alpha) \times \gamma(\alpha + \varepsilon)}{\varepsilon} = \frac{\sin(d)}{1 - \cos(d)^2\cos(\alpha - \tau)^2} \\
 & = \frac{\sin(\dist(f,\bS^1))}{\sin(f_\alpha)^2} ,
\end{align*}
and the length of $\gamma|_{[0,2\pi]}$ is given by
\begin{align*}
2\pi & = \int_{0}^{2\pi} |\gamma'(\alpha)| \,d\alpha = \int_0^{2\pi} \frac{\sin(d)}{1 - \cos(d)^2\cos(\alpha - \tau)^2} \, d\alpha
\end{align*}
as claimed.
\end{proof}

The integral identity above is the primary motivation for the definition of the differential form $\tilde \omega$.

\section{Definition of omega}
\label{sec:definition}

For $f \in E(\bS^1) \setminus \bS^1$ and $\alpha,\beta \in \R$ with $\alpha \neq \beta \text{ \rm mod } \pi$ (i.e., $\sin(\beta-\alpha) \neq 0$), coefficients are defined by
\begin{equation*}
p_{\alpha,\beta}(f) \defl \frac{1 - \cos(\beta-\alpha)^2 - \cos(f_\alpha)^2 - \cos(f_\beta)^2 + 2 \cos(\beta - \alpha)\cos(f_\alpha)\cos(f_\beta)}{\sin(\beta-\alpha)^2\sin(f_\alpha)^2\sin(f_\beta)^2} .
\end{equation*}
First, note that since $f \notin \bS^1$, it follows from Proposition~\ref{prop:injectivehull} that $f$ takes values in $(0,\pi)$. In particular, $p_{\alpha,\beta}(f)$ is well-defined. The differential two-form $\omega \in \Omega^2(E(\bS^1) \setminus \bS^1)$ is defined by
\begin{equation}
\label{eq:diffform}
\omega_f \defl \int_0^\pi \int_\alpha^\pi p_{\alpha,\beta}(f)\, d\pi_\alpha \wedge d\pi_\beta \, d\beta \, d\alpha .
\end{equation}
The precise interpretation of $\omega_f$ will be given in Subsection~\ref{sec:actioncurrents}, where the action of $\omega$ on currents is introduced. Note that $\omega$ differs from $\tilde \omega$ defined in the introduction by a factor of $\pi$. For convenience of notation, we will work with $\omega$ until the proof of Theorem~\ref{thm:mainthm1}.

\subsection{Geometric interpretation}
\label{subsec:geometric}

The coefficients $p_{\alpha,\beta}(f)$ have geometric meaning. Fix $f \in E(\bS^1)$ and $\alpha,\beta \in \R$ that represent points in $\bS^1$ (also denoted by $\alpha$ and $\beta$) such that $\alpha \neq \beta \text{ \rm mod } \pi$ ($d(\alpha,\beta)$ is neither $0$ nor $\pi$). The three values $f_\alpha$, $f_\beta$ and $d(\alpha,\beta)$ lie in the interval $[0,\pi]$ and satisfy the triangle inequality by the defining properties of $E(\bS^1)$. Thus, there exists a unique point $p \in \bS^2_+ \subset \R^3$ with spherical distances $d(p,\alpha) = f_\alpha$ and $d(p,\beta) = f_\beta$. Let $A$, $B$, and $C$ denote the angles of the spherical triangle with vertices $\alpha$, $\beta$, and $p$, respectively. Denote by $h_{\alpha,\beta}(f) \geq 0$ the height of $p$ above the horizontal plane $\R^2 \times {0} \subset \R^3$. Then the spherical law of sines yields
\[
\frac{\sin(C)^2}{\sin(\beta-\alpha)^2} = \frac{\sin(A)}{\sin(f_\beta)} \frac{\sin(B)}{\sin(f_\alpha)} ,
\]
as well as $\sin(A) = \sin(d)/\sin(f_\alpha)$ and $\sin(B) = \sin(d)/\sin(f_\beta)$, where $d \in [0,\frac{\pi}{2}]$ is the intrinsic distance in $\bS^2_+$ from $p$ to $\bS^1$. Note that $\sin(d) = h_{\alpha,\beta}(f)$, and therefore
\begin{equation}
\label{eq:calibheight1}
\frac{\sin(C)^2}{\sin(\beta-\alpha)^2} = \frac{h_{\alpha,\beta}(f)^2}{\sin(f_\alpha)^2\sin(f_\beta)^2}.
\end{equation}
Next, we derive a formula for the height function $h_{\alpha,\beta}(f)$. The point $p \in \bS^2_+$ is represented by the function $g \in E(\bS^1)$ with $g_\alpha = f_\alpha$, $g_\beta = f_\beta$, and satisfies $\cos(g_x) = \cos(d)\cos(x - \tau)$ for some $\tau \in \R$. Let $\tilde x \defl x - \tau$. By definition,
\begin{align*}
p_{\alpha,\beta}(f) & \sin(f_\alpha)^2 \sin(f_\beta)^2 \\
 & = 1 - \frac{\cos(g_\alpha)^2 + \cos(g_\beta)^2 - 2 \cos(\beta - \alpha) \cos(g_\alpha)\cos(g_\beta)}{\sin(\beta-\alpha)^2} \\
 & = 1 - \cos(d)^2 \frac{\cos(\tilde\alpha)^2 + \cos(\tilde\beta)^2 - 2 \cos(\tilde \beta - \tilde \alpha) \cos(\tilde \alpha)\cos(\tilde \beta)}{\sin(\tilde \beta - \tilde \alpha)^2} \\
 & = 1 - \cos(d)^2 = \sin(d)^2 = h_{\alpha,\beta}(f)^2.
\end{align*}
Together with \eqref{eq:calibheight1}, this provides a geometric interpretation of the coefficients $p_{\alpha,\beta}(f)$ in terms of spherical geometry.

\begin{lem}
	\label{lem:geometriccoeff}
For all $f \in E(\bS^1)$ and $\alpha,\beta \in \bS^1$ with $d(\alpha,\beta) \notin \{0,\pi\}$, we have
\[
p_{\alpha,\beta}(f) = \frac{\sin(\angle_p(\alpha,\beta))^2}{\sin(\beta-\alpha)^2} = \frac{h_{\alpha,\beta}(f)^2}{\sin(f_\alpha)^2\sin(f_\beta)^2} ,
\]
where $p \in \bS^2_+ \subset \R^3$ is the unique point satisfying $d(p,\alpha) = f_\alpha$ and $d(p,\beta) = f_\beta$. Here, $h_{\alpha,\beta}(f)$ is the height of the point $p$ above $\R^2 \times \{0\}$, and $\angle_p(\alpha,\beta)$ is the angle at $p$ of the spherical triangle induced by $\alpha$, $\beta$ and $p$.
\end{lem}

Having established this geometric perspective, we can now deduce quantitative information about the coefficients.

\begin{lem}
	\label{lem:pfunction}
For $f \in E(\bS^1) \setminus \bS^1$ and $\alpha,\beta \in \bS^1$ with $d(\alpha,\beta) \notin \{0,\pi\}$, we have:
\begin{enumerate}
	\item $p_{\alpha,\beta}(f) = p_{\beta,\alpha}(f)$.
	\item $p_{\alpha,\beta}(f)$ is $\pi$-periodic in $\alpha$ and $\beta$.
	\item $p_{\alpha,\beta}(f) \geq 0$ with equality if and only if one the values $f_\alpha$, $f_\beta$ and $d(\alpha,\beta)$ is the sum of the other two.
	\item $\sup_{\alpha \neq \beta \text{ \rm mod } \pi} p_{\alpha,\beta}(f) \leq \sin(\dist(f,\bS^1))^{-2}$.
\end{enumerate}
\end{lem}

\begin{proof}
(1) is immediate from the definition, and (2) follows from the antipodal symmetry inherent in $E(\bS^1)$. Specifically, the relation $f_\alpha + f_{\alpha + \pi} = \pi$ implies
\[
\cos(f_{\alpha + \pi}) = \cos(\pi - f_\alpha) = -\cos(-f_\alpha) = -\cos(f_\alpha) ,
\]
while it is also clear that $\cos(\beta - \alpha - \pi) = -\cos(\beta - \alpha)$. Therefore, the sign changes that appear when transforming $p_{\alpha,\beta}(f)$ into $p_{\alpha + \pi,\beta}(f)$ cancel out. Statements (3) and (4) follow directly from the geometric interpretation of $p_{\alpha,\beta}(f)$ given in Lemma~\ref{lem:geometriccoeff}.
\end{proof}

For technical reasons, we consider the subset $E^+(\bS^1)$ consisting of functions $f \in E(\bS^1)$ for which $p_{\alpha,\beta}(f) > 0$ for all $\alpha \neq \beta \text{ \rm mod } \pi$.

\begin{lem}
	\label{lem:epluslem}
If $f \in E(\bS^1)$ satisfies $\Lip(f) < 1$, then $f \in E^+(\bS^1)$. In particular, $E^+(\bS^1)$ is dense in $E(\bS^1)$.
\end{lem}

\begin{proof}
Assume that $f \notin E^+(\bS^1)$. By Lemma~\ref{lem:pfunction} there exist $\alpha,\beta \in \R$ with $\alpha \neq \beta \text{ \rm mod } \pi$ such that the triple $f_\alpha$, $f_\beta$, $d(\alpha,\beta)$ forms a degenerate triangle. Since $f \in E(\bS^1)$, the triangle inequality is satisfied: $d(\alpha,\beta) \leq f_\alpha + f_\beta$ and $|f_\alpha - f_\beta| \leq d(\alpha,\beta)$. Hence, the only possibilities for degeneracy are:
\begin{itemize}
	\item $f_\alpha = f_\beta + d(\alpha,\beta)$,
	\item $f_\beta = f_\alpha + d(\alpha,\beta)$, or
	\item $d(\alpha,\beta) = f_\alpha + f_\beta$.
\end{itemize}
In the first two cases, we have $|f_\alpha - f_\beta| = d(\alpha,\beta)$, which implies $\Lip(f) = 1$. In the third case, let $\alpha' \defl \alpha + \pi$. Then
\[
d(\alpha',\beta) = \pi - d(\alpha,\beta) = \pi - f_\alpha - f_\beta = f_{\alpha'} - f_\beta.
\]
Hence, again $\Lip(f) = 1$. This proves the first statement.

The second statement is immediate: For any $\lambda \in (0,1)$ and $f \in E(\bS^1)$, define $f_\lambda \defl (1-\lambda)\frac{\pi}{2} + \lambda f$. Since $E(\bS^1)$ is convex, it follows that $f_\lambda \in E(\bS^1)$. Moreover, $\Lip(f_\lambda) \leq \lambda$, and
\[
\|f - f_\lambda\|_\infty = (1-\lambda)\|f - \tfrac{\pi}{2}\|_\infty \leq (1-\lambda)\tfrac{\pi}{2} .
\]
\end{proof}

\subsection{Coefficient estimates}

For two functions $f^0,f^1 \in E_\varepsilon(\bS^1)$, the convex combination $f^t \defl (1-t)f^0 + tf^1$, with $t \in [0,1]$, also lies in $E_\varepsilon(\bS^1)$ by Lemma~\ref{lem:truncated}. Since we wish to interchange integration over $p_{\alpha,\beta}(f^t)$ with differentiation in $t$, we are interested in uniform bounds for $p_{\alpha,\beta}(f^t)$ and its derivatives. Lemma~\ref{lem:pfunction}(4) shows that $p_{\alpha,\beta}(f^t)$ is bounded by a constant depending only on $\varepsilon$. This is the main reason for working with the truncated space $E_\varepsilon(\bS^1)$ instead of the full injective hull. Maybe this restriction, also in the main theorem, can be avoided with a more careful study.

The function
\begin{equation}
\label{eq:pdefinition}
p(a,x,y) \defl \frac{1 - \cos(a)^2 - \cos(x)^2 - \cos(y)^2 + 2 \cos(a)\cos(x)\cos(y)}{\sin(a)^2\sin(x)^2\sin(y)^2}
\end{equation}
is defined for $a,x,y \in \R \setminus \pi\Z$. It is clear that $p$ is symmetric and smooth. The partial derivatives are stated below.

\begin{lem}
	\label{lem:derivative}
The first and second derivatives of $p$ in $(a,x,y)$ are given by
\begin{align*}
p_x & = 2\frac{(\cos(a)\cos(x) - \cos(y))(\cos(a) - \cos(x)\cos(y))}{\sin(a)^2\sin(x)^3\sin(y)^2} , \\
p_{xx} & = 2\frac{\cos(a)\cos(x)\cos(y)(5 + \cos(x)^2) - (1 + 2\cos(x)^2)(\cos(a)^2 + \cos(y)^2)}{\sin(a)^2\sin(x)^4\sin(y)^2} , \\
p_{xy} & = 2\frac{\cos(a)(1 + \cos(x)^2)(1 + \cos(y)^2) - 2\cos(x)\cos(y)(1 + \cos(a)^2)}{\sin(a)^2\sin(x)^3\sin(y)^3} .
\end{align*}
\end{lem}

The proof is left to the reader. We now establish uniform bounds for the first and second derivatives of the variation $t \mapsto p_{\alpha,\beta}((1-t)f^0 + t f^1)$.

\begin{lem}
	\label{lem:unifboundder}
There is a constant $C > 0$ with the following property: If $\varepsilon \in (0,\frac{\pi}{2})$ and $f^t \defl (1-t)f^0 + t f^1$ for $f^0, f^1 \in E_\varepsilon(\bS^1)$ and $t \in [0,1]$, then
\begin{align*}
\sup_{t \in [0,1], \alpha \neq \beta \text{ \rm mod } \pi} \left|\tfrac{d}{dt}p_{\alpha,\beta}(f^t)\right| & \leq C\sin(\varepsilon)^{-6} , \\
\sup_{t \in [0,1], \alpha \neq \beta \text{ \rm mod } \pi} \left|\tfrac{d^2}{dt^2}p_{\alpha,\beta}(f^t)\right| & \leq C\sin(\varepsilon)^{-8} , \\
\sup_{t \in [0,1], \alpha \neq \beta \text{ \rm mod } \pi} \left|\sin(\beta-\alpha)\tfrac{d}{dt}p_{\alpha,\beta}(f^t)\right| & \leq C \sin(\varepsilon)^{-5} \|f^1 - f^0\|_\infty , \\
\sup_{t \in [0,1], \alpha \neq \beta \text{ \rm mod } \pi} \left|\sin(\beta-\alpha)^{2}\tfrac{d^2}{dt^2}p_{\alpha,\beta}(f^t)\right| & \leq C \sin(\varepsilon)^{-6} \|f^1 - f^0\|_\infty^2 .
\end{align*}
\end{lem}

\begin{proof}
Fix $t \in [0,1]$ and $\alpha,\beta \in \R$ such that $\alpha \neq \beta  \text{ \rm mod } \pi$. By Lemma~\ref{lem:pfunction}, it suffices to consider the case where $0 < |\delta| \leq \frac{\pi}{2}$, with $\delta \defl \beta - \alpha$. In this range, we have $|\sin(\delta)| \leq |\delta| \leq \frac{\pi}{2}|\sin(\delta)|$ and $1 - \cos(\delta) \leq \sin(\delta)^2$. Note that the constants $C_k > 0$ that appear in the estimates below are independent of $\delta$.

We abbreviate $\Delta_x \defl f_x^1 - f_x^0$, $c_x \defl \cos(f^t_x)$ and $s_x \defl \sin(f^t_x)$ for $x \in \{\alpha,\beta\}$, $c_\delta \defl \cos(\delta)$, $s_\delta \defl \sin(\delta)$, $q_x \defl\Delta_x \sin(f^t_x)^{-1}$.

For the first estimate of the lemma we need to bound
\begin{align*}
S_1 & \defl p_x(\beta-\alpha,f^t_\alpha,f^t_\beta)\Delta_\alpha + p_y(\beta-\alpha,f^t_\alpha,f^t_\beta)\Delta_\beta \\
 & = \frac{2(c_\delta - c_\alpha c_\beta)}{s_\delta^2 s_\alpha^2 s_\beta^2}\left[(c_\delta c_\alpha - c_\beta)q_\alpha + (c_\delta c_\beta - c_\alpha)q_\beta\right] .
\end{align*}
Here we used Lemma~\ref{lem:derivative}. Note that $\alpha \mapsto q_\alpha$ satisfies the Lipschitz condition
\begin{align*}
\left|q_\alpha - q_\beta\right| & \leq \tfrac{1}{s_\alpha s_\beta}|s_\alpha\Delta_\beta - s_\beta \Delta_\alpha| = \tfrac{1}{s_\alpha s_\beta}\left|s_\alpha\Delta_\beta - s_\alpha\Delta_\alpha +  s_\alpha \Delta_\alpha - s_\beta\Delta_\alpha\right| \\
 & \leq \sin(\varepsilon)^{-2}(\Lip(\Delta) + \Lip(f^t)\|\Delta\|_\infty)|\delta| \\
 & \leq \sin(\varepsilon)^{-2}(2 + 1\cdot 2\pi)|\delta| \\
 & \leq 3\pi\sin(\varepsilon)^{-2}|\delta| .
\end{align*}
By setting $c_\delta = 1$ in the square bracket of the expression for $S_1$ above, we obtain an upper bound
\begin{align*}
\left|(c_\alpha - c_\beta)q_\alpha + (c_\beta - c_\alpha) q_\beta \right| & = \left|(c_\beta - c_\alpha)\left(q_\alpha - q_\beta\right)\right| \\
& \leq 3\pi \sin(\varepsilon)^{-2} |\delta|^2 \\
& \leq 50 \sin(\varepsilon)^{-2}s_\delta^2.
\end{align*}
The difference to the term with arbitrary $c_\delta$ is bounded by
\begin{align*}
\left|(1 - c_\delta) c_\alpha q_\alpha + (1 - c_\delta) c_\beta q_\beta \right| 
& \leq \left|c_\alpha q_\alpha + c_\beta q_\beta \right|s_\delta^2 \\
& \leq 2\|\Delta\|_\infty \sin(\varepsilon)^{-1}s_\delta^2 \\
& \leq 20\sin(\varepsilon)^{-1}s_\delta^2 .
\end{align*}
Combined, we have
\[
|S_1| \leq \tfrac{2}{s_\alpha^2s_\beta^2}(50 + 20)\sin(\varepsilon)^{-2}|c_\delta - c_\alpha c_\beta| \leq 280 \sin(\varepsilon)^{-6} .
\]
Applying Lemma~\ref{lem:derivative}, the second term we need to estimate is
\begin{align*}
S_2 & \defl p_{xx}(\beta-\alpha,f^t_\alpha,f^t_\beta)\Delta_\alpha^2 + 2 p_{xy}(\beta-\alpha,f^t_\alpha,f^t_\beta)\Delta_\alpha \Delta_\beta + p_{yy}(\beta-\alpha,f^t_\alpha,f^t_\beta)\Delta_\beta^2 \\
 & = \frac{2}{s_\delta^2 s_\alpha^2 s_\beta^2}\bigl[\bigl(c_\delta c_\alpha c_\beta (5 + c_\alpha^2) - (1 + 2c_\alpha^2)(c_\delta^2 + c_\beta^2)\bigr)q_\alpha^2 \\
 & \qquad + \bigl(c_\delta c_\alpha c_\beta (5 + c_\beta^2) - (1 + 2c_\beta^2)(c_\delta^2 + c_\alpha^2)\bigr)q_\beta^2 \\
 & \qquad + \bigl(c_\delta(1 + c_\alpha^2)(1 + c_\beta^2) - 2c_\alpha c_\beta(1 + c_\delta^2)\bigr)q_\alpha q_\beta \bigr].
\end{align*}
Let $A$ be the term in square brackets with $c_\delta$ set to $1$. It reads:
\begin{align*}
A & = \bigl((1 + c_\alpha^2 - 2 c_\alpha c_\beta) (c_\alpha c_\beta - 1) - (c_\alpha - c_\beta)^2\bigr)q_\alpha^2 \\
 & \quad + \bigl((1 + c_\beta^2 - 2 c_\alpha c_\beta) (c_\alpha c_\beta - 1) - (c_\alpha - c_\beta)^2\bigr)q_\beta^2 \\
 & \quad + 2\bigl((c_\alpha - c_\beta)^2 + (c_\alpha c_\beta - 1)^2\bigr)q_\alpha q_\beta .
\end{align*}
An upper bound for $A$ is obtained by
\begin{align*}
|A| & \leq 2\left|(1 + c_\alpha^2 - 2 c_\alpha c_\beta)q_\alpha^2 + (1 + c_\beta^2 - 2 c_\alpha c_\beta)q_\beta^2 + 2(c_\alpha c_\beta - 1)q_\alpha q_\beta\right| \\
 & \quad + (c_\alpha - c_\beta)^2\left|2q_\alpha q_\beta - q_\alpha^2 - q_\beta^2\right| \\
 & = 2\left|(1 - c_\alpha c_\beta)(q_\alpha - q_\beta)^2 + (c_\beta - c_\alpha)(c_\beta q_\beta^2 - c_\alpha q_\alpha^2)\right| \\
 & \quad + (c_\alpha - c_\beta)^2(q_\alpha - q_\beta)^2 \\
 & \leq 2\left(6\Lip(q)^2 + \Lip(\cos(f^t)q^2)\right) \delta^2 \\
 & \leq C_1 \sin(\varepsilon)^{-4} s_\delta^2
\end{align*}
for some constant $C_1 > 0$. In the last line we used $\Lip(q)^2 \leq C_2 \sin(\varepsilon)^{-4}$ and
\[
\Lip(\cos(f^t)q^2) \leq \Lip(f^t)\|q\|^2_\infty + 2\|\cos(f^t)\|_\infty\|q\|_\infty\Lip(q) \leq C_2 \sin(\varepsilon)^{-3}
\]
for some $C_2 > 0$. Since we plugged $c_\delta = 0$ in the square bracket, the absolute value of the difference is
\begin{align*}
|A - \tfrac{1}{2}s_\delta^2s_\alpha^2s_\beta^2S_2| & = \bigl|\left((1-c_\delta)c_\alpha c_\beta(5 + c_\alpha^2) - (1 + 2c_\alpha^2)(1-c_\delta^2)\right) q_\alpha^2 \\
& \quad + \left((1-c_\delta)c_\alpha c_\beta(5 + c_\beta^2) - (1 + 2c_\beta^2)(1-c_\delta^2) \right) q_\beta^2 \\
& \quad + 2\left((1-c_\delta)(1 + c_\alpha^2)(1 + c_\beta^2) - 2 c_\alpha c_\beta(1-c_\delta^2) \right) q_\alpha q_\beta\bigr|.
\end{align*}
This expression contains a factor of $1 - c_\delta$, and is therefore bounded above by $C_3\sin(\varepsilon)^{-2}s_\delta^2$ for some $C_3 > 0$. Hence,
\[
|S_2| \leq C_4\sin(\varepsilon)^{-8}
\]
for some $C_4 > 0$ as claimed. For the third estimate in the lemma, as for $S_1$ above,
\begin{align*}
s_\delta\left.\tfrac{d}{dt}\right|_{t=0} p_{\alpha,\beta}(f^t) & = \frac{2(c_\delta - c_\alpha c_\beta)} {s_\delta s_\alpha^2s_\beta^2} \left(\frac{c_\delta c_\alpha - c_\beta}{s_\alpha}\Delta_\alpha + \frac{c_\delta c_\beta - c_\alpha}{s_\beta}\Delta_\beta\right) .
\end{align*}
Further,
\begin{align*}
|c_\delta c_\alpha - c_\beta| & \leq |(1-c_\delta)c_\alpha| + |c_\alpha - c_\beta| \\
& \leq s_\delta^2 + |\delta| \leq |s_\delta|(1 + \tfrac{\pi}{2}) \\
& \leq 3|s_\delta|
\end{align*}
and similarly for $|c_\delta c_\beta - c_\alpha|$. Consequently,
\begin{equation*}
\left|s_\delta \left.\tfrac{d}{dt}\right|_{t=0} p_{\alpha,\beta}(f^t)\right| \leq \frac{24}{\sin(\varepsilon)^5} \|f^1 - f^0\|_\infty .
\end{equation*}
For the last estimate, the trivial bound
\[
\max\{\Delta_\alpha^2, |\Delta_\alpha\Delta_\beta|, \Delta_\beta^2\} \leq \|f^1 - f^0\|_\infty^2
\]
is applied to $S_2$ as given above.
\end{proof}

\subsection{Action on currents and paths}
\label{sec:actioncurrents}

Let $\varepsilon \in (0,\frac{\pi}{2})$, and assume that $T \in \bM_2(E_\varepsilon(\bS^1))$ is a metric current of finite mass, as recalled in Subsection~\ref{sec:finslermass}. The action of $\omega$ is defined by
\[
T(\omega) \defl \int_0^\pi \int_\alpha^\pi T(p_{\alpha,\beta}(f) \, d\pi_\alpha \wedge d\pi_\beta) \, d\beta \, d\alpha .
\]
In the notation of metric currents, the integrand can be expressed as $T(p_{\alpha,\beta}, \pi_\alpha, \pi_\beta)$, where $\pi_x : E(\bS^1) \to \R$ is the evaluation map $\pi_x(f) \defl f_x$ for $x \in \R$.

\begin{lem}
	\label{lem:action}
If $T \in \bM_2(E_\varepsilon(\bS^1))$ for some $\varepsilon \in (0,\frac{\pi}{2})$, then $T(\omega)$ is well-defined and depends only on $\partial T$. Moreover
\[
|T(\omega)| \leq C \bM(T) ,
\]
for some $C(\varepsilon) > 0$.
\end{lem}

\begin{proof}
The function
\[
(0,\pi)^3 \ni (a,x,y) \mapsto p(a,x,y),
\]
as defined in \eqref{eq:pdefinition}, is smooth. Consequently, the map
\[
\{(s,t) : 0 < s < t < \pi\} \times [\varepsilon, \pi - \varepsilon]^2 \ni (\alpha,\beta,x,y) \mapsto p(\beta - \alpha,x,y)
\]
is continuous in $(\alpha,\beta)$ and Lipschitz in $(x,y)$. Since $p_{\alpha,\beta}(f) = p(\beta - \alpha,f_\alpha,f_\beta)$, the function
\[
\{(s,t) : 0 < s < t < \pi\} \times E_\varepsilon(\bS^1) \ni (\alpha,\beta,f) \mapsto p_{\alpha,\beta}(f)
\]
is continuous in $(\alpha,\beta)$ and Lipschitz in $f$. Since the evaluation functionals $\pi_\alpha,\pi_\beta : E_\varepsilon(\bS^1) \to \R$ are also Lipschitz, we conclude that $T(p_{\alpha,\beta}, \pi_\alpha, \pi_\beta)$ is well-defined.

Since every $f \in E(\bS^1)$ is continuous, the evaluation $\alpha \mapsto \pi_\alpha(f)$ depends continuously on $\alpha$. Similarly, $(\alpha,\beta) \mapsto p_{\alpha,\beta}(f)$ is continuous and $f \mapsto p_{\alpha,\beta}(f)$ has a locally bounded Lipschitz constant. By the continuity axiom for metric currents, it follows that $(\alpha,\beta) \mapsto T(p_{\alpha,\beta}, \pi_\alpha, \pi_\beta)$ is continuous. Moreover, since $p_{\alpha,\beta}(f)$ is uniformly bounded by Lemma~\ref{lem:pfunction}(4) and $T$ has finite mass, the function $(\alpha, \beta) \mapsto T(p_{\alpha,\beta}, \pi_\alpha, \pi_\beta)$ is bounded and therefore integrable. Hence, $T(\omega)$ is well defined.

Furthermore,
\[
T(p_{\alpha,\beta}(f)\, d\pi_\alpha \wedge d\pi_\beta) = (\pi_\alpha,\pi_\beta)_\# T(g_{\alpha,\beta}(x,y)\,dx\wedge dy)
\]
where $g_{\alpha,\beta} : (0,\pi)^2 \to \R$ is the smooth function defined by $g_{\alpha,\beta}(x,y) \defl p(\beta - \alpha, x, y)$. Since $g_{\alpha,\beta}(x,y)\,dx \wedge dy$ is a closed 2-form on the contractible domain $(0,\pi)^2$, it is exact by the Poincaré lemma; that is, there exists a 1-form $\mu$ such that $g_{\alpha,\beta}(x,y)\,dx\wedge dy = d\mu$. Thence $(\pi_\alpha,\pi_\beta)_\# T(g_{\alpha,\beta}(x,y)\,dx\wedge dy) = (\pi_\alpha,\pi_\beta)_\# (\partial T)(\mu )$. This shows that $T(\omega)$ depends only on $\partial T$.

The mass bound follows directly from the uniform boundedness of $p_{\alpha,\beta}(f)$ and the fact that the evaluation maps $\pi_x : E(\bS^1) \to \R$ are $1$-Lipschitz.
\end{proof}

As noted in Subsection~\ref{subsec:inj_hull}, the metric space $E(\bS^1)$ can be identified isometrically with a subset of the space of measurable, essentially bounded, $2\pi$-periodic functions on $\R$. The latter is isometrically isomorphic to $L^\infty([0,2\pi))$. Since $L^\infty([0,2\pi))$ is the dual of the separable Banach space $L^1([0,2\pi))$, the results at the end of Subsection~\ref{sec:finslermass} apply. In particular, any $T \in \cR_2(E(\bS^1))$ can be represented by $\curr{S,\theta,\tau}$ in $\cR_2(L^\infty([0,2\pi)))$.

\begin{lem}
	\label{lem:actiononpaths}
Let $T \in \cR_2(E_\varepsilon(\bS^1))$ for some $\varepsilon \in (0,\frac{\pi}{2})$ with representation $\curr{S,\theta,\tau}$ in $\cR_2(L^\infty([0,2\pi)))$. Then
\begin{align*}
T(\omega) & = \int_S \theta(f) \int_0^\pi \int_\alpha^\pi p_{\alpha,\beta}(f) (\tau_{1,\alpha}(f) \tau_{2,\beta}(f) - \tau_{1,\beta}(f) \tau_{2,\alpha}(f)) \, d\beta \, d\alpha \,d\cH^2(f) .
\end{align*}
\end{lem}

\begin{proof}
The smoothing operator $A_\delta : L^\infty([0,2\pi))\to L^\infty([0,2\pi))$ for $\delta \in (0,\frac{\pi}{2})$ is defined by
\[
f^\delta_\alpha \defl A_\delta(f)_\alpha \defl \frac{1}{2\delta}\int_{\alpha-\delta}^{\alpha+\delta} f_t \,dt .
\]
Here we assume that $f$ is extended to a $2\pi$-periodic function on $\R$. The $\delta$-approximation of $\omega$ is defined by
\[
\omega_f^\delta \defl \int_0^\pi \int_\alpha^\pi p_{\alpha,\beta}(f)\, d\pi^\delta_\alpha \wedge d\pi^\delta_\beta \, d\beta \, d\alpha ,
\]
with the obvious action on $T$ given by
\[
T(\omega^\delta) \defl \int_0^\pi \int_\alpha^\pi T(p_{\alpha,\beta}, \pi_\alpha^\delta, \pi_\beta^\delta) \, d\beta \, d\alpha ,
\]
where $\pi_x^\delta : L^\infty([0,2\pi)) \to L^\infty([0,2\pi))$ is defined by $\pi_x^\delta = \pi_x \circ A_\delta$ for all $x$. It is easy to check that $A_\delta$ maps $E_\varepsilon(\bS^1)$ into $E_\varepsilon(\bS^1)$, $\|A_\delta\| \leq 1$ and $\lim_{\delta \to 0} A^\delta = \id$ uniformly on $E_\varepsilon(\bS^1)$. This allows us to apply the continuity axiom for metric currents together with the dominated convergence theorem to conclude that $\lim_{\delta \to 0} T(\omega^\delta) = T(\omega)$. The key point is that each $\pi_x^\delta$ is well defined and has finite operator norm for all $x$. Indeed,
\[
|\pi_x^\delta(f)| \leq \frac{1}{2\delta}\int_{\alpha-\delta}^{\alpha+\delta} |f_t| \,dt \leq \|f\|_\infty .
\]
Using \eqref{eq:currentrepres}, we can express
\[
T(\omega^\delta) = \int_0^\pi \int_\alpha^\pi \int_S \theta(f) p_{\alpha,\beta}(f) \left\langle \BigWedge_2 d^S_f (\pi^\delta_{\alpha}, \pi^\delta_{\beta}), \tau(f) \right\rangle\,d\cH^2(f) \, d\beta \, d\alpha .
\]
Note that for fixed $\alpha$, $\beta$ and $\delta$, the evaluation $(\pi^\delta_{\alpha}, \pi^\delta_{\beta}) : L^\infty([0,2\pi)) \to \R^2$ is well-defined, linear with operator norm bounded by $\sqrt{2}$. Consequently, the induced map
\[
\BigWedge_2 d^S_f (\pi^\delta_{\alpha}, \pi^\delta_{\beta}) : \BigWedge_2 \operatorname{Tan}^{(n)}(S,f) \to \BigWedge_2 \R^2
\]
is given by
\[
\left \langle \BigWedge_2 d^S_f (\pi^\delta_{\alpha}, \pi^\delta_{\beta}), v \wedge w\right\rangle = v^\delta_\alpha w^\delta_\beta - v^\delta_\beta w^\delta_\alpha
\]
for $\cH^2$-almost every $f \in S$, all $(\alpha, \beta) \in \Delta \defl \{(\alpha,\beta) \in \R^2 : 0 < \alpha < \beta < \pi\}$ and all vectors $v,w\in \operatorname{Tan}^{(n)}(S,f) \subset L^\infty([0,2\pi))$. We abbreviate $v \defl \tau_1$ and $w \defl \tau_2$. By applying Fubini’s theorem, we obtain
\[
T(\omega^\delta) = \int_S \int_\Delta I_\delta(q,f)\,d\cL^2(q)\,d\cH^2(f) ,
\]
where
\[
I_\delta((\alpha,\beta),f) \defl \theta(f) p_{\alpha,\beta}(f) \left(v^\delta_\alpha(f) w^\delta_\beta(f) - v^\delta_\beta(f) w^\delta_\alpha(f)\right) .
\]
We now provide the details for the prerequisites needed to apply Fubini’s theorem. The integrand $\Delta \times S \ni (q,f) \mapsto I_\delta(q,f)$ is $\cL^2 \otimes \cH^2$-measurable due to the following two facts:
\begin{itemize}
	\item For fixed $q \in \Delta$, the map $f \mapsto I_\delta(q,f)$ is Borel measurable. This follows from the measurability of $v$, $w$, and $\theta$, combined with the continuity of $A_\delta$ and the map $f \mapsto p_q(f)$, as established in Lemma~\ref{lem:smoothness}.
	\item For fixed $f \in S$, the map $q \mapsto I_\delta(q,f)$ is continuous. Indeed, the function $q \mapsto p_q(f)$ is continuous, and the maps $\alpha \mapsto v^\delta_\alpha(f)$ and $\alpha \mapsto w^\delta_\alpha(f)$ are Lipschitz continuous, since
	\[
	\left|v_\alpha^\delta(f) - v_\beta^\delta(f)\right| \leq \frac{C_2}{\delta}|\beta - \alpha|,
	\]
	as shown, for example, in the proof of \cite[Theorem~4.7]{Z1}. Note that $C_2 \geq \|v\|_\infty, \|w\|_\infty$ is some universal upper bound.
\end{itemize}
Additionally, for all $\delta$ we have the uniform bound $I_\delta(q,f) \leq \theta(f) M(\varepsilon)$, since $p_q(f)$ is uniformly bounded on $\Delta \times E_\varepsilon(\bS^1)$ by Lemma~\ref{lem:pfunction}(4). Moreover,
\[
\left|\left \langle \BigWedge_2 (\pi^\delta_{\alpha},\pi^\delta_{\beta}), \tau(f) \right\rangle\right| \leq \|(\pi^\delta_{\alpha},\pi^\delta_{\beta})\|^2\mu^{\rm bh}(\tau(f)) \leq 2,
\]
for $\cH^2$-almost every $f \in S$. This estimate follows from the normalization $\mu^{\mathrm{bh}}(\tau(f)) = 1$ and the operator norm bound $\|(\pi^\delta_{\alpha}, \pi^\delta_{\beta})\| \leq \sqrt{2}$.

For $\delta \to 0$, the dominated convergence theorem implies the limit identity 
\begin{align*}
T(\omega) & = \int_S \int_0^\pi \int_\alpha^\pi \theta(f) p_{\alpha,\beta}(f) (v_\alpha(f) w_\beta(f) - v_\beta(f) w_\alpha(f)) \, d\beta \, d\alpha \,d\cH^2(f) .
\end{align*}
We provide the details: On the left-hand side, the continuity of $T$ implies that $T(\omega^\delta) \to T(\omega)$ as observed previously. On the right-hand side, for each $\delta$, the function
\[
S \ni f \mapsto \int_\Delta I_\delta(p,f) \,d\cL^2(q)
\]
is $\cH^2$-measurable by Fubini’s theorem. Moreover,  the pointwise $\cH^2$-almost everywhere limit as $\delta \to 0$
\[
S \ni f \mapsto \int_0^\pi \int_\alpha^\pi \theta(f) p_{\alpha,\beta}(f) (v_\alpha(f) w_\beta(f) - v_\beta(f) w_\alpha(f)) \, d\beta \, d\alpha ,
\]
exists and is therefore also $\cH^2$-measurable. The existence of this limit follows from the uniform bound $I_\delta(q,f) \leq 2\theta(f) M(\varepsilon)$ for all $f \in S \subset E_\varepsilon(\bS^1)$, and from the pointwise convergence
\[
\lim_{\delta \to 0} v_\alpha^\delta(f) w_\beta^\delta(f) - v_\beta^\delta(f) w_\alpha^\delta(f) = v_\alpha(f) w_\beta(f) - v_\beta(f) w_\alpha(f)
\]
whenever $\alpha$ and $\beta$ are density points of $v(f)$ and $w(f)$.
\end{proof}

The Banach space
\[
L^\infty_{\rm{ap}}(\R) \defl \{f \in L^\infty(\R) : f_{\alpha + \pi} = -f_\alpha \text{ for almost every } \alpha\}
\]
is considered as a subspace of $L^\infty([0,2\pi))$ by restricting functions to $[0,2\pi)$, and is weak$\ast$ closed. 
The obvious isometric embedding realizes $E(\bS^1)$ in the affine subspace
\[
E(\bS^1) \subset \tfrac{\pi}{2} + L^\infty_{\rm{ap}}(\R) \subset L^\infty([0,2\pi)) .
\]
If $S \subset E(\bS^1)$ is as in the lemma above, with orienting vector fields $\tau_1,\tau_2 : S \to L^\infty([0,2\pi))$, then $\tau_1$ and $\tau_2$ take values in $L^\infty_{\rm{ap}}(\R)$. Note that they arise as weak$\ast$ derivatives of parametrizations of $S$. We identify $L^\infty_{\rm{ap}}(\R)$ with $L^\infty([0,\pi))$ when convenient.

With this lemma at hand, the pointwise definition
\[
\omega_f(v \wedge w) \defl \int_0^\pi \int_\alpha^\pi p_{\alpha,\beta}(f) (v_{\alpha} w_{\beta} - v_{\beta} w_{\alpha}) \, d\beta \, d\alpha ,
\]
for $f \in E(\bS^1)\setminus \bS^1$ and $v,w \in L^\infty([0,\pi))$ is useful. The pointwise comass $\|\omega_f\|_{\rm ir}$ of $\omega$ at $f \in E(\bS^1)\setminus \bS^1$ is defined as the infimum over all constants $M \geq 0$ such that
\begin{equation}
\label{eq:comassdef}
|\omega_f(v \wedge w)| \leq M \mu^{\rm ir}(v \wedge w) 
\end{equation}
holds for all $v,w \in L^\infty([0,\pi))$. It is clear from this definition that $\|\omega_f\|_{\rm ir}$ depends only on the plane spanned by $v$ and $w$. Building on Lemma~\ref{lem:actiononpaths} and the definitions above, we obtain the following characterizations of mass and comass.

Let $T \in \cR_2(E_\varepsilon(\bS^1))$ be represented as $T = \curr{S,\theta,\tau} \in \cR_2(L^\infty([0,2\pi)))$. Then, by Lemma~\ref{lem:masscor}, it follows that
\[
\bM_{\rm ir}(T) = \int_S \theta(f) \lambda^{\rm ir}_{\operatorname{Tan}^{(2)}(S,f)}\, d\cH^2(f) = \int_S \theta(f) \mu^{\rm ir}(\tau_1(f) \wedge \tau_2(f))\, d\cH^2(f) ,
\]
where
\[
\lambda^{\rm ir}_{\operatorname{Tan}^{(2)}(S,f)} = \frac{\mu^{\rm ir}(\tau_1(f) \wedge \tau_2(f))}{\mu^{\rm bh}(\tau_1(f) \wedge \tau_2(f))}
\]
for $\cH^2$-almost every $f \in S$.

\begin{prop}
	\label{prop:massbound}
Let $T \in \cR_2(E_\varepsilon(\bS^1))$ for some $\varepsilon \in (0,\frac{\pi}{2})$ with representation $\curr{S,\theta,\tau}$ in $\cR_2(L^\infty([0,2\pi))$. Then
\[
|T(\omega)| \leq \bM_{\rm ir}(T)\sup_{f \in \spt(T)}\|\omega_f\|_{\rm ir} .
\]
Moreover,
\begin{equation}
\label{eq:omegamass}
\|\omega_f\|_{\rm ir} = \sup\left\{\omega_f(v\wedge w) : v,w \in L^\infty([0,\pi)), \|v^2 + w^2\|_\infty \leq 1\right\}
\end{equation}
for all $f \in E(\bS^1)\setminus \bS^1$.
\end{prop}

\begin{proof}
Due to Lemma~\ref{lem:actiononpaths}, we have
\begin{align*}
|T(\omega)| & = \int_S \theta(f) \omega_f(\tau_1(f) \wedge \tau_2(f))\,d\cH^2(f) \\
 & \leq \sup_{f \in \spt(T)}\|\omega_f\|_{\rm ir} \int_S \theta(f)\mu^{\rm ir}(\tau_1(f) \wedge \tau_2(f))\,d\cH^2(f) \\
 & = \sup_{f \in \spt(T)}\|\omega_f\|_{\rm ir} \int_S \theta(f)\lambda^{\rm ir}_{\operatorname{Tan}^{(n)}(S,f)}\,d\cH^2(f) \\
 & = \sup_{f \in \spt(T)}\|\omega_f\|_{\rm ir}\bM_{\rm ir}(T) .
\end{align*}
This establishes the first part.

For the second part, let $V \subset L^\infty([0,\pi))$ be a two-dimensional subspace, $\mathcal{E}_V$ the L\"owner–John ellipse of the closed unit ball $\B_V$ of $(V, \|\cdot\|\infty)$, and $e : V \to [0,\infty)$ the Euclidean norm with unit ball $\mathcal{E}_V$. Equation \eqref{eq:omegamass} follows directly from the inequality
\begin{equation}
	\label{eq:irnormestimate}
\mu_{\|\cdot\|_\infty}^{\rm ir}(v \wedge w) \leq \|v^2 + w^2\|_\infty
\end{equation}
which holds for all $v,w \in V$, with equality if and only if $e(v) = e(w)$ and $v,w$ are orthogonal with respect to $e$. The right-hand side above can equivalently be expressed as
\[
\|v^2 + w^2\|_\infty = \sup_t \|\cos(t)v + \sin(t)w\|^2_\infty.
\]

Let $v,w \in V$ be linearly independent with $\|v^2 + w^2\|_\infty = 1$. This implies that for all $t$, $\|\cos(t)v + \sin(t)w\|^2_\infty \leq 1$, so the path $t \mapsto \cos(t)v + \sin(t)w$ traces the boundary of an ellipse $E$ contained in $\B_V$. The area of $E$ (or any origin-symmetric ellipse in the Euclidean space $(V,e)$) is given by 
\[
\pi\max_{x,y \in E} \mu_e(x \wedge y) = \pi\max_{x,y \in E} \mu_{\|\cdot\|_\infty}^{\rm ir}(x \wedge y) .
\]
Since $\mathcal E_V$ is the ellipse of maximal area in $\B_V$, it follows that
\[
\mu^{\rm ir}_{\|\cdot\|_\infty}(v \wedge w) \leq \frac{1}{\pi} \mu_{\|\cdot\|_\infty}^{\rm ir}(E) \leq \frac{1}{\pi} \mu_{\|\cdot\|_\infty}^{\rm ir}(\mathcal E_V) = 1 = \|v^2 + w^2\|_\infty.
\]
Since $\mathcal E_V$ is uniquely determined, equality can occur only if the curve $t \mapsto \cos(t)v + \sin(t)w$ traces the boundary of $\mathcal{E}_V$. Consequently, $e(v) = e(w) = 1$, and evaluating at $t = \frac{\pi}{4}$ yields $1 = e(\frac{1}{\sqrt{2}}(v + w))$. By the standard polarization identities, this implies that $v$ and $w$ form an orthonormal basis with respect to $e$.

If $v$ and $w$ form an orthonormal basis of $V$ with respect to $e$, then
\begin{align*}
1 & = e(\cos(t)v + \sin(t)w)^2 \geq \|\cos(t)v + \sin(t)w\|^2_\infty
\end{align*}
holds for all $t \in \R$. This implies 
\[
\|v^2 + w^2\|_\infty \leq 1 = e(v)e(w) = \mu_{\|\cdot\|_\infty}^{\rm ir}(v \wedge w) .
\]
This establishes inequality \eqref{eq:irnormestimate}, and consequently also \eqref{eq:omegamass}.
\end{proof}

We can view $\omega_f$ for $f \in E(\bS^1)\setminus \bS^1$ as an operator on plane paths. More precisely, we consider the path spaces
\begin{align*}
L^\infty_{\rm ap}(\R,\R^2) & \defl \left\{(\gamma_1,\gamma_2) : \gamma_1,\gamma_2 \in L_{\rm ap}^\infty(\R)\right\} , \\
B^\infty_{\rm ap}(\R,\R^2) & \defl \left\{\gamma \in L^\infty_{\rm ap}(\R,\R^2) : \|\gamma\|_\infty \defl \|\gamma_1^2 + \gamma_2^2\|_\infty \leq 1 \right\} .
\end{align*}
The action of $\omega_f$ on $\gamma \in L^\infty_{\rm ap}(\R,\R^2)$ is defined by
\begin{equation}
	\label{eq:pathaction}
\omega_f(\gamma) \defl \int_0^\pi \int_\alpha^\pi p_{\alpha,\beta}(f) \, \gamma(\alpha) \times \gamma(\beta) \, d\beta \, d\alpha ,
\end{equation}
where $v \times w = v_1w_2 - v_2w_1$ is the signed area of the parallelogram spanned by $v,w \in \R^2$. With this notation, it is understood that the coordinate functions of $\gamma$ represent vectors in $L^\infty([0,\pi))$. The above proposition shows that $\|\omega_f\|_{\rm ir}$ can be expressed as
\begin{equation}
\label{eq:curveaction}
\|\omega_f\|_{\rm ir} = \sup_{\gamma \in B^\infty_{\rm ap}(\R,\R^2)} \omega_f(\gamma)
\end{equation}
for $f \in E(\bS^1)\setminus \bS^1$.

We conclude this subsection by justifying that the action of $\omega_f$ on $\gamma \in L^\infty_{\rm ap}(\R,\R^2)$ does not depend on the choice of base point in $\bS^1$ used to define $\omega_f$. By symmetry,
\begin{align*}
0 & = \int_0^{\pi} \int_0^{\pi} p_{\alpha,\beta}(f) \, \gamma(\alpha) \times \gamma(\beta) \, d\beta \, d\alpha,
\end{align*}
since $\gamma(\alpha) \times \gamma(\beta) = -\gamma(\beta) \times \gamma(\alpha)$ and $p_{\alpha,\beta}(f) = p_{\beta,\alpha}(f)$ by Lemma~\ref{lem:pfunction}. Thus,
\begin{align*}
\omega_f(\gamma) & = -\int_0^{\pi} \int_0^{\alpha} p_{\alpha,\beta}(f) \, \gamma(\alpha) \times \gamma(\beta) \, d\beta \, d\alpha \\
 & = \int_0^{\pi} \int_\pi^{\alpha + \pi} p_{\alpha,\beta}(f) \, \gamma(\alpha) \times \gamma(\beta) \, d\beta \, d\alpha,
\end{align*}
because $p_{\alpha,\beta + \pi}(f) = p_{\alpha,\beta}(f)$ and $\gamma(\alpha) \times \gamma(\beta + \pi) = - \gamma(\alpha) \times \gamma(\beta)$. Hence,
\begin{align}
	\nonumber
\omega_f(\gamma) & = \frac{1}{2}\int_0^{\pi} \int_\alpha^{\alpha + \pi} p_{\alpha,\beta}(f) \, \gamma(\alpha) \times \gamma(\beta) \, d\beta \, d\alpha \\
	\label{eq:indepbasepoint}
 & = \frac{1}{4}\int_0^{2\pi} \int_\alpha^{\alpha + \pi} p_{\alpha,\beta}(f) \, \gamma(\alpha) \times \gamma(\beta) \, d\beta \, d\alpha,
\end{align}
because $p_{\alpha + \pi, \beta + \pi}(f) = p_{\alpha,\beta}(f)$ and $\gamma(\alpha + \pi) \times \gamma(\beta + \pi) = \gamma(\alpha) \times \gamma(\beta)$. The final expression is manifestly independent of the choice of base point in $\bS^1$.

\subsection{Coefficients of product type}
\label{sec:prodtype}

In this subsection it is assumed that the coefficient function $p : \R \times \R \to \R$ are of product type. More precisely, we assume that
\begin{enumerate}
	\item $p_{\alpha,\beta} = p_\alpha p_\beta$,
	\item $p : \R \to \R$ is locally integrable and $\pi$-periodic,
	\item $p_\alpha > 0$ for almost every $\alpha$.
\end{enumerate}
Note that the coefficients $p_{\alpha,\beta}(f)$ are of this type in case $f \in \bS^2_+\setminus \bS^1$ by Lemma~\ref{lem:geometriccoeff} and Lemma~\ref{lem:pfunction}. These coefficients act on paths $\gamma \in L^\infty_{\rm ap}(\R,\R^2)$ by
\begin{equation}
	\label{eq:omegadef}
\omega_p(\gamma) \defl \int_0^\pi \int_\alpha^\pi p_{\alpha,\beta} \, \gamma(\alpha) \times \gamma(\beta) \, d\beta \, d\alpha .
\end{equation}
By the planar isoperimetric inequality, there exists a unique maximizer $\gamma \in B^\infty_{\mathrm{ap}}(\R, \R^2)$ of $\omega_p$, up to rotations of $\R^2$.

\begin{lem}
	\label{lem:prodtype}
Assume that $p_{\alpha,\beta} = p_\alpha p_\beta$ is as above, and let $\nu : \R \to \R$ be the homeomorphism satisfying $\nu(0) = 0$ and
\[
\nu'(\alpha) = p_\alpha \left(\frac{1}{2\pi}\int_0^{2\pi}p_\beta\,d\beta \right)^{-1}
\]
for almost every $\alpha$. Define
\[
\gamma(\alpha) \defl e^{i\nu(\alpha)} \quad \text{and} \quad \sigma (\alpha) \defl \frac{1}{2}\int_\alpha^{\alpha + \pi} p_\beta \gamma(\beta)\, d\beta .
\]
It follows that $\gamma \in B^\infty_{\mathrm{ap}}(\R, \R^2)$ and
\[
\sup_{\delta \in B^\infty_{\rm ap}(\R,\R^2)} \omega_p(\delta) = \omega_p(\gamma) =  \operatorname{Area}(\sigma|_{[0,2\pi]}) = \frac{1}{4\pi} \left(\int_0^{2\pi}p_\alpha\,d\alpha\right)^2 .
\]
Moreover, any maximizer $\delta \in B^\infty_{\mathrm{ap}}(\R, \R^2)$ of $\omega_p$ is of the form $\alpha \mapsto e^{i\nu(\alpha) + ic}$ for some $c \in \R$ and satisfies $|\delta'(\alpha)| = \nu'(\alpha)$ for almost every $\alpha$.
\end{lem}

\begin{proof}
For an antipodal plane path $\gamma \in B^\infty_{\rm ap}(\R,\R^2)$, define $\sigma$ as in the statement. Note that $\sigma \in L^\infty_{\rm ap}(\R,\R^2)$ due to the symmetries of $p$ and $\gamma$. Moreover, $\sigma'(\alpha) = -p_\alpha\gamma(\alpha)$, so that $|\sigma'(\alpha)| \leq p_\alpha$ for almost every $\alpha$. By the properties of $p$ and $\gamma$, we have
\begin{align*}
\omega_p(\gamma) & = \int_0^\pi \int_\alpha^\pi p_{\alpha,\beta} \gamma(\alpha) \times \gamma(\beta)\,d\beta\, d\alpha \\
& = \frac{1}{2} \int_0^\pi \int_\alpha^{\alpha + \pi} p_{\alpha} \gamma(\alpha) \times p_{\beta}\gamma(\beta)\,d\beta\, d\alpha \\
& = \int_0^\pi p_{\alpha} \gamma(\alpha) \times \sigma(\alpha) \,d\alpha \\
& = \int_0^\pi \sigma(\alpha) \times \sigma'(\alpha) \,d\alpha \\
& = \frac{1}{2}\int_0^{2\pi} \sigma(\alpha) \times \sigma'(\alpha) \,d\alpha .
\end{align*}
This is the signed area $\operatorname{Area}(\sigma|_{[0,2\pi]})$ enclosed by $\sigma|_{[0,2\pi]}$. By the isoperimetric inequality for plane curves, it follows that
\begin{equation}
\label{eq:producttype}
\operatorname{Area}(\sigma|_{[0,2\pi]}) \leq \frac{1}{4\pi} \operatorname{Length}(\sigma|_{[0,2\pi]})^2 \leq \frac{1}{4\pi} \left(\int_0^{2\pi}p_\alpha\,d\alpha\right)^2 .
\end{equation}

Assume now that $\sigma$ achieves equality. If $|\gamma(\alpha)| < 1$ on a set of positive measure, then the inequality is strict, since $L(\sigma|_{[0,2\pi]}) < \int_0^{2\pi}p_\alpha\,d\alpha$ due to $p_\alpha > 0$ almost everywhere. Hence, equality in \eqref{eq:producttype} holds if and only if $|\gamma(\alpha)| = 1$ almost everywhere and $\sigma$ is a counterclockwise pa\-ram\-e\-trization of the circle around the origin with radius $r \defl \frac{1}{2\pi}\int_0^{2\pi}p_\alpha\,d\alpha > 0$. In this situation, $|\sigma'(\alpha)| = p_\alpha = r\nu'(\alpha)$ for almost every $\alpha$, and it follows that $\sigma(\alpha) = r e^{i\nu(\alpha) + ic}$ for some $c \in \R$. This implies
\[
-p_\alpha \gamma(\alpha) = \sigma'(\alpha) = r i \nu'(\alpha)e^{i\nu(\alpha) + ic}
\] 
and hence
\[
\gamma(\alpha) = -ie^{i\nu(\alpha) + ic} = e^{i\nu(\alpha) + i(c -\frac{\pi}{2})}
\]
for almost every $\alpha$. In particular,
\[
|\gamma'(\alpha)| = \nu'(\alpha) = p_\alpha r^{-1} .
\]
Thus, $\gamma$ is a counterclockwise parametrization of the unit circle with speed $\nu'(\alpha)$. It is now straightforward to verify that any path of the form $\alpha \mapsto e^{i\nu(\alpha) + ic}$ achieves equality in \eqref{eq:producttype}.
\end{proof}

Such a product structure is present for the coefficients induced by $h \in \bS^2_+\setminus \bS^1$ as established in Lemma~\ref{lem:geometriccoeff}. In fact, one may write
\[
p_{\alpha,\beta}(h) = \frac{\sin(d(h))^2}{\sin(h_\alpha)^2\sin(h_\beta)^2} = \frac{\sin(d(h))}{\sin(h_\alpha)^2} \frac{\sin(d(h))}{\sin(h_\beta)^2} \defr p_\alpha(h) p_\beta(h) .
\]
By Lemma~\ref{lem:tangentsphere}, we have $1 = \frac{1}{2\pi} \int_0^{2\pi}p_\alpha(h)\,d\alpha$. It follows that a maximizer $\gamma \in B^\infty_{\mathrm{ap}}(\R, \R^2)$ for $\omega_h$, which is unique up to rotations of $\R^2$, parametrizes the unit circle with $|\gamma'(\alpha)| = p_\alpha(h)$, and
\begin{equation}
\label{eq:spherecalibration}
\omega_h(\gamma) = \pi .
\end{equation}
The corresponding $\sigma$ likewise parametrizes a unit circle.

\section{Variational analysis}

\subsection{Structure of maximizing paths}
\label{sec:structure}
Throughout this subsection, we assume that the measurable coefficient function $p : \R \times \R \to \R$ satisfies the following conditions:
\begin{enumerate}
	\item $p_{\alpha,\beta} > 0$ almost everywhere,
	\item $p_{\alpha,\beta} = p_{\beta,\alpha}$,
	\item $p_{\alpha,\beta}$ is $\pi$-periodic in both arguments,
	\item $p_{\alpha,\beta}$ is (essentially) uniformly bounded.
\end{enumerate}
Conditions (2),(3) and (4) are satisfied by $p_{\alpha,\beta}(f)$ for all $f \in E(\bS^1)\setminus \bS^1$ by Lemma~\ref{lem:pfunction}. By definition, condition (1) is also satisfied if $f \in E^+(\bS^1)$. As in \eqref{eq:omegadef}, these coefficients define an action $\omega_p$ on $L^\infty_{\rm ap}(\R,\R^2)$. This action is well-defined since the integrand is measurable and uniformly bounded. To establish the existence of maximizing paths for $\omega_p$, we first prove that this action is weak$\ast$ continuous by using the duality $L^\infty([0,\pi))=L^1([0,\pi))^\ast$.

\begin{lem}
	\label{lem:weakstarlem}
Let $p$ and $\omega_p$ as above, and let $(\gamma_n)$ be a sequence in $L^\infty([0,\pi),\R^2)$ that converges with respect to the weak$\ast$ topology to $\gamma$, that is, the corresponding coordinate functions converge, then
\[
\lim_{n\to\infty}\omega_p(\gamma_n) = \omega_p(\gamma) .
\]
\end{lem}

\begin{proof}
We define the bilinear form
\[
B(\gamma,\delta) \defl \int_0^\pi\int_\alpha^\pi p_{\alpha,\beta}\, \gamma(\alpha) \times \delta(\beta)\,d\beta\,d\alpha
\]
for $\gamma,\delta \in L^\infty([0,\pi),\R^2)$, and show that it is (sequentially) weak$\ast$ continuous. So let $(\gamma_{n})$ and $(\delta_{n})$ be sequences that converge to $\gamma$ and $\delta$ respectively in the weak$\ast$ topology. In particular, by the Banach–Steinhaus theorem, both sequences are bounded in $L^\infty([0,\pi),\R^2)$. By the bilinearity of $B$, we have
\[
B(\gamma_n,\delta_n) - B(\gamma,\delta) = B(\gamma_n - \gamma,\delta) + B(\gamma_n,\delta_n-\delta) ,
\]
so it suffices to consider the cases $\gamma = 0$ or $\delta = 0$. First assume that $\delta_n \stackrel{\ast}{\to} 0$. Note that 
\begin{align*}
|B(\gamma_n,\delta_n)| = \left|\int_0^\pi \gamma_n(\alpha) \times \mu_n(\alpha) \,d\alpha\right| \leq \int_0^\pi |\gamma_n(\alpha)||\mu_n(\alpha)| \,d\alpha ,
\end{align*}
where
\[
\mu_n(\alpha) \defl \int_0^\pi \chi_{[\alpha,\pi]}(\beta) p_{\alpha,\beta} \delta_n(\beta)\,d\beta .
\]
Now, $(\mu_n)$ is bounded since both $p_{\alpha,\beta}$ and $(\delta_n)$ are bounded, and for each fixed $\alpha$, $\mu_n(\alpha) \to 0$ because $\delta_n \stackrel{\ast}{\rightarrow} 0$. Since both $(\gamma_n)$ and $(\mu_n)$ are bounded sequences, the bounded convergence theorem implies $B(\gamma_n,\delta_n) \to 0$.

In case $\gamma_n \stackrel{\ast}{\to} 0$, Fubini's theorem implies
\begin{align*}
B(\gamma_n,\delta_n) & = \int_0^\pi \int_0^\beta p_{\alpha,\beta} \gamma_n(\alpha) \times \delta_n(\beta)\,d\alpha\,d\beta \\
 & = - \int_0^\pi \delta_n(\beta) \times \int_0^\beta p_{\alpha,\beta} \gamma_n(\alpha)\,d\alpha\,d\beta .
\end{align*}
The argument then proceeds exactly as in the first case.
\end{proof}

As an application, the direct method in the calculus of variations applies.

\begin{lem}
	\label{lem:compactness}
The functional $\gamma \mapsto \omega_p(\gamma)$ has a maximum in $B^\infty_{\rm ap}(\R,\R^2)$.
\end{lem}

\begin{proof}
Note that $B^\infty_{\rm ap}(\R,\R^2)$ can be identified with those elements $\gamma = (\gamma_x,\gamma_y) \in L^\infty([0,\pi),\R^2)$ satisfying $\|\gamma_x^2 + \gamma_y^2\|_\infty \leq 1$. It is clear that $\omega_p$ is bounded on $B^\infty_{\rm ap}(\R,\R^2)$. Denote its supremum by $S$, and let $(\gamma_n)$ be a sequence in $B^\infty_{\rm ap}(\R,\R^2)$ such that $\lim_{n\to\infty}\omega_p(\gamma_n) = S$. Since $L^\infty([0,\pi))$ is the dual space of the separable Banach space $L^1([0,\pi))$, the Banach–Alaoglu theorem applied to the coordinate functions of $(\gamma_n)$ ensures the existence of a subsequence converging weak$\ast$ to some $\gamma \in L^\infty([0,\pi),\R^2)$. Since $\omega_p$ is weak$\ast$ continuous by Lemma~\ref{lem:weakstarlem}, we conclude that $\omega_p(\gamma) = S$.

It remains to verify that $\gamma \in B^\infty_{\rm ap}(\R,\R^2)$. To this end, note that for $v,w \in L^\infty([0,\pi))$, we have
\[
\|v^2 + w^2\|_\infty^{\frac{1}{2}} = 
\sup_{\substack{
		\|g\|_1 \leq 1 \\
		\|a^2 + b^2\|_\infty \leq 1
}} 
\int_0^\pi g(t) \big( a(t) v(t) + b(t) w(t) \big) \, dt.
\]
This identity follows from the Cauchy–Schwarz inequality together with the duality representation $\|f\|_\infty = \sup_{\|g\|_1 \leq 1}\int_0^\pi g(t)f(t)\,dt$ for $f \in L^\infty([0,\pi))$. If $(v_n,w_n) \stackrel{\ast}{\to} (v,w)$ in $L^\infty([0,\pi))^2$ and $g,a,b$ are as above, then $ga$ and $gb$ belong to $L^1([0,\pi))$. It follows that
\begin{align*}
\int_0^\pi g(t)(a(t)v(t) + b(t)w(t))\,dt & = \lim_{n\to\infty}\int_0^\pi g(t)(a(t)v_n(t) + b(t)w_n(t))\,dt \\
 & \leq \limsup_{n\to\infty}\|v_n^2 + w_n^2\|_\infty^\frac{1}{2} .
\end{align*}
Taking the supremum over all such $g,a,b$, we obtain
\[
\|v^2 + w^2\|_\infty^\frac{1}{2} \leq \limsup_{n\to\infty}\|v_n^2 + w_n^2\|_\infty^\frac{1}{2} .
\]
This shows that the limit $\gamma$ lies in $B^\infty_{\rm ap}(\R,\R^2)$, as the approximating sequence $(\gamma_n)$ belongs to this set.
\end{proof}

Using a variation argument, we show that any maximizer $\gamma$ of $\omega_p$ takes values in the unit circle. This relies on the strict positivity of the coefficient function $p_{\alpha,\beta}$.

\begin{lem}
	\label{lem:gammaboundary}
Let $\gamma \in B^\infty_{\rm ap}(\R,\R^2)$ be a maximizer of $\omega_p(\gamma)$. Then $|\gamma(\alpha)| = 1$ for almost every $\alpha$.
\end{lem}

\begin{proof}
We aim to show that the set $A = \{\alpha \in [0,\pi) : |\gamma(\alpha)| < 1\}$ has measure zero. Let $\delta \in L^\infty_{\rm ap}(\R,\R^2)$ be such that $|\delta(\alpha)| + |\gamma(\alpha)| \leq 1$ for almost every $\alpha \in [0,\pi)$. Then $\gamma + t\delta$ is in $B^\infty_{\rm ap}(\R,\R^2)$ for all $t \in [-1,1]$, and thus
\begin{align*}
0 & \geq \left.\frac{d^2}{d t^2}\right|_{t=0}\omega_f(\gamma + t\delta) \\
& = \left.\frac{d^2}{d t^2}\right|_{t=0} \int_0^\pi \int_\alpha^\pi p_{\alpha,\beta}\, (\gamma + t\delta)(\alpha) \times (\gamma + t\delta)(\beta)\,d\beta\,d\alpha \\
& = 2\int_0^\pi \int_\alpha^\pi p_{\alpha,\beta}\, \delta(\alpha) \times \delta(\beta) \,d\beta\,d\alpha .
\end{align*}
Because $p_{\alpha,\beta} > 0$ almost everywhere, we can vary $\delta$ on the set $A$ to conclude that $A$ has measure zero. For example, for each $n \in \N$, define 
\[
\delta(\alpha) = \frac{1}{n} e^{i\alpha} \quad \text{on} \quad A_n \defl \left\{\alpha \in [0,\pi) : |\gamma(\alpha)| \leq 1 - \frac{1}{n}\right\}
\]
and set $\delta(\alpha) = 0$ on $[0,\pi) \setminus A_n$. Then the maximality of $\gamma$ implies that $\mathcal{L}^1(A_n) = 0$ for all $n$, and hence $\mathcal{L}^1(A) = 0$.
\end{proof}

We can extract further information by considering suitable variations of $\gamma$.

\begin{lem}
	\label{lem:gammaprop}
Let $\gamma \in B^\infty_{\rm ap}(\R,\R^2)$ be a maximizer of $\omega_p$ and define
\[
\mu(\alpha) \defl \int_\alpha^{\alpha+\pi} p_{\alpha,\beta}\,\gamma(\beta)\,d\beta.
\]
Then $\mu \in L^\infty_{\rm ap}(\R,\R^2)$, and for almost every $\alpha$ the vector $\mu(\alpha)$ is orthogonal to $\gamma(\alpha)$ and satisfies $\gamma(\alpha) \times \mu(\alpha) \geq 0$.
\end{lem}

\begin{proof}
Because $p$ is bounded and satisfies $p_{\alpha + \pi,\beta + \pi} \gamma(\beta + \pi) = -p_{\alpha,\beta}\gamma(\beta)$, we have $\mu \in L^\infty_{\rm ap}(\R,\R^2)$. By Lemma~\ref{lem:gammaboundary} we may assume that $|\gamma(\alpha)| = 1$ almost everywhere, so that $\gamma(\alpha) = e^{i\eta(\alpha)}$ for some measurable function $\eta : \R \to \R$ with the property $\eta(\alpha + \pi) = \eta(\alpha) + \pi$ for all $\alpha$. We can further assume that $\eta(\alpha) \in (-\pi, \pi]$ for $\alpha \in [0,\pi)$. Let $\delta : \R \to \R$ be a bounded, $\pi$-periodic measurable function and consider the variation $\gamma_t \defl e^{i(\eta + t\delta)}$ in $B^\infty_{\rm ap}(\R,\R^2)$. Define
\begin{align*}
v(\alpha) \defl \left.\frac{d}{d t}\right|_{t=0} \gamma_t(\alpha) = i\delta(\alpha) e^{i\eta(\alpha)} , \quad
w(\alpha) \defl \left.\frac{d^2}{d t^2}\right|_{t=0} \gamma_t(\alpha) = -\delta(\alpha)^2 e^{i\eta(\alpha)} . 
\end{align*}
By the maximality of $\gamma$, we have
\begin{align*}
0 & = \left.\frac{d}{d t}\right|_{t=0}\omega_p(\gamma_t) = \int_0^\pi \int_\alpha^\pi p_{\alpha,\beta} [v(\alpha) \times \gamma(\beta) + \gamma(\alpha) \times v(\beta)]\,d\beta\,d\alpha \\
& = \int_0^\pi v(\alpha) \times \int_\alpha^\pi p_{\alpha,\beta} \gamma(\beta)\,d\beta\,d\alpha-\int_0^\pi v(\beta) \times \int_0^\beta p_{\alpha,\beta}\gamma(\alpha)\,d\alpha\,d\beta \\
& = \int_0^\pi v(\alpha) \times \int_\alpha^\pi p_{\alpha,\beta} \gamma(\beta)\,d\beta\,d\alpha -\int_0^\pi v(\alpha) \times \int_0^\alpha p_{\alpha,\beta}\gamma(\beta)\,d\beta\,d\alpha \\
& = \int_0^\pi v(\alpha) \times \mu(\alpha)\,d\alpha \\
& = \int_0^\pi \delta(\alpha)\,(i\gamma(\alpha)) \times \mu(\alpha)\,d\alpha .
\end{align*}
Since this holds for all bounded $\delta$, it follows that $\gamma$ and $\mu$ are orthogonal almost everywhere. This establishes the first statement.

The second variation satisfies
\begin{align*}
0 & \geq \left.\frac{d^2}{d t^2}\right|_{t=0}\omega_p(\gamma_t) \\
& = \int_0^\pi \int_\alpha^\pi p_{\alpha,\beta} [\gamma(\alpha) \times w(\beta) + w(\alpha) \times \gamma(\beta) + 2 v(\alpha) \times v(\beta)]\,d\beta\,d\alpha \\
& = \int_0^\pi \int_\alpha^\pi p_{\alpha,\beta} \left[(-\delta(\alpha)^2 - \delta(\beta)^2)\gamma(\alpha) \times \gamma(\beta) + 2 \delta(\alpha)\delta(\beta)\gamma(\alpha) \times \gamma(\beta)\right]\,d\beta\,d\alpha \\
& = -\int_0^\pi \int_\alpha^\pi p_{\alpha,\beta} (\delta(\alpha) - \delta(\beta))^2\gamma(\alpha) \times \gamma(\beta) \,d\beta\,d\alpha .
\end{align*}
Fix $0 < a < b < \pi$ and let $\delta$ be the $\pi$-periodic extension of $\chi_{[a,b]}$, then
\begin{align*}
0 & \leq \int_0^a \int_a^b p_{\alpha,\beta} \gamma(\alpha) \times \gamma(\beta) \,d\beta\,d\alpha + \int_a^b \int_b^\pi p_{\alpha,\beta} \gamma(\alpha) \times \gamma(\beta) \,d\beta\,d\alpha \\
& = \int_a^b \int_0^a p_{\alpha,\beta} \gamma(\alpha) \times \gamma(\beta) \,d\alpha\,d\beta + \int_a^b \int_b^\pi p_{\alpha,\beta} \gamma(\alpha) \times \gamma(\beta) \,d\beta\,d\alpha \\
& = \int_a^b \int_0^a p_{\beta,\alpha} \gamma(\beta) \times \gamma(\alpha) \,d\beta \,d\alpha + \int_a^b \int_b^\pi p_{\alpha,\beta} \gamma(\alpha) \times \gamma(\beta) \,d\beta\,d\alpha \\
& = -\int_a^b \int_0^a p_{\alpha,\beta} \gamma(\alpha) \times \gamma(\beta) \,d\beta \,d\alpha + \int_a^b \int_b^\pi p_{\alpha,\beta} \gamma(\alpha) \times \gamma(\beta) \,d\beta\,d\alpha \\
& = \int_a^b \gamma(\alpha) \times \int_b^{a+\pi} p_{\alpha,\beta}\gamma(\beta) \,d\beta\,d\alpha .
\end{align*}
We denote the inner integral in the last line by $\mu_{a,b}(\alpha)$ for $\alpha \in [a,b]$. It then follows that
\[
|\mu_{a,b}(\alpha) - \mu(\alpha)| \leq \left|\int_\alpha^{b} p_{\alpha,\beta}\gamma(\beta) \,d\beta\right| + \left|\int_{a + \pi}^{\alpha + \pi} p_{\alpha,\beta}\gamma(\beta) \,d\beta\right| \leq C|a-b| ,
\]
where $p_{\alpha,\beta} \leq C$ is a uniform bound. Thus,
\begin{align*}
\int_a^b \gamma(\alpha) \times \mu(\alpha) \,d\alpha & \geq \int_a^b \gamma(\alpha) \times \mu_{a,b}(\alpha) \,d\alpha - \int_a^b C|a-b|\,d\alpha \geq -C|a-b|^2 .
\end{align*}
If $\alpha \in (0,\pi)$ is a Lebesgue density point of $\gamma \times \mu$, setting $a = \alpha-\varepsilon$ and $b =  \alpha + \varepsilon$, and dividing both sides by $2\varepsilon$, and taking the limit $\varepsilon \to 0$ yields $\gamma(\alpha) \times \mu(\alpha) \geq 0$, by the Lebesgue differentiation theorem.
\end{proof}

By Lemma~\ref{lem:compactness}, there exists a maximizer $\gamma$ of $\omega_p$. Lemma~\ref{lem:gammaboundary} and Lemma~\ref{lem:gammaprop} imply that this maximizer satisfies $|\gamma(\alpha)| = 1$ and $\mu(\alpha) = |\mu(\alpha)|i\gamma(\alpha)$ for almost every $\alpha$. Under additional assumptions, we further obtain that this maximizer $\gamma$ admits a continuous representation.

\begin{lem}
	\label{lem:smoothness}
Assume that $(\alpha,\beta) \mapsto p_{\alpha,\beta}$ is additionally continuous on $\{(\alpha,\beta) : 0 < \alpha < \beta < \pi\}$. Then $\mu$ is continuous, and any maximizer $\gamma\in B^\infty_{\rm ap}(\R,\R^2)$ of $\omega_p$ has a continuous representation on $F = \{\alpha \in \R : \mu(\alpha) \neq 0\}$.
\end{lem}

\begin{proof}
The path $\mu$ is continuous as a consequence of the continuity of $p$ and the Lebesgue dominated convergence theorem. Hence, the set $F$ is open. For almost every $\alpha \in F$, we have
\begin{equation*}
-i\frac{\mu(\alpha)}{|\mu(\alpha)|} = \gamma(\alpha)
\end{equation*}
by Lemma~\ref{lem:gammaprop}. Since the left-hand side is continuous on $F$, it follows that $\gamma$ admits a continuous representation on $F$ as well.
\end{proof}

\subsection{Variation of paths and coefficients}

Let $r \in \left(0, \frac{\pi}{2}\right)$ be fixed, and suppose that $h \in \bS^2_+ \setminus \bS^1$ satisfies $\dist(h, \bS^1) = d \geq r$. Then $h = \arccos(\cos(d)\cos(\cdot - \tau))$ for some $\tau \in \R$. By Lemma~\ref{lem:geometriccoeff}, we have $p_{\alpha,\beta}(h) = p_\alpha(h) p_\beta(h)$, where
\begin{equation}
\label{eq:productagain}
p_\alpha(h) = \frac{\sin(d)}{\sin(h_\alpha)^2} = \frac{\sin(d)}{1 - \cos(d)^2\cos(\alpha - \tau)^2} \in \left[\sin(r), \sin(r)^{-1}\right] .
\end{equation}
By Lemma~\ref{lem:tangentsphere}, the coefficients $p_\alpha(h)$ satisfy $\int_0^{2\pi} p_\alpha(h) \,d\alpha = 2\pi$. As in Lemma~\ref{lem:prodtype}, we assume that $\nu_h : \R \to \R$ is the unique bi-Lipschitz function with $\nu_h(0) = 0$ and derivative $\nu_h'(\alpha) = p_\alpha(h)$. Since $\alpha \mapsto p_\alpha(h)$ is $\pi$-periodic, the function $\nu_h$ satisfies $\nu_h(\alpha + \pi) = \nu_h(\alpha) + \pi$. The path $\gamma(\alpha) = e^{i\nu_h(\alpha)}$ is a maximizer of $\omega_h$, and we have the identity
\begin{equation}
	\label{eq:areaformula2}
\sin\big(\nu_h(\beta) - \nu_h(\alpha)\big) = \gamma(\alpha) \times \gamma(\beta) = \frac{\sin(d)\sin(\beta - \alpha)}{\sin(h_\alpha)\sin(h_\beta)},
\end{equation}
as a consequence of Lemmas~\ref{lem:tangentsphere} and~\ref{lem:prodtype}.

We denote by $L_\pi^2(\R)$ the space of $\pi$-periodic functions $\eta \in L^2_{\mathrm{loc}}(\R)$, equipped with the inner product $\langle \eta_1,\eta_2\rangle \defl \pi\int_0^\pi \eta_1(t)\eta_2(t)\,dt$. The subspace of functions with zero mean is defined by
\[
L_{\pi,0}^2(\R) \defl \left\{\eta \in L_\pi^2(\R)\ : \ \int_0^\pi \eta = 0 \right\}.
\]
The norm induced by the inner product admits the following double integral representation: For $\eta \in L^2_{\pi,0}(\R)$,
\begin{align}
\nonumber
\int_0^\pi \int_\alpha^\pi (\eta(\beta)-\eta(\alpha))^2\,d\beta\,d\alpha & = \frac{1}{2}\int_0^\pi \int_0^\pi \eta(\beta)^2 - 2\eta(\alpha)\eta(\beta) + \eta(\alpha)^2\,d\beta\,d\alpha \\
\nonumber
& = \int_0^\pi \int_0^\pi \eta(\beta)^2\,d\beta\,d\alpha \\
\label{eq:normdefinition}
& = \int_0^\pi \frac{1}{\pi}\|\eta\|_2^2\,d\alpha = \|\eta\|_2^2.
\end{align}
Since $\eta$ is $\pi$-periodic, the following shifted identity holds for all $x \in \R$:
\begin{align}
 \nonumber
\int_0^\pi \int_0^\pi (\eta(\beta + x)-\eta(\alpha))^2\,d\beta\,d\alpha & = \int_0^\pi \int_x^{\pi + x} (\eta(\beta)-\eta(\alpha))^2\,d\beta\,d\alpha \\
 \nonumber
 & = \int_0^\pi \int_0^{\pi} (\eta(\beta)-\eta(\alpha))^2\,d\beta\,d\alpha \\
 \label{eq:normshift}
 & = 2\|\eta\|_2^2 .
\end{align}

Define the function $\Psi_h : E(\bS^1)\setminus \bS^1 \times L^2_{\pi}(\R) \to \R$ by
\begin{equation}
	\label{eq:psidef}
\Psi_h(f,\eta) \defl \int_0^\pi\int_\alpha^\pi p_{\alpha,\beta}(f) \sin(\Delta_{\alpha,\beta}^h + \Delta^\eta_{\alpha,\beta})\,d\beta\,d\alpha ,
\end{equation}
where $\Delta_{\alpha,\beta}^h \defl \nu_h(\beta) - \nu_h(\alpha)$ and $\Delta^\eta_{\alpha,\beta} \defl \eta(\beta)-\eta(\alpha)$.

\begin{lem}
	\label{lem:psismooth}
For fixed $h \in \bS^2_+\setminus \bS^1$, $f \in E(\bS^1)\setminus \bS^1$ and $\eta,v \in L^2_{\pi}(\R)$, the function $t \mapsto \Psi_h(f,\eta + tv)$ is of class $C^2$, with derivatives given by
\begin{align*}
\frac{d}{dt} \Psi_h(f,\eta + tv) & = \int_0^\pi\int_\alpha^\pi p_{\alpha,\beta}(f) \cos(\Delta_{\alpha,\beta}^h + \Delta_{\alpha,\beta}^{\eta+tv})\Delta_{\alpha,\beta}^v\,d\beta\,d\alpha , \\
\frac{d^2}{dt^2} \Psi_h(f,\eta + tv) & = -\int_0^\pi\int_\alpha^\pi p_{\alpha,\beta}(f) \sin(\Delta_{\alpha,\beta}^h + \Delta_{\alpha,\beta}^{\eta+tv})(\Delta_{\alpha,\beta}^v)^2\,d\beta\,d\alpha .
\end{align*}
\end{lem}

\begin{proof}
By Lemma~\ref{lem:pfunction}(4), the function $(\alpha,\beta) \mapsto p_{\alpha,\beta}(f)$ is uniformly bounded. Moreover, $t \mapsto \Delta^{\eta+tv}_{\alpha,\beta}$ is smooth and satisfies
\[
\left|\frac{1}{s}\left(\Delta^{\eta+(t+s)v}_{\alpha,\beta} - \Delta^{\eta+tv}_{\alpha,\beta} \right)\right| \leq |v(\beta)-v(\alpha)| .
\]
The function $(\alpha,\beta) \mapsto p_{\alpha,\beta}(f)|v(\beta)-v(\alpha)|$ is integrable, since $v \in L^2([0,\pi)) \subset L^1([0,\pi))$. Consequently, $t \mapsto\Psi_h(f,\eta + tv)$ is differentiable with derivative
\begin{align*}
\frac{d}{dt} \Psi_h(f,\eta + tv) & = \int_0^\pi\int_\alpha^\pi p_{\alpha,\beta}(f)\frac{d}{dt}\sin(\Delta_{\alpha,\beta}^h + \Delta_{\alpha,\beta}^{\eta + tv})\,d\beta\,d\alpha \\
 & = \int_0^\pi\int_\alpha^\pi p_{\alpha,\beta}(f) \cos(\Delta_{\alpha,\beta}^h + \Delta_{\alpha,\beta}^{\eta + tv})\Delta_{\alpha,\beta}^v \,d\beta\,d\alpha,
\end{align*}
as a consequence of the dominated convergence theorem. The second derivative is computed similarly, observing that the function $(\alpha,\beta) \mapsto p_{\alpha,\beta}(f)(v(\beta) - v(\alpha))^2$ is integrable. Finally, $t \mapsto \frac{d^2}{dt^2} \Psi_h(f,\eta + tv)$ is continuous again by the dominated convergence theorem.
\end{proof}

The following lemma establishes the concavity of $\eta \mapsto \Psi_h(f,\eta)$, provided certain conditions are met.

\begin{lem}
	\label{lem:strictconcavity}
Let $h \in \bS^2_+ \cap E_r(\bS^1)$ for some $r \in (0,\frac{\pi}{2})$. Then there exist $\varepsilon(r), c(r) > 0$ such that for any $\eta \in L^2_{\pi,0}(\R)$ and $f \in E(\bS^1)$ satisfying $\max\{\|f-h\|_\infty,\|\eta\|_\infty\} \leq \varepsilon$, the inequality
\[
\int_0^\pi\int_\alpha^\pi p_{\alpha,\beta}(f) \sin(\Delta_{\alpha,\beta}^h + \Delta_{\alpha,\beta}^\eta)(\Delta_{\alpha,\beta}^v)^2\,d\beta\,d\alpha \geq c\|v\|_2^2
\]
holds for all $v \in L^2_{\pi,0}(\R)$.
\end{lem}

\begin{proof}
Assume that $h_\alpha = \arccos(\cos(d)\cos(\alpha - \tau))$ for some $d \geq r$ and $\tau \in \R$. For $\delta \in (0,\frac{\pi}{2})$, define the sets
\begin{gather}
	\begin{aligned}
		\label{eq:setdef}
		D & \defl \left\{(\alpha,\beta)\in [0,\pi]^2\setminus \{(0,\pi)\} : \alpha < \beta\right\} , \\
		A_\delta & \defl \left\{(\alpha,\beta) \in D : \inf_{k \in \Z}|\beta - \alpha + \pi k| < \delta\right\} .
	\end{aligned}
\end{gather}
Note that $D \setminus A_\delta$ is compact and $\cL^2(A_\delta) = \pi \delta$. Before proceeding, we require an $L^2$-estimate of $v$ over $A_\delta$. With $(x+y)^2 \leq 2(x^2 + y^2)$ and \eqref{eq:normshift}, we obtain
\begin{align}
 \nonumber
\int_{A_\delta} (\Delta_{\alpha,\beta}^v)^2\,d\beta\,d\alpha & = \int_{0}^{\pi} \int_\alpha^{\alpha + \delta} (v(\beta) - v(\alpha))^2 \,d\beta\,d\alpha \\
 \nonumber
 & = \int_{0}^{\pi} \int_0^{\delta} (v(\alpha + x) - v(\alpha))^2 \,dx\,d\alpha \\
 \nonumber
 & = \frac{1}{\pi}\int_0^\pi \int_{0}^{\pi} \int_0^{\delta} ((v(\alpha + x) - v(z)) + (v(z) - v(\alpha)))^2 \,dx\,d\alpha \, dz \\
 \nonumber
 & \leq \frac{2}{\pi} \int_0^{\delta} \int_0^\pi \int_{0}^{\pi} (v(\alpha + x) - v(z))^2 + (v(\alpha) - v(z))^2 \,d\alpha \, dz \,dx \\
 \label{eq:diagonalestimate}
 & = \frac{8}{\pi} \int_0^{\delta} \|v\|_2^2 \,dx = \frac{8}{\pi}\delta \|v\|_2^2 .
\end{align}
In the first line, the integral over the triangular region defined by $\alpha \geq 0$ and $\alpha + \pi - \delta \leq \beta \leq \pi$ is replaced by the integral over the triangular region defined by $\beta \geq \pi$ and $\beta - \delta \leq \alpha \leq \pi$, via the isometry $(\alpha,\beta) \mapsto (\beta, \alpha + \pi)$. This change of variables preserves $(v(\beta) - v(\alpha))^2$ because $v$ is $\pi$-periodic.

Let $m(r) \defl \sin(r)$ and $M(r) \defl \sin(\frac{r}{2})^{-2}$. The following statements are true:
\begin{enumerate}
	\item $p_{\alpha,\beta}(h) \geq m^2 > 0$ for all $(\alpha,\beta) \in D$ due to \eqref{eq:productagain}.
	\item $(\alpha, \beta, f) \mapsto p_{\alpha,\beta}(f)$ is continuous on $D \times \B(h, \frac{r}{2})$ and takes values in the interval $[0, M]$, as established in Lemma~\ref{lem:pfunction}.
	\item $\nu_h : [0,\pi] \to [0,\pi]$ is increasing and bi-Lipschitz with $m|\beta - \alpha| \leq |\nu_h(\beta) - \nu_h(\alpha)|$, since  $\nu_h'(\alpha) = p_\alpha(h) \geq m$.
\end{enumerate}
Accordingly, for any $\delta \in (0,\frac{\pi}{2})$, there exists $\varepsilon(\delta,r) \in (0,\min\{\frac{r}{2},\delta\})$ such that for all $(\alpha,\beta) \in D \setminus A_\delta$ and all $f,\eta$ with $\max\{\|f-h\|_\infty,\|\eta\|_\infty\} \leq \varepsilon$, we have:
\begin{enumerate}[label=(\alph*)]
	\item $p_{\alpha,\beta}(f) \geq p_{\alpha,\beta}(h) - \delta$.
	\item $\sin(\Delta_{\alpha,\beta}^h) \geq 2\varepsilon$.
	\item $\sin(\Delta_{\alpha,\beta}^h + \Delta^\eta_{\alpha,\beta}) \geq \sin(\Delta_{\alpha,\beta}^h) - 2\varepsilon \geq 0$.
\end{enumerate}
Property (c) follows from (b) together with the uniform bound on $\eta$. As a consequence of (a), (b), (c), \eqref{eq:normdefinition} and \eqref{eq:diagonalestimate}, we obtain the following estimate:
\begin{align*}
& \int_D p_{\alpha,\beta}(f) \sin(\Delta_{\alpha,\beta}^h + \Delta_{\alpha,\beta}^\eta)(\Delta_{\alpha,\beta}^v)^2 \\
& \geq \int_{D \setminus A_\delta} (p_{\alpha,\beta}(h) - \delta) (\sin(\Delta_{\alpha,\beta}^h) - 2\varepsilon) (\Delta_{\alpha,\beta}^v)^2  - \int_{A_\delta} M (\Delta_{\alpha,\beta}^v)^2 \\
 & \geq \int_{D \setminus A_\delta} p_{\alpha,\beta}(h)\sin(\Delta_{\alpha,\beta}^h)(\Delta_{\alpha,\beta}^v)^2 - \delta(1 + 2 M) \int_{D \setminus A_\delta}(\Delta_{\alpha,\beta}^v)^2 - M \int_{A_\delta} (\Delta_{\alpha,\beta}^v)^2 \\
 & \geq \int_{D} p_{\alpha,\beta}(h)\sin(\Delta_{\alpha,\beta}^h)(\Delta_{\alpha,\beta}^v)^2 - \delta (1 + 2M) \int_{D}(\Delta_{\alpha,\beta}^v)^2 - (M + 1) \int_{A_\delta} (\Delta_{\alpha,\beta}^v)^2 \\
 & \geq \int_{D} p_{\alpha,\beta}(h)\sin(\Delta_{\alpha,\beta}^h)(\Delta_{\alpha,\beta}^v)^2 - C\delta \|v\|_2^2
\end{align*}
for some $C(r) > 0$. If $a \in (0,\frac{\pi}{2})$ and $(\alpha,\beta) \in D \setminus A_a$, then
\[
\sin(\Delta_{\alpha,\beta}^h) = \sin(\nu_h(\beta) - \nu_h(\alpha)) \geq \tfrac{2}{\pi} a m
\]
by (3). Thus with (1), \eqref{eq:normdefinition} and \eqref{eq:diagonalestimate}, we have
\begin{align*}
\int_{D} p_{\alpha,\beta}(h)\sin(\Delta_{\alpha,\beta}^h)(\Delta_{\alpha,\beta}^v)^2 & \geq m^2\int_{D\setminus A_a} \sin(\Delta_{\alpha,\beta}^h)(\Delta_{\alpha,\beta}^v)^2 \\
 & \geq \tfrac{2}{\pi} a m^3 \int_{D \setminus A_a} (\Delta_{\alpha,\beta}^v)^2 \\
 & \geq \tfrac{2}{\pi}a m^3(1 - \tfrac{8}{\pi}a)\|v\|_2^2 .
\end{align*}
By choosing $a = \frac{\pi}{16}$ and $\delta(r) > 0$ sufficiently small so that $C\delta \leq \frac{1}{32}m^3$, we obtain the estimate
\[
\int_D p_{\alpha,\beta}(f) \sin(\Delta_{\alpha,\beta}^h + \Delta_{\alpha,\beta}^\eta)(\Delta_{\alpha,\beta}^v)^2 \geq \tfrac{1}{16} m^3\|v\|_2^2,
\]
for all $\eta,v \in L^2_{\pi,0}(\R)$ and $f \in E(\bS^1)$ whenever $\max\{\|f - h\|_\infty,\|\eta\|_\infty\} \leq \varepsilon(\delta,r)$.
\end{proof}

For $f \in E(\bS^1)\setminus \bS^1$ and $\gamma \in B^\infty_{\rm ap}(\R,\R^2)$, the path $\mu_{f,\gamma} \in L^\infty_{\rm ap}(\R,\R^2)$ is defined by
\begin{equation}
\label{eq:mudef}
\mu_{f,\gamma}(\alpha) \defl \int_\alpha^{\alpha + \pi}q_{\alpha,\beta}(f)\gamma(\beta)\,d\beta ,
\end{equation}
where $q_{\alpha,\beta}(f) \defl \sin(f_\alpha)^{2} p_{\alpha,\beta}(f)$. Next, we show that for fixed $\gamma \in B^\infty_{\rm ap}(\R,\R^2)$, the map $f \mapsto \mu_{f,\gamma}$ (as well as other related maps) is H\"older continuous. Recall that $\|\omega_f\|_{\rm ir}$ is characterized by \eqref{eq:curveaction}.

\begin{lem}
	\label{lem:hoeldercont}
For any $\xi \in (0,1)$ and $r \in \left(0, \frac{\pi}{2}\right)$, there exists $H(\xi, r) > 0$ such that
\begin{align*}
	H\|f-g\|_\infty^\xi \geq & \max\biggl\{ \int_0^\pi\int_\alpha^\pi |p_{\alpha,\beta}(f) - p_{\alpha,\beta}(g)| \, d\beta \, d\alpha , \\
	& |\omega_f(\gamma) - \omega_g(\gamma)| , \ |\|\omega_f\|_{\rm ir} - \|\omega_g\|_{\rm ir}| , \ \|\mu_{f,\gamma} - \mu_{g,\gamma}\|_\infty \biggr\} ,
\end{align*}
for all $\gamma \in B^\infty_{\rm ap}(\R,\R^2)$ and all $f,g \in E_r(\bS^1)$.
\end{lem}

\begin{proof}
Throughout the proof, we fix $f, g \in E_r(\bS^1)$ and define $\phi^t \defl (1 - t)f + t g$ for $t \in [0,1]$.
The sets $D$ and $A_\delta$ for $\delta \in \left(0, \frac{\pi}{2}\right)$ are defined as in \eqref{eq:setdef}.
	
It follows from Lemma~\ref{lem:pfunction}(4), Lemma~\ref{lem:derivative}, and Lemma~\ref{lem:unifboundder} that the function $p_{\alpha,\beta}(\phi^t)$ is smooth in $t$, and that both $p_{\alpha,\beta}(\phi^t)$ and its first derivative with respect to $t$ are uniformly bounded by a constant $M(r) > 0$. Moreover, by Lemma~\ref{lem:unifboundder}, we may assume that $M$ is chosen sufficiently large so that
\begin{equation*}
|\partial_t p_{\alpha,\beta}(\phi^t)| \leq \frac{M}{\sin(\beta-\alpha)} \|f - g\|_\infty
\end{equation*}
holds for all $t \in [0,1]$ and all $(\alpha,\beta) \in D$. Since all terms to be estimated are uniformly bounded for functions in $E_r(\bS^1)$, we may, without loss of generality, assume that $0 < \delta \defl \|f - g\|_\infty^\xi < \frac{\pi}{2}$. Since $\cL^2(A_\delta) = \pi \delta$ and $\frac{\pi}{2}\sin(\beta - \alpha) \geq \delta$ for $(\alpha,\beta) \in D\setminus A_\delta$, as in the proof of Lemma~\ref{lem:strictconcavity}, it follows that
\begin{align*}
|\omega_f(\gamma) - \omega_g(\gamma)| & \leq \int_D |p_{\alpha,\beta}(f)-p_{\alpha,\beta}(g)| \leq \int_D \sup_{t \in [0,1]} |\partial_t p_{\alpha,\beta}(\phi^t)| \\
 & \leq \int_{A_\delta} M + \int_{D\setminus A_\delta} \frac{M}{\sin(\beta-\alpha)} \|f - g\|_\infty \\
 & \leq M\pi \delta + M\|f - g\|_\infty^{\xi}  \int_{D\setminus A_\delta} \frac{1}{\sin(\beta-\alpha)}\|f - g\|_\infty^{1-\xi} \\
 & \leq \|f-g\|_\infty^{\xi}\left(M\pi + M\left(\frac{\pi}{2}\right)^{\frac{1}{\xi} - 1} \int_D \sin(\beta - \alpha)^{\frac{1}{\xi} - 2}\right) .
\end{align*}
In the last line, we used the inequality
\[
\|f - g\|_\infty^{1-\xi} = \delta^\frac{1-\xi}{\xi} \leq \left(\frac{\pi}{2}\right)^{\frac{1}{\xi} - 1}\sin(\beta - \alpha)^{\frac{1}{\xi} - 1}.
\]
The remaining integral needs to be bounded. Without loss of generality, we may assume $\xi > \frac{1}{2}$, since the bound is trivial for $\xi \leq \frac{1}{2}$. The boundedness then follows from the condition $\frac{1}{\xi} - 2 > -1$, because
\begin{align*}
\int_D \sin(\beta - \alpha)^{\frac{1}{\xi} - 2} & \leq \int_0^\pi \int_\alpha^{\alpha + \pi} |\sin(\beta - \alpha)|^{\frac{1}{\xi} - 2}\,d\beta\,d\alpha \\
& = 2\int_0^\pi \int_0^{\frac{\pi}{2}} \sin(t)^{\frac{1}{\xi} - 2}\,dt\,d\alpha \\
& \leq 2\pi \int_0^{\frac{\pi}{2}} (\tfrac{2}{\pi}t)^{\frac{1}{\xi} - 2}\,dt < \infty ,
\end{align*}
where we have used the substitution $t = \beta - \alpha$ and the inequality $\sin t \geq \frac{2}{\pi} t$ for $t \in [0, \frac{\pi}{2}]$. This establishes the estimates for the first two terms in the statement with some $H(\xi,r) > 0$.

For fixed $f$ and $g$ as above, Lemma~\ref{lem:compactness} guarantees the existence of $\gamma \in B^\infty_{\rm ap}(\R,\R^2)$ such that $\|\omega_f\|_{\rm ir} = \omega_f(\gamma)$. Hence,
\[
\|\omega_f\|_{\rm ir} = \omega_f(\gamma) \leq \omega_g(\gamma) + H\|f-g\|_\infty^\xi \leq \|\omega_g\|_{\rm ir} + H\|f-g\|_\infty^\xi .
\]
Exchanging the roles of $f$ and $g$ yields the third estimate.

Next, we verify that the paths $\mu_{f,\gamma}$, defined in \eqref{eq:mudef}, depend continuously on $f$. Inheriting the smoothness from $p_{\alpha,\beta}(\phi^t)$, the function $q_{\alpha,\beta}(\phi^t)$ is also smooth in $t$. Moreover, there exists a constant $C(r) > 0$ such that
\[
\max\{|q_{\alpha,\beta}(\phi^t)|, |\partial_t q_{\alpha,\beta}(\phi^t)|\} \leq C
\]
and, provided $\alpha \neq \beta \text{ mod } \pi$,
\[
|\partial_t q_{\alpha,\beta}(\phi^t)| \leq \frac{C}{|\sin(\beta-\alpha)|}\|g - f\|_\infty.
\]
For $\alpha \in \R$, define $A_{\alpha,\delta} \defl (\alpha,\alpha + \delta) \cup (\alpha + \pi - \delta,\pi)$. Since $\cL(A_{\alpha,\delta}) = 2 \delta$ and $\frac{\pi}{2}\sin(\beta - \alpha) \geq \delta$ for $\beta \in (0,\pi)\setminus A_{\alpha,c}$, we obtain an estimate similar to the one above:
\begin{align*}
|\mu_{f,\gamma}(\alpha) - \mu_{g,\gamma}(\alpha)| & \leq \int_\alpha^{\alpha + \pi} \sup_{t \in [0,1]} |\partial_t q_{\alpha,\beta}(\phi^t)| \\
& \leq \int_{A_{\alpha,\delta}} C + \int_{(\alpha,\alpha + \pi) \setminus A_{\alpha,\delta}} \frac{C} {\sin(\beta-\alpha)} \|f - g\|_\infty \\
& \leq 2C \delta + C\|f - g\|_\infty^{\xi}  \int_{(\alpha,\alpha + \pi) \setminus A_{\alpha,\delta}} \frac{1}{\sin(\beta-\alpha)}\|f - g\|_\infty^{1-\xi} \\
& \leq \|f-g\|_\infty^{\xi}\left(2C + C\left(\frac{\pi}{2}\right)^{\frac{1}{\xi} - 1} \int_\alpha^{\alpha + \pi} \sin(\beta - \alpha)^{\frac{1}{\xi} - 2}\right) .
\end{align*}
The remaining integral is bounded as before, which completes the proof.
\end{proof}

For $h \in \bS^2_+$ with $d \defl \dist(h,\bS^1) \in (0,\frac{\pi}{2}]$, recall that
\[
\nu_h'(\alpha) = p_\alpha(h) = \frac{\sin(d)}{\sin(h_\alpha)^2} .
\]
To each $\pi$-periodic and measurable function $\eta : \R \to \R$, we associate the paths $\gamma_\eta, \sigma_{\eta} \in L^\infty_{\rm ap}(\R,\R^2)$ defined by
\begin{equation}
	\label{eq:sigmadef}
	\gamma_\eta \defl e^{i(\nu_h(s) + \eta(s))} \quad \text{and} \quad
	\sigma_{\eta}(\alpha) \defl \frac{1}{2}\int_\alpha^{\alpha + \pi} p_\beta(h) \gamma_\eta(\beta)\,d\beta,
\end{equation}
as in Lemma~\ref{lem:prodtype}. Note the following properties:
\begin{itemize}
	\item $\mu_{h,\gamma_\eta} = 2\sin(d)\sigma_\eta$.
	\item $\sigma_{\eta}|_{[0,2\pi]}$ is a closed Lipschitz path of length $2\pi$.
	\item $\Psi_h(f,\eta) = \omega_f(\gamma_\eta)$ for $(f,\gamma) \in E(\bS^1)\setminus \bS^1 \times L^2_{\pi}(\R)$.
\end{itemize}
The second point follows from the fact that $|\sigma_{\eta}'(\alpha)| = |-p_\alpha(h) \gamma_\eta(\alpha)| = p_\alpha(h)$ for almost every $\alpha$, and $\int_0^{2\pi}p_\alpha(h)\,d\alpha = 2\pi$, as established in Lemma~\ref{lem:tangentsphere}. The third point is immediate from the definition of $\gamma_\eta$ and $\Psi$ in \eqref{eq:psidef}.

For the special case $\gamma_0 = e^{i\nu_h}$, we have $\omega_{h}(\gamma_0) = \pi$ by \eqref{eq:spherecalibration}, and
\begin{align}
	\nonumber
	\mu_{h,\gamma_0}(\alpha) & = \int_\alpha^{\alpha + \pi}\frac{\sin(d)^2}{\sin(h_\beta)^2}\gamma_0(\beta)\,d\beta = -i\sin(d)\int_\alpha^{\alpha + \pi}i\nu_h'(\beta)e^{i\nu_h(\beta)}\,d\beta  \\
	\label{eq:mudef0}
	& = -i\sin(d)(e^{i\nu_h(\alpha + \pi)} - e^{i\nu_h(\alpha)}) = 2i\sin(d)e^{i\nu_h(\alpha)}.
\end{align}
The dependence of $\gamma_\eta$ and $\sigma_{\eta}$ on $h \in \bS^2_+\setminus \bS^1$ will be clear from the context.

By Lemma~\ref{lem:prodtype}, the action $\Psi(h, \eta)$ equals the signed area spanned by $\sigma_\eta$, which is maximized by $\sigma_0$. The stability of the planar isoperimetric inequality allows one to bound the defect of $\sigma_\eta$ from a round circle in terms of $|\Psi(h, \eta) - \Psi(h, 0)|$. This yields estimates on $\mu_{f,\gamma_\eta}$.

\begin{lem}
	\label{lem:stability}
Assume that $\xi \in (0,1)$ and $h \in \bS^2_+ \cap E_r(\bS^1)$ for some $r \in (0,\frac{\pi}{2})$. Then there exists $C(\xi,r) > 0$ with the property that for all $f \in E_\frac{r}{2}(\bS^1)$ and all $\eta \in L_\pi^2(\R)$, there exists $c(h,\eta) \in (-\pi,\pi]$ such that
\[
\|\mu_{f,\gamma_{\eta+c}} - \mu_{h,\gamma_{0}}\|_\infty \leq C\left(\|f - h\|_\infty^\frac{\xi}{2} + |\Psi_h(h,0) - \Psi_h(f,\eta)|^\frac{1}{2}\right) .
\]
\end{lem}

\begin{proof}
Assume that $h_\alpha = \arccos(\cos(d)\cos(\alpha - \tau))$ for parameters $d \in [r,\frac{\pi}{2}]$ and $\tau \in \R$, and let $f \in E_\frac{r}{2}(\bS^1)$. In Lemma~\ref{lem:prodtype} we observed that the signed area spanned by $\sigma_{\eta}|_{[0,2\pi]}$, defined in \eqref{eq:sigmadef}, is given by
\begin{equation}
	\label{eq:psiineq2}
A_\eta \defl \frac{1}{2}\int_0^{2\pi} \sigma_{\eta}(\alpha) \times \sigma_{\eta}'(\alpha)\,d\alpha = \omega_h(\gamma_\eta) = \Psi_h(h,\eta) .
\end{equation}
The length of $\sigma_{\eta}|_{[0,2\pi]}$ is $\int_0^{2\pi} p_\alpha(h)\,d\alpha = 2\pi$. Hence, by the plane isoperimetric inequality, the signed area satisfies $|A_\eta| \leq \pi$. The inverse function $g \defl \nu_h^{-1} : \R \to \R$ is strictly increasing and satisfies the periodicity condition $g(t + \pi) = g(t) + \pi$, just as $\eta_h$ does. The path $\tilde\sigma_{\eta}(t) \defl \sigma_{\eta}(g(t))$ is parametrized by arc length, since
\[
1 = \nu_h'(g(t))g'(t) = p_{g(t)}(h)g'(t) = |\sigma_{\eta}'(g(t))g'(t)| = |\tilde\sigma_{\eta}'(t)|
\]
for almost every $t$. Thus, the path $\tilde \sigma_\eta$ can be written as
\begin{align*}
\tilde\sigma_{\eta}(t) & = \frac{1}{2}\int_{g(t)}^{g(t + \pi)} p_\beta(h) \gamma_{\eta}(\beta)\,d\beta \\
& = \frac{1}{2}\int_{t}^{t + \pi} p_{g(s)}(h) \gamma_{\eta}(g(s))g'(s)\,ds \\
& = \frac{1}{2}\int_{t}^{t + \pi} e^{i(s + \eta(g(s)))} \,ds,
\end{align*}
with derivative
\[
\tilde\sigma_{\eta}'(t) = -e^{i(t + \eta(g(t)))}
\]
for almost every $t$.

The stability result of Fuglede \cite[\S1]{Fu} is stated in terms of the dissimilarity function
\[
w(t) \defl c_0(\eta) + c_1(\eta)e^{it} - \tilde\sigma_{\eta}(t) ,
\]
where
\[
c_n(\eta) \defl \frac{1}{2\pi}\int_0^{2\pi}\tilde\sigma_{\eta}(t)e^{-int}\,dt
\]
denotes the $n$th Fourier coefficient of $\tilde\sigma_{\eta}$. The coefficients of interest satisfy $c_0(\eta) = 0$ by antipodal symmetry of $\tilde \sigma_\eta$, and $|c_1(\eta)| \leq 1$, since
\[
1 = \frac{1}{2\pi}\int_0^{2\pi}|\tilde\sigma_{\eta}'(t)|^2 \,dt = \sum_{n \in \Z} n^2|c_n(\eta)|^2 \geq |c_1(\eta)|^2 .
\]
In the special case $\eta = 0$, we have $c_1(0) = i$, since
\[
\tilde\sigma_{0}(t) = \frac{1}{2}\int_{t}^{t + \pi} e^{is} \,ds = \frac{1}{2}(-ie^{i(t+\pi)} - (-ie^{it})) = ie^{it} .
\]
A translation by a constant affects these Fourier coefficients in the following way: For any $c \in \R$, we have
\[
\tilde \sigma_{\eta + c}(t) = \frac{1}{2}\int_{t}^{t + \pi} e^{i(s + \eta(g(s)))}e^{ic}\,ds = e^{ic}\tilde \sigma(t),
\]
and consequently
\begin{equation}
	\label{eq:fourier_translation}
c_n(\eta + c) = e^{ic}c_n(\eta).
\end{equation}
Choose $c = c(h,\eta) \in (-\pi,\pi]$ such that $c_1(\eta + c) = bi$ for some $b \geq 0$. For the remainder of the proof, we replace $\eta$ by $\eta+c$. This is justified, as $\Psi_h(f,\eta) = \Psi_h(f,\eta+c)$.

Each $\tilde\sigma_n$ is parametrized by arc-length and has total length $2\pi$. Under this assumption, two estimates from \cite[\S1]{F} apply:
\begin{equation}
\label{eq:fuglede}
\int_0^{2\pi}|w|^2 + |w'|^2 \leq 5(\pi - A_\eta) \quad \text{and} \quad \|w\|_\infty^2 \leq 5\pi(\pi - A_\eta) ,
\end{equation}
where the second estimate follows as a consequence of the first. Since $\sigma_{0} = i\gamma_{0} = ie^{i\nu_h}$, it follows that
\begin{align}
	\nonumber
5\pi(\pi - A_\eta) & \geq \sup_t|\tilde\sigma_{\eta}(t) - c_1(\eta)e^{it}|^2 = \sup_{\alpha}|\tilde\sigma_{\eta}(\nu_h(\alpha)) - ibe^{i\nu_h(\alpha)}|^2 \\
	\label{eq:fuglede2}
 & = \|\sigma_{\eta} - b\sigma_{0}\|_\infty^2 .
\end{align}
Since
\[
|w'(t)| \geq |\tilde\sigma_{\eta}'(t)| - |c_1(\eta)ie^{it}| = 1 - b \geq 0 ,
\]
it follows from the first inequality in \eqref{eq:fuglede}, by integrating $|w'|^2$, that
\begin{equation}
	\label{eq:fuglede3}
0 \leq 1 - b \leq (\pi - A_\eta)^\frac{1}{2} .
\end{equation}
Using $|\sigma_{0}| \equiv 1$ and combining \eqref{eq:fuglede2} and \eqref{eq:fuglede3}, we estimate
\begin{align}
\nonumber
\|\sigma_{\eta} - \sigma_{0}\|_\infty & \leq \|\sigma_{\eta} - b \sigma_{0}\|_\infty + \| \sigma_{0} - b \sigma_{0}\|_\infty \\
\nonumber
 & \leq (5\pi)^\frac{1}{2}(\pi - A_\eta)^\frac{1}{2} + (1 - b) \\
 \nonumber
 & \leq (5\pi)^\frac{1}{2}(\pi - A_\eta)^\frac{1}{2} + (\pi - A_\eta)^\frac{1}{2}\\
 \label{eq:muinequality1}
 & \leq C_1(\pi - A_\eta)^\frac{1}{2},
\end{align}
where $C_1 \defl (5\pi)^\frac{1}{2} + 1$.

As a consequence of the identity $\mu_{h,\gamma_{\eta}} = 2\sin(d)\sigma_{\eta}$, together with \eqref{eq:mudef0}, \eqref{eq:psiineq2}, \eqref{eq:muinequality1}, and Lemma~\ref{lem:hoeldercont}, there exists $H(\xi,\frac{r}{2}) > 0$ such that
\begin{align*}
\|\mu_{f,\gamma_{\eta}} - \mu_{h,\gamma_{0}}\|_\infty & \leq \|\mu_{f,\gamma_{\eta}} - \mu_{h,\gamma_{\eta}}\|_\infty + \|\mu_{h,\gamma_{\eta}} - \mu_{h,\gamma_0}\|_\infty \\
 & \leq H\|f - h\|_\infty^\xi + 2\sin(d)\|\sigma_{\eta} - \sigma_{0}\|_\infty \\
 & \leq H\|f - h\|_\infty^\xi + 2\sin(d)C_1(\pi - \Psi_h(h,\eta))^\frac{1}{2} .
\end{align*}
Furthermore, applying Lemma~\ref{lem:hoeldercont} once more and using that $\omega_f(\gamma_{\eta}) = \Psi_h(f,\eta)$, we obtain
\begin{align*}
|\pi - \Psi_h(h,\eta)|^\frac{1}{2} & \leq |\pi - \Psi_h(f,\eta)|^\frac{1}{2} + |\Psi_h(f,\eta) - \Psi_h(h,\eta)|^\frac{1}{2} \\
 & \leq |\Psi_h(h,0) - \Psi_h(f,\eta)|^\frac{1}{2} + H^\frac{1}{2}\|f - h\|_\infty^\frac{\xi}{2} .
\end{align*}
This proves the lemma.
\end{proof}

Owing to the global estimate derived from the stability of the planar isoperimetric inequality, a maximizer $\eta_f$ of $\Psi_h(f, \cdot)$ can be found locally near $0$, provided that $\|f - h\|_\infty$ is sufficiently small. For this result, we temporarily assume that $f \in E^+(\bS^1)$.

\begin{lem}
	\label{lem:optimallinfinity}
Let $r \in (0,\frac{\pi}{2})$ and $\xi \in (0,1)$. Then there exist $\varepsilon \in \left(0,\frac{r}{2}\right)$ and $C > 0$ such that for all $h \in \bS^2_+ \cap E_r(\bS^1)$ and all $f \in \B(h,\varepsilon) \cap E^+(\bS^1)$, there exists $\eta_f \in L^2_{\pi,0}(\R)$ with the following properties:
\begin{enumerate}
	\item $\Psi_h(f,\eta_f) = \sup_{\eta \in L^2_{\pi,0}(\R)} \Psi_h(f,\eta) = \sup_{\gamma \in B^\infty_{\rm ap}(\R,\R^2)} \omega_f(\gamma)$.
	\item $\eta_f$ is continuous.
	\item $\|\eta_f\|_\infty \leq C\|f-h\|_\infty^\frac{\xi}{2}$.
	\item $\|\mu_{f,\gamma_{\eta_f}}\|_\infty \geq \sin(r)$ and $\gamma_{\eta_f} = -i\frac{\mu_{f,\gamma_{\eta_f}}}{|\mu_{f,\gamma_{\eta_f}}|}$.
\end{enumerate}
\end{lem}

\begin{proof}
By the definition of $E^+(\bS^1)$, the coefficients $p_{\alpha,\beta}(f)$ are strictly positive for almost every pair $(\alpha, \beta)$. According to Lemma~\ref{lem:compactness} and Lemma~\ref{lem:gammaboundary}, there exists a maximizer $\gamma \in B^\infty_{\mathrm{ap}}(\R,\R^2)$ of $\omega_f$, which satisfies $|\gamma| = 1$ almost everywhere. Hence there exists a measurable, $\pi$-periodic function $\eta : \R \to \R$ such that $\gamma_{\eta} = e^{i(\nu_h + \eta)}$ is a maximizer of $\omega_f$. By adding appropriate integer multiples of $2\pi$ to $\eta$ pointwise, we may assume that $\eta$ takes values in $(-\pi, \pi]$, and hence lies in $L^2_\pi(\R)$. Let $c = c(h,\eta) \in (-\pi, \pi]$ be the constant provided by Lemma~\ref{lem:stability}. The shifted function $\eta + c$ takes values in $(-2\pi, 2\pi]$. We define $\eta_f \in L^2_\pi(\R)$ to be the further adjustment of $\eta + c$ taking values in $(-\pi, \pi]$. Observe that the corresponding path satisfies $\gamma_{\eta_f} = \gamma_{\eta + c} = e^{ic}\gamma_\eta$, so that $\gamma_{\eta_f}$ remains a maximizer of $\omega_f$ by rotation invariance. In a subsequent step, we will modify $\eta_f$ again to ensure that it lies in $L^2_{\pi,0}(\R)$.

Set $\varepsilon_0 \defl \frac{r}{2} < \dist(h,\bS^1)$. According to Lemma~\ref{lem:hoeldercont} and Lemma~\ref{lem:stability}, there exist constants $H(\xi,\frac{r}{2}), C(\xi,r) > 0$ such that for all $f \in \B(h,\varepsilon_0) \cap E^+(\bS^2)$, the following estimates hold:
\begin{equation*}
|\Psi_h(f,\eta_f) - \Psi_h(h,0)| = |\|\omega_f\|_{\rm ir} - \|\omega_h\|_{\rm ir}| \leq H\|f-h\|_\infty^\xi,
\end{equation*}
\begin{equation*}
\left\|\mu_{f,\gamma_{\eta_f}} - \mu_{h,\gamma_0}\right\|_\infty \leq C\left(\|f - h\|_\infty^\frac{\xi}{2} + |\Psi_h(f,\eta_f) - \Psi_h(h,0)|^\frac{1}{2}\right).
\end{equation*}
Abbreviate $\varepsilon \defl \|f - h\|_\infty$ with $\varepsilon \leq \varepsilon_0$. It follows that
\begin{equation}
\label{eq:psieq4}
\left\|\mu_{f,\gamma_{\eta_f}} - \mu_{h,\gamma_0}\right\|_\infty \leq C\left(\varepsilon^\frac{\xi}{2} + H^\frac{1}{2} \varepsilon^\frac{\xi}{2}\right) = C_1\varepsilon^\frac{\xi}{2}
\end{equation}
for some $C_1(\xi,r) > 0$. The path $\mu_{f,\gamma_{\eta_f}}$ is continuous by Lemma~\ref{lem:smoothness}. Indeed, the path $\mu \in L^\infty_{\mathrm{ap}}(\R,\R^2)$ defined there is given by
\[
\mu(\alpha) = \int_\alpha^{\alpha+\pi} p_{\alpha,\beta}(f)\gamma_{\eta_f}(\beta)\,d\beta,
\]
and thus satisfies $\mu(\alpha) = \sin(f_\alpha)^2\mu_{f,\gamma_{\eta_f}}(\alpha)$. Since $\mu_{h,\gamma_0} = 2i \sin(d) e^{i \nu_h}$ for some $d \geq r$ by \eqref{eq:mudef0}, we may choose $\varepsilon_1(\xi,r) \in (0, \varepsilon_0]$ sufficiently small so that $|\mu_{f,\gamma_{\eta_f}}(\alpha)| \geq \sin(r)$ for all $\alpha$ whenever $\varepsilon \leq \varepsilon_1$. In this setting, we also have $\mu(\alpha) \neq 0$ for all $\alpha$. Moreover, since $\gamma_{\eta_f}$ is a maximizer of $\omega_f$, it is continuous and satisfies $i\gamma_{\eta_f}|\mu| = \mu$ by Lemma~\ref{lem:smoothness}. Equivalently, this can be expressed as $i\gamma_{\eta_f} |\mu_{f,\gamma_{\eta_f}}| = \mu_{f,\gamma_{\eta_f}}$. Because the map $\rho : \C \setminus \oB(0,\sin(r)) \to \bS^1$, defined by $\rho(z) \defl -i\frac{z}{|z|}$, is Lipschitz, it follows from \eqref{eq:psieq4} that
\begin{align*}
\|e^{i\eta_f} - 1\|_\infty = \|e^{i(\nu_h + \eta_f)} - e^{i\nu_h}\|_\infty = \|\gamma_{\eta_f} - \gamma_{\eta_0}\|_\infty & \leq C_2\varepsilon^\frac{\xi}{2}
\end{align*}
for some $C_2(\xi,r) > 0$, provided $\varepsilon \leq \varepsilon_1$. The paths $e^{-i\nu_h}$ and $\gamma_{\eta_f} = e^{i(\nu_h + \eta_f)}$ are continuous; hence, their pointwise product $e^{i\eta_f}$ is also continuous.

Assume that $\varepsilon_2(\xi,r) \in (0,\varepsilon_1]$ is sufficiently small so that $C_2\varepsilon_2^{\xi/2} \leq \sqrt{2}$. Then, for all $\varepsilon \leq \varepsilon_2$, we have $\|e^{i\eta_f} - 1\|_\infty \leq C_2\varepsilon^{\xi/2} \leq \sqrt{2}$, which implies $\| \eta_f\|_\infty \leq \frac{\pi}{2}$. Recall that we have ensured that $\eta_f$ takes values in $(-\pi,\pi]$ a priori. Therefore
\[
\|\eta_f\|_\infty \leq \tfrac{\pi}{2}\|\sin(\eta_f)\|_\infty \leq \tfrac{\pi}{2} \left\|e^{i\eta_f} - 1 \right\|_\infty \leq \tfrac{\pi}{2}C_2\varepsilon^\frac{\xi}{2} .
\]
The function $\eta_f$ satisfies all properties stated in the lemma, except for the normalization condition $\int_0^\pi \eta_f = 0$. This can be corrected by replacing $\eta_f$ with its translate $\eta_f - \frac{1}{\pi} \int_0^\pi\eta_f$. The resulting function lies in $L^2_{\pi,0}(\R)$ with upper bound $\|\eta_f\|_\infty \leq \pi C_2\varepsilon^\frac{\xi}{2}$, and retains all other required properties. This completes the proof of the lemma.
\end{proof}

The restriction to $E^+(\bS^1)$ poses no issue, since $E^+(\bS^1)$ is dense in $E(\bS^1)$, as established in Lemma~\ref{lem:truncated}, and the map $f \mapsto \|\omega_f\|_{\mathrm{ir}}$ is continuous by Lemma~\ref{lem:hoeldercont}. However, we must refine the estimate $\|\eta_f\|_\infty \leq C\|f-h\|^{\xi/2}$ in Lemma~\ref{lem:optimallinfinity}, upgrading the Hölder exponent to $\xi$. This refinement constitutes the main technical step prior to the proof of Theorem~\ref{thm:mainthm1}.

\begin{prop}
	\label{prop:strictconcavity2}
Let $\xi \in (0,1)$ and $r \in (0,\frac{\pi}{2})$. Then there exist $0 < \varepsilon_1 \leq \varepsilon_2 < \frac{r}{2}$, and $c, C > 0$ such that for all $h \in \bS^2_+ \cap E_r(\bS^1)$ and all $f \in \B(h,\varepsilon_1) \cap E^+(\bS^1)$, the following properties hold:
\begin{enumerate}
	\item $\eta \mapsto \Psi_h(f,\eta)$ is strictly concave on $L^2_{\pi,0}(\R)\cap\B_{L^\infty(\R)}(0,\varepsilon_2)$, and admits a unique maximizer $\eta_f$ satisfying
	\[
	\Psi_h(f,\eta_f) = \sup_{\eta \in L^2_{\pi,0}(\R)} \Psi_h(f,\eta) = \|\omega_f\|_{\rm ir}.
	\]
	Moreover, for any $\eta^0,\eta^1 \in L^2_{\pi,0}(\R) \cap \B_{L^\infty(\R)}(0,\varepsilon_2)$, the inequality
	\[
	\frac{d^2}{dt^2}\Psi_h(f,\eta^t) \leq -c\|\eta^1 - \eta^0\|_2^2
	\]
	holds for all $t \in [0,1]$, where $\eta^t \defl (1-t)\eta^0 + t\eta^1$.
	\item $\eta_f$ is continuous and $f \mapsto \eta_f$ is continuous at $h$ in the sense that
	\[
	\|\eta_f\|_\infty \leq C\|f - h\|_\infty^\xi .
	\]
\end{enumerate}
\end{prop}

\begin{proof}
Fix $\xi \in (0,1)$ and $r \in (0,\frac{\pi}{2})$. We may choose $0 < \varepsilon_1(\xi,r) < \frac{r}{2}$ sufficiently small such that Lemma~\ref{lem:optimallinfinity} applies with parameter $\varepsilon_1$. That is, for all $h \in \bS^2_+ \cap E_r(\bS^1)$ and $f \in \B(h,\varepsilon_1) \cap E^+(\bS^1)$, the function $\eta_f \in L^2_{\pi,0}(\R)$, provided by Lemma~\ref{lem:optimallinfinity}, satisfies
\[
\|\eta_f\|_\infty \leq C(\xi,r)\|f-h\|_\infty^\frac{\xi}{2} \leq C(\xi,r)\varepsilon_1^\frac{\xi}{2}.
\]
If $\varepsilon_1$ is sufficiently small, we can assume that Lemma~\ref{lem:strictconcavity} holds with parameter
\[
\varepsilon_2 \defl \max\left\{ \varepsilon_1, C(\xi,r)\varepsilon_1^\frac{\xi}{2} \right\} < \frac{r}{2}.
\]

Fix $h \in \bS^2_+ \cap E_r(\bS^1)$ and $f \in \B(h,\varepsilon_1) \cap E^+(\bS^1)$. The function $\eta_f$ is continuous, satisfies
\[
\Psi_h(f,\eta_f) = \sup_{\eta \in L^2_{\pi,0}(\R)} \Psi_h(f,h),
\]
and is contained in 
$L^2_{\pi,0}(\R) \cap \B_{L^\infty(\R)}(0,\varepsilon_2)$ by the choice of $\varepsilon_1$ and $\varepsilon_2$. Let $\eta^t$ be a variation in $L^2_{\pi,0}(\R) \cap \B_{L^\infty(\R)}(0,\varepsilon_2)$ as in (1). Since $\varepsilon_1 \leq \varepsilon_2$, it follows from Lemmas~\ref{lem:psismooth} and \ref{lem:strictconcavity} that
\begin{align}
	\nonumber
\frac{d^2}{dt^2}\Psi_h(f,\eta^t) & = -\int_0^\pi\int_\alpha^\pi p_{\alpha,\beta}(f) \sin\left(\Delta_{\alpha,\beta}^h + \Delta_{\alpha,\beta}^{\eta^t}\right)\left(\Delta_{\alpha,\beta}^{\eta^t}\right)^2\,d\beta\,d\alpha \\
	\label{eq:strictconcave}
 & \leq -c\|\eta^1 - \eta^0\|_2^2
\end{align}
for some $c(r) > 0$. This implies that $\eta \mapsto \Psi_h(f,\eta)$, when restricted to $L^2_{\pi,0}(\R) \cap \B_{L^\infty(\R)}(0,\varepsilon_2)$, is strictly concave. Consequently, $\eta_f$ is the unique maximizer in this set. This proves (1).

For $\eta \in L^2_{\pi,0}(\R) \cap \B_{L^\infty(\R)}(0,\varepsilon_2)$, let $\gamma_\eta$ and $\mu_{f,\eta}$ denote the associated paths as defined in \eqref{eq:sigmadef} and \eqref{eq:mudef}, respectively. By Lemma~\ref{lem:pfunction}(4), there exists $M(r) > 0$ such that $q_{\alpha,\beta}(h) \leq M(r)$ . Recall that $\|\eta\|_{2}^2 = \pi\int_0^{\pi}\eta^2$, and note that $\|\eta\|_\infty \leq \varepsilon_2 \leq \frac{\pi}{2}$. From Lemma~\ref{lem:hoeldercont} it follows that
\begin{align*}
\left|\mu_{f,\gamma_{\eta}}(\alpha) - \mu_{h,\gamma_{0}}(\alpha)\right| & \leq \left|\mu_{f,\gamma_{\eta}}(\alpha) - \mu_{h,\gamma_{\eta}}(\alpha)\right| + \left|\mu_{h,\gamma_{\eta}}(\alpha) - \mu_{h,\gamma_{0}}(\alpha)\right| \\
 & \leq H \|f - h\|_\infty^\xi + \int_\alpha^{\alpha + \pi}q_{\alpha,\beta}(h)\left|\gamma_{\eta}(\beta)-\gamma_{0}(\beta)\right|\,d\beta \\
 & \leq H \|f - h\|_\infty^\xi + M\int_\alpha^{\alpha + \pi}|e^{i\eta(\beta)}-1|\,d\beta \\
 & \leq H \|f - h\|_\infty^\xi + M\int_\alpha^{\alpha + \pi}|\eta(\beta)|\,d\beta \\
 & \leq H \|f - h\|_\infty^\xi + M\left(\pi\int_\alpha^{\alpha + \pi}|\eta(\beta)|^2\,d\beta\right)^\frac{1}{2} \\
 & = H \|f-h\|_\infty^\xi + M\|\eta\|_2 .
\end{align*}
Together with Lemma~\ref{lem:optimallinfinity}(4) and $\|\eta_f\|_\infty \leq \frac{\pi}{2}$, this implies
\begin{equation}
\label{eq:linftyuperbound}
\|\eta_f\|_\infty \leq \tfrac{\pi}{2}\|e^{i\eta_f} - 1\|_\infty = \tfrac{\pi}{2}\left\|\gamma_{\eta_f} - \gamma_{0}\right\|_\infty \leq a(\|f-h\|_\infty^\xi + \|\eta_f\|_2)
\end{equation}
for some constant $a(\xi,r) > 0$.

For $\eta \in L^2_{\pi,0}(\R) \cap \B_{L^\infty(\R)}(0,\varepsilon_2)$ and $v \in L^2_{\pi,0}(\R)$, define
\[
F(f,\eta)(v) \defl \int_0^\pi\int_\alpha^\pi p_{\alpha,\beta}(f)\cos(\Delta_{\alpha,\beta}^h + \Delta_{\alpha,\beta}^{\eta})\Delta^v_{\alpha,\beta}\,d\beta\,d\alpha.
\]
We consider the variation $\psi(t) \defl F(f,t\eta_f)(\eta_f)$ for $t \in [0,1]$. By Lemma~\ref{lem:psismooth}, the first derivative of $\psi$ is given by
\[
\psi'(t) = -\int_0^\pi\int_\alpha^\pi p_{\alpha,\beta}(f) \sin(\Delta_{\alpha,\beta}^h + \Delta_{\alpha,\beta}^{t\eta_f})(\Delta_{\alpha,\beta}^{\eta_f})^2\,d\beta\,d\alpha.
\]
Moreover, with \eqref{eq:strictconcave}, this satisfies the lower bound
\[
|\psi'(t)| \geq c \|\eta_f\|_2^2.
\]
Since $\psi$ is continuously differentiable by Lemma~\ref{lem:psismooth}, the mean value theorem yields the existence of some $m \in (0,1)$ with
\[
|F(f,\eta_f)(\eta_f) - F(f,0)(\eta_f)| = |\psi'(m)| .
\]
It holds $F(f,\eta_f)(\eta_f) = 0 = F(h,0)(\eta_f)$ because $\Psi_h(f,\cdot)$ is stationary at $\eta_f$ and $\Psi_h(f,\cdot)$ is stationary at $0$. Together with Lemma~\ref{lem:hoeldercont}, this implies that
\begin{align*}
c\|\eta_f\|_2^2 & \leq |\psi'(m)| = |F(f,\eta_f)(\eta_f) - F(f,0)(\eta_f)| \\
  & = \left|F(f,0)(\eta_f) - F(h,0)(\eta_f)\right| \\
 & = \left|\int_0^\pi\int_\alpha^\pi (p_{\alpha,\beta}(f) - p_{\alpha,\beta}(h)) \cos(\Delta_{\alpha,\beta}^h) \Delta_{\alpha,\beta}^{\eta_f} \,d\beta\,d\alpha\right| \\
 & \leq 2\|\eta_f\|_\infty \int_0^\pi\int_\alpha^\pi |p_{\alpha,\beta}(f) - p_{\alpha,\beta}(h)|\,d\beta\,d\alpha \\
 & \leq 2H\|\eta_f\|_\infty \|f - h\|_\infty^\xi .
\end{align*}
Combined with \eqref{eq:linftyuperbound}, this yields constants $a(\xi,r),b(\xi,r) > 0$ such that
\begin{equation*}
\|\eta_f\|_\infty \leq a(\|f-h\|_\infty^\xi + \|\eta_f\|_2) , \quad
\|\eta_f\|_2 \leq b\|\eta_f\|_\infty^\frac{1}{2} \|f - h\|_\infty^\frac{\xi}{2} .
\end{equation*}
Setting $x \defl \|\eta_f\|_\infty^\frac{1}{2}$ and $y \defl \|f - h\|_\infty^\frac{\xi}{2}$, we estimate
\begin{align*}
x^4 & \leq a^2(\|\eta_f\|_2^2 + 2\|\eta_f\|_2\|f-h\|_\infty^\xi + \|f-h\|_\infty^{2\xi}) \\
 & \leq a^2\bigl(b^2\|\eta_f\|_\infty \|f - h\|_\infty^\xi + 2b \|\eta_f\|_\infty^\frac{1}{2} \|f - h\|_\infty^\frac{3\xi}{2} + \|f-h\|_\infty^{2\xi}\bigr) \\
 & = a^2(b^2x^2y^2 + 2bxy^3 + y^4) \\
 & = a^2y^2(y + bx)^2 .
\end{align*}
Hence $x^2 \leq a y(y + bx)$. We claim that this implies $x \leq Cy$, where
\[
C \defl \left(\sqrt{\frac{b^2}{4} + \frac{1}{a}} - \frac{b}{2}\right)^{-1} .
\]
Assume for contradiction that $x > Cy$. Then:
\begin{align*}
y^2 + bxy - a^{-1}x^2 & = (y + \tfrac{1}{2}bx)^2 - x^2(\tfrac{1}{4}b^2 + a^{-1}) \\
 & < x^2(C^{-1} + \tfrac{1}{2}b)^2 - x^2(\tfrac{1}{4}b^2 + a^{-1}) \\
 & = x^2(\tfrac{1}{4}b^2 + a^{-1}) - x^2(\tfrac{1}{4}b^2 + a^{-1}) \\
 & = 0 ,
\end{align*}
which contradicts $x^2 \leq a y(y + bx)$. Therefore, $x \leq Cy$ as claimed. Substituting back the definitions of $x$ and $y$, we conclude that
\[
\|\eta_f\|_\infty \leq C^2\|f - h\|_\infty^\xi .
\]
This establishes the estimate in (2).
\end{proof}

\subsection{Proofs of the main theorems}

Fix $\xi \in (0,1)$, $r \in (0,\frac{\pi}{2})$ and $h \in \bS^2_+\setminus \bS^1$ with representation $h = \arccos(\cos(d)\cos(\cdot - \tau))$ for $\tau \in \R$ and $d \in [r,\frac{\pi}{2}]$. Let $\varepsilon_1(\xi,r),\varepsilon_2(\xi,r) > 0$ be as in Proposition~\ref{prop:strictconcavity2}, ensuring that for all $f \in \B(h,\varepsilon_1)$, a maximizer $\eta_f$ of $\Psi_h(f,\cdot)$ exists in $L^2_{\pi,0}(\R) \cap \B_{L^\infty(\R)}(0,\varepsilon_2)$. The function $\Psi_h$ is defined as in the last subsection by
\[
\Psi_h(f,\eta) = \int_0^\pi\int_\alpha^\pi p_{\alpha,\beta}(f) \sin(\nu_h(\beta) - \nu_h(\alpha) + \eta(\beta) - \eta(\alpha))\,d\beta\,d\alpha ,
\]
where $\nu_h$ is the unique bi-Lipschitz function satisfying the conditions $\nu_h(0) = 0$, $\nu_h(\alpha + \pi) = \nu_h(\alpha) + \pi$ and
\begin{equation*}
\nu_h'(\alpha) = p_\alpha(h) = \frac{\sin(d)}{\sin(h_\alpha)^2} = \frac{\sin(d)}{1 - \cos(d)^2\cos(\alpha - \tau)^2} \in [m_1(r),m_2(r)] .
\end{equation*}
For $f \in \B(h,\varepsilon_1) \cap E^+(\bS^1)$, we aim to estimate
\begin{equation}
\label{eq:mainestimate}
	|\Psi_h(h,0) - \Psi_h(f,\eta_f)| \leq |\Psi_h(h,0) - \Psi_h(f,0)| + |\Psi_h(f,0) - \Psi_h(f,\eta_f)|.
\end{equation}
We start with the second term on the right-hand side. Define $\psi(t) \defl \Psi_h(f,(1-t)\eta_f)$. By Lemma~\ref{lem:psismooth}, $\psi$ belongs to $C^2([0,1])$ and satisfies $\psi'(0) = 0$ since $\eta_f$ is a maximizer of $\Psi_h(f,\cdot)$. Moreover, Lemma~\ref{lem:pfunction}(4) ensures that $p_{\alpha,\beta}(f) \leq M(r)$ for all $0 < \alpha < \beta < \pi$, and Proposition~\ref{prop:strictconcavity2}(2) yields the bound $\|\eta_f\|_\infty \leq C(\xi,r)\|f-h\|_\infty^{\xi}$. By Taylor’s theorem, there exists some $m \in (0,1)$ such that
\begin{align}
\nonumber
|\Psi_h(f,0) - \Psi_h(f,\eta_f)| & = |\psi(1) - \psi(0) - \psi'(0)| = \left|\tfrac{1}{2}\psi''(m) \right| \\
\nonumber
& = \left|\frac{1}{2}\int_0^\pi\int_\alpha^\pi p_{\alpha,\beta}(f) \sin(\Delta_{\alpha,\beta}^h + \Delta_{\alpha,\beta}^{(1-m)\eta_f})(\Delta_{\alpha,\beta}^{\eta_f})^2\,d\beta\,d\alpha\right| \\
\nonumber
& \leq \frac{1}{2}M\int_0^\pi\int_\alpha^\pi (\Delta_{\alpha,\beta}^{\eta_f})^2\,d\beta\,d\alpha \\
\nonumber
& \leq \pi^2 M \|\eta_f\|_\infty^2 \\
\label{eq:secondterm}
& \leq \pi^2 M C\|f-h\|_\infty^{2\xi} .
\end{align}

To estimate the first term on the right-hand side of \eqref{eq:mainestimate}, define $f^t \defl (1-t)h + tf$ for $t \in [0,1]$. Analogous to the previous argument, consider the function
\[
\phi(t) \defl \Psi_h(f^t,0) = \int_0^\pi\int_\alpha^\pi p_{\alpha,\beta}(f^t)\sin(\nu_h(\beta)-\nu_h(\alpha))\,d\beta\,d\alpha .
\]
The function $\phi$ belongs to $C^2([0,1])$ by Lemma~\ref{lem:derivative}. For all $0 < \alpha < \beta < \pi$ we have the uniform bound
\[
\sup_{t\in[0,1]}|\partial_t^2 p_{\alpha,\beta}(f^t)| \leq C(r) \max\left\{ 1, \frac{\|f-h\|_\infty}{\sin(\beta-\alpha)^{2}} \right\}
\]
by Lemma~\ref{lem:unifboundder}. The second derivative of $\phi$ is therefore uniformly bounded by
\[
|\phi''(t)| \leq \int_0^\pi\int_\alpha^\pi \sup_{t\in[0,1]}|\partial_t^2 p_{\alpha,\beta}(f^t)| \sin(\nu_h(\beta)-\nu_h(\alpha))\,d\beta\,d\alpha .
\]
Note that $\sin(\nu_h(\beta)-\nu_h(\alpha)) \leq \frac{\pi}{2}m_2 \sin(\beta-\alpha)$ for $0 < \alpha < \beta < \pi$. This follows because
\[
\sin(\nu_h(\beta)-\nu_h(\alpha)) \leq \nu_h(\beta)-\nu_h(\alpha) \leq m_2(\beta - \alpha) \leq \tfrac{\pi}{2}m_2\sin(\beta-\alpha)
\]
if $\beta - \alpha \leq \frac{\pi}{2}$. For the case $\beta - \alpha \leq \frac{\pi}{2}$, replace $(\alpha,\beta)$ by $(\beta,\alpha+\pi)$ and apply the same argument. Thus, the integrand above admits the upper bound
\[
C_1\sin(\beta-\alpha)^{-1}\|f-h\|_\infty^2 ,
\]
for some $C_1(r) > 0$. Exactly as in the proof of Lemma~\ref{lem:hoeldercont}, this yields the estimate
\begin{equation}
\label{eq:phiderder}
|\phi''(t)| \leq C_2\|f-h\|_\infty^{2\xi}
\end{equation}
for all $t \in [0,1]$ and $C_2(\xi,r) > 0$.

We claim that $\phi'(0) = 0$. Denoting $\delta \defl f - h$, and using the notation from Lemma~\ref{lem:derivative}, we have
\begin{align*}
\phi'(0) & = \int_0^\pi\int_\alpha^\pi (\partial_x p(\beta-\alpha,h_\alpha,h_\beta)\delta_\alpha + \partial_y p(\beta-\alpha,h_\alpha,h_\beta)\delta_\beta)\sin(\Delta_{\alpha,\beta}^h)\,d\beta\,d\alpha .
\end{align*}
For simplicity, assume $\tau = 0$, so that $h_\alpha = \arccos(\cos(d)\cos(\alpha))$. Using the trigonometric identity
\[
\cos(\beta-\alpha)\cos(\alpha) - \cos(\beta) = \sin(\beta - \alpha)\sin(\alpha) ,
\]
and Lemma~\ref{lem:derivative}, the first partial derivative is, for almost every pair $(\alpha,\beta)$,
\begin{align*}
\frac{1}{2}\partial_x p(\beta & - \alpha,h_\alpha,h_\beta) \\
& = \frac{(\cos(\beta-\alpha)\cos(h_\alpha) - \cos(h_\beta))(\cos(\beta-\alpha) - \cos(h_\alpha)\cos(h_\beta))} {\sin(\beta-\alpha)^2\sin(h_\alpha)^3\sin(h_\beta)^2} \\
& = \cos(d)\frac{\sin(\alpha)(\cos(\beta-\alpha) - \cos(h_\alpha)\cos(h_\beta))} {\sin(\beta-\alpha)\sin(h_\alpha)^3\sin(h_\beta)^2}.
\end{align*}
By \eqref{eq:areaformula2}, we have
\begin{align*}
\sin(\Delta_{\alpha,\beta}^h) = \sin(\nu_h(\beta)-\nu_h(\alpha)) = \frac{\sin(d)\sin(\beta-\alpha)}{\sin(h_\alpha)\sin(h_\beta)}.
\end{align*}
Multiplying this with the first partial derivative term from before yields
\begin{align*}
\frac{1}{2}\partial_x p(\beta-\alpha,h_\alpha,h_\beta)\sin(\Delta_{\alpha,\beta}^h) = \frac{\cos(d)\sin(d)\sin(\alpha)}{\sin(h_\alpha)^4}\frac{\cos(\beta-\alpha) - \cos(h_\alpha)\cos(h_\beta)} {\sin(h_\beta)^3} .
\end{align*}
In particular, $\partial_x p(\beta-\alpha,h_\alpha,h_\beta)\delta_\alpha\sin(\Delta_{\alpha,\beta}^h)$ is integrable over $[0,\pi]^2$. By symmetry, the analogous term $\partial_y p(\beta-\alpha,h_\alpha,h_\beta)\delta_\beta\sin(\Delta_{\alpha,\beta}^h)$ is also integrable and can be rewritten as
\begin{align*}
& \int_0^\pi\int_\alpha^\pi \partial_y p(\beta-\alpha,h_\alpha,h_\beta)\delta_\beta\sin(\Delta_{\alpha,\beta}^h)\,d\beta\,d\alpha \\
 & \qquad = \int_0^\pi\int_\alpha^\pi \partial_x p(\beta-\alpha,h_\beta,h_\alpha)\delta_\beta\sin(\Delta_{\alpha,\beta}^h)\,d\beta\,d\alpha \\
 & \qquad = \int_0^\pi\int_0^\beta \partial_x p(\beta-\alpha,h_\beta,h_\alpha)\delta_\beta\sin(\Delta_{\alpha,\beta}^h)\,d\alpha\,d\beta \\
 & \qquad = \int_0^\pi\int_0^\alpha \partial_x p(\alpha-\beta,h_\alpha,h_\beta)\delta_\alpha\sin(\Delta_{\beta,\alpha}^h)\,d\beta\,d\alpha \\
 & \qquad = -\int_0^\pi\int_0^\alpha \partial_x p(\beta-\alpha,h_\alpha,h_\beta)\delta_\alpha\sin(\Delta_{\alpha,\beta}^h)\,d\beta\,d\alpha \\
 & \qquad = \int_0^\pi\int_\pi^{\alpha + \pi} \partial_x p(\beta-\alpha,h_\alpha,h_\beta) \delta_\alpha \sin(\Delta_{\alpha,\beta}^h)\,d\beta\,d\alpha .
\end{align*}
In the last two lines, we used the identities
\[
\partial_x p(\beta-\alpha,h_\alpha,h_\beta) = \partial_x p(\alpha-\beta,h_\alpha,h_\beta) = \partial_x p(\pi + \beta-\alpha,h_\alpha,h_{\beta+\pi}),
\]
\[
\sin(\Delta_{\beta,\alpha}^h) = -\sin(\Delta_{\alpha,\beta}^h) = \sin(\Delta_{\alpha,\beta+\pi}^h).
\]
It follows that
\[
\phi'(0) = \int_0^\pi \delta_\alpha \int_\alpha^{\alpha + \pi} \partial_x p(\beta-\alpha,h_\alpha,h_\beta) \sin(\Delta_{\alpha,\beta}^h)\,d\beta\,d\alpha .
\]
Therefore, the claim will hold if the inner integral
\begin{align}
	\label{eq:finalvariation}
	\int_\alpha^{\alpha + \pi} \frac{\cos(\beta - \alpha) - \cos(h_\alpha)\cos(h_\beta)} {\sin(h_\beta)^3} \, d\beta
\end{align}
vanishes for all $\alpha$. The integrand in \eqref{eq:finalvariation} is equal to
\begin{align*}
	\frac{\cos(\beta - \alpha) - \cos(h_\alpha)\cos(h_\beta)} {\sin(h_\beta)^3}
	& = \frac{\cos(\beta - \alpha) - \cos(d)^2\cos(\alpha)\cos(\beta)} {(1-\cos(d)^2\cos(\beta)^2)^\frac{3}{2}} \\
	& = \frac{\sin(\alpha)\sin(\beta) + \sin(d)^2\cos(\alpha)\cos(\beta)}{(1-\cos(d)^2\cos(\beta)^2)^\frac{3}{2}}.
\end{align*}
Since
\begin{align*}
	\frac{\partial}{\partial \beta} \frac{\sin(\beta)}{(1-\cos(d)^2\cos(\beta)^2)^\frac{1}{2}} & = \frac{\sin(d)^2\cos(\beta)}{(1-\cos(d)^2\cos(\beta)^2)^\frac{3}{2}} , \\
	\frac{\partial}{\partial \beta} \frac{\cos(\beta)}{(1-\cos(d)^2\cos(\beta)^2)^\frac{1}{2}} & = \frac{-\sin(\beta)}{(1-\cos(d)^2\cos(\beta)^2)^\frac{3}{2}} ,
\end{align*}
it follows that the integral in \eqref{eq:finalvariation} is equal to
\begin{align*}
	\left. \frac{-\sin(\alpha)\cos(\beta) + \cos(\alpha)\sin(\beta)}{(1-\cos(d)^2\cos(\beta)^2)^\frac{1}{2}} \right|_{\alpha}^{\alpha+\pi} = \left. \frac{\sin(\beta - \alpha)}{(1-\cos(d)^2\cos(\beta)^2)^\frac{1}{2}} \right|_{\alpha}^{\alpha+\pi} = 0 . 
\end{align*}
This shows that $\phi'(0) = 0$. Applying Taylor's theorem once more, there exists some $m \in (0,1)$ such that, using \eqref{eq:phiderder}, we have
\begin{align*}
|\Psi_h(f,0) - \Psi_h(h,0)| & = |\phi(1) - \phi(0) - \phi'(0)| = |\tfrac{1}{2}\phi''(m)| \leq C_2\|f-h\|_\infty^{2\xi} .
\end{align*}
Combining this with \eqref{eq:secondterm} and Proposition~\ref{prop:strictconcavity2}, it follows that
\begin{equation}
\label{eq:eplusestimate}
|\|\omega_{h}\|_{\rm ir} - \|\omega_f\|_{\rm ir}| = |\Psi_h(h,0) - \Psi_h(f,\eta_f)| \leq C_3\|f - h\|_\infty^{2\xi}
\end{equation}
for some $C_3(\xi,r) > 0$, provided that $f \in \B(h,\varepsilon_1) \cap E^+(\bS^1)$. Since $f \mapsto \|\omega_f\|_{\rm ir}$ is uniformly continuous on $\B(h,\varepsilon_1)$ by Lemma~\ref{lem:hoeldercont}, the estimate \eqref{eq:eplusestimate} extends to all $f \in \B(h,\varepsilon_1)$ by density of $E^+(\bS^1)$ in $E(\bS^1)$, as established in Lemma~\ref{lem:epluslem}.

\begin{proof}[Proof of Theorem~\ref{thm:mainthm1}]
Let $r \in (0,\frac{\pi}{2})$ and $\xi \in (1,2)$ be fixed as in the statement of Theorem~\ref{thm:mainthm1}. Choose $\varepsilon_1(\frac{\xi}{2},\frac{r}{2}) > 0$ as above, and suppose that $f \in E_r(\bS^1) \cap \B(\bS^2_+,\varepsilon_1)$. Note that $\varepsilon_1 \leq \frac{r}{4}$ as assumed in Proposition~\ref{prop:strictconcavity2}. Let $h \in \bS^2_+$ be such that $\|f-h\|_\infty = \dist(f,\bS^2_+)$. By construction, $h \in E_{\frac{r}{2}}(\bS^1)$, since for every $g \in E(\bS^1)\setminus E_{\frac{r}{2}}(\bS^1)$, we have $\|f-g\|_\infty > \frac{r}{2} > \varepsilon_1$. Hence, by \eqref{eq:eplusestimate} we obtain
\[
|\|\omega_f\|_{\rm ir} - \pi| \leq C_3(\tfrac{\xi}{2},\tfrac{\varepsilon}{2})\|f - h\|_\infty^{\xi} = C_3(\tfrac{\xi}{2},\tfrac{r}{2})\dist(f,\bS^1)^{\xi} .
\]
Here we used that $\|\omega_{h}\|_{\rm ir} = \pi$ for every $h \in \bS^2_+\setminus \bS^1$, which follows from \eqref{eq:spherecalibration}. Since $\|\omega_f\|_{\rm ir}$ is uniformly bounded on $E_r(\bS^1)$ by Lemma~\ref{lem:hoeldercont}, the above estimate extends to all $f \in E_r(\bS^1)$ with a possibly larger constant. Since the two-form introduced in the introduction satisfies $\tilde \omega = \frac{1}{\pi}\omega$, this concludes the proof of Theorem~\ref{thm:mainthm1}.
\end{proof}

Theorem~\ref{thm:mainthm2} follows rather directly from Theorem~\ref{thm:mainthm1} and Proposition~\ref{prop:massbound}. 

\begin{proof}[Proof of Theorem~\ref{thm:mainthm2}]
Fix $r \in (0,\frac{\pi}{2})$ and $\xi \in (1,2)$. The radial retraction, $\rho_r : \bS^2_+ \to \bS^2_+ \cap E_r(\bS^1)$ is $1$-Lipschitz. This follows from Gauss’s lemma and the fact that, for any two unit-speed geodesics $\gamma_1,\gamma_2 : [0,\frac{\pi}{2}] \to \bS^2_+$ emitting from the north pole, the distance between them is non-decreasing; that is, $d(\gamma_1(s),\gamma_2(s)) \leq d(\gamma_1(t),\gamma_2(t))$ for all $s \leq t$. The image of this retraction is denoted by $S_r \defl \{p \in \bS^2_+ : \dist(p,\bS^1)\geq r\}$. Since $E(\bS^1)$ is an injective metric space, the map $\rho_r$ admits a $1$-Lipschitz extension $\bar \rho_r : E(\bS^1) \to E(\bS^1)$. By Lemma~\ref{lem:truncated}(4), we may assume that the image of $\bar\rho_r$ is contained in $E_r(\bS^1)$, since this holds for $\rho_r$. Now let $T,S \in \cR_2(X)$ be as in the statement of Theorem~\ref{thm:mainthm1}. The current $S$ is isometric to $\curr{\bS^2_+}$, and hence there exists an isometric embedding $\varphi : \spt(S) \to E(\bS^1)$ such that $\varphi(\spt(S)) = \bS^2_+$ and $\varphi_\# S = \curr{\bS^2_+}$. The map $\varphi$ admits a $1$-Lipschitz extension $\bar \varphi : X \to E(\bS^1)$. We consider the composition $\psi \defl \bar\rho_r\circ \bar\varphi : X \to E_r(\bS^1)$, as well as the restrictions $S^\circ \defl S\res (X \setminus N_r)$ and $T^\circ \defl T\res (X \setminus N_r)$, to the portions away from the collar $N_r \defl \B(\spt(\partial S),r) \cap \spt(S)$. By assumption, $S - S^\circ = T - T^\circ$, and
\[
\bar \varphi_\#(T - T^\circ) = \bar \varphi_\#(S - S^\circ) = \curr{\bS^2_+\setminus S_r} .
\]
Define $T'\defl \psi_\# T^\circ$ and observe that
\[
\partial T' = \psi_\# \partial T^\circ = \psi_\# \partial S^\circ = \partial \curr{S_r}.
\]
Set $d \defl h(\spt(T),\spt(S))$ to be the directed Hausdorff distance from $\spt(T)$ to $\spt(S)$. Since $\psi$ is $1$-Lipschitz, $\spt(T')$ is contained in $\B(S_r,d) \cap E_r(\bS^1)$. Applying Theorem~\ref{thm:mainthm1} and Proposition~\ref{prop:massbound}, we obtain
\[
T'(\tilde \omega) \leq \bM_{\rm ir}(T') \sup_{f \in \spt(T')} \|\tilde\omega_f\| \leq \bM_{\rm ir}(T')(1 + Cd^{\xi})
\]
for some constant $C(\xi,r) > 0$. The form $\tilde \omega$ is exact by Lemma~\ref{lem:action}, and therefore $T'(\tilde \omega) = \curr{S_r}(\tilde \omega)$. Applying Lemma~\ref{lem:masslem}(5) to the $1$-Lipschitz maps $\psi$ and $\bar\varphi$ implies
\begin{align*}
2\pi & = \curr{\bS^2_+}(\tilde \omega) = (T' + \curr{\bS^2_+\setminus S_r})(\tilde\omega) = T'(\tilde\omega) + \bM_{\rm ir}(\curr{\bS^2_+\setminus S_r}) \\
 & \leq (\bM_{\rm ir}(T') + \bM_{\rm ir}(\curr{\bS^2_+ \setminus S_r}))(1 + Cd^{\xi}) \\
 & = (\bM_{\rm ir}(\psi_\# T^\circ) + \bM_{\rm ir}(\bar\varphi_\# (T - T^\circ)))(1 + Cd^{\xi}) \\
 & \leq (\bM_{\rm ir}(T^\circ) + \bM_{\rm ir}(T - T^\circ))(1 + Cd^{\xi}) \\
 & = \bM_{\rm ir}(T)(1 + Cd^{\xi}) .
\end{align*}
In the last line we used that $T = T^\circ + (T - T^\circ) = T \res (X \setminus N_r) + T \res N_r$ is a disjoint decomposition. Theorem~\ref{thm:mainthm2} trivially holds if $\bM_{\rm ir}(T) \geq 2\pi$. If $\bM_{\rm ir}(T) \leq 2\pi$, then
\[
2\pi \leq \bM_{\rm ir}(T) + 2\pi C d^\xi.
\]
This completes the proof of Theorem~\ref{thm:mainthm2}.
\end{proof}

\section{Comments}
\label{sec:comments}

\subsection{Other definitions of area}

Gromov's filling area conjecture, as stated in the introduction, is formulated for the inscribed Riemannian Finsler volume, for the natural reason that it is the largest such volume, as stated in Proposition~\ref{lem:masslem}. For certain alternative choices of volume, the conjecture fails to hold. To illustrate this, consider the cone $C \defl N \cone \curr{\bS^1}$ as a current in $\cR_2(E(\bS^1))$, where $N = \frac{\pi}{2} \in E(\bS^1)$ is the constant function. The current $C$ is an oriented Lipschitz disk that contains all the functions $f \in E(\bS^1)$ of the form
\begin{equation}
\label{eq:conefunction}
f = (1-r)\tfrac{\pi}{2} + r d_\alpha ,
\end{equation}
where $d_\alpha$ is the distance function to $\alpha \in \bS^1$, and $r \in [0,1]$. Alternatively, $C = \varphi_\# \curr{[0,1] \times [0,2\pi]}$ with the parametrization $\varphi(r,\alpha)$ given on the right-hand side of \eqref{eq:conefunction}. For a particular Finsler area $\mu$, the $\mu$-mass of $C$ is given by
\[
\bM_\mu(C) = \int_{[0,1] \times [0,2\pi]} \mathbf{J}_{\mu}(\operatorname{md}\varphi_x) \,dx .
\]
Let $X \defl \ell^2_1$, that is, $\R^2$ equipped with the norm $\|xe_1 + ye_2\|_1 = |x| + |y|$, and let $\B_X$ denote the unit disk in $X$. Then, for almost every $(r,\alpha)$, we have
\[
d(\varphi(r,\alpha),\varphi(r + h,\alpha + k)) = \tfrac{\pi}{2}|h| + r |k| + o(|h| + |k|) .
\]
Thus
\[
\operatorname{md} \varphi_{(r,\alpha)}(h,k) = \tfrac{\pi}{2}|h| + r |k|,
\]
and
\[
\mathbf{J}_{\mu}(\operatorname{md} \varphi_{(r,\alpha)}) = \tfrac{\pi}{2}r \mathbf{J}_{\mu}(\|\cdot\|_1) = \tfrac{\pi}{2}r \mu_X(e_1 \wedge e_2).
\]
By integration, it follows that
\begin{align*}
\bM_\mu(C) & = \int_0^1 \int_0^{2\pi} \mathbf{J}_{\mu}(\operatorname{md} \varphi_{(r,\alpha)}) \,d\alpha \,dr \\
 & = 2\pi \int_0^1 \tfrac{\pi}{2}r \mu_X(e_1 \wedge e_2) \, dr \\
 & = \tfrac{\pi^2}{2} \mu_X(e_1 \wedge e_2) .
\end{align*}
Using the properties of various area definitions as stated in \cite[\S3]{APT}, and observing that $\mu_X(\B_X) = 2\mu_X([0,1]^2) = 2\mu_X(e_1 \wedge e_2)$:
\begin{itemize}
	\item (Gromov-mass) $\mu_X^{\rm m}(e_1 \wedge e_2) = \inf\{\|v\|_1\|w\|_1 : v \wedge w = e_1 \wedge e_2\} = 1$.
	\item (Gromov-mass$\ast$) $\mu_X^{\rm m\ast}(e_1 \wedge e_2) = \inf\{\langle\xi \wedge \eta, e_1 \wedge e_2\rangle : \|\xi\|_\infty,\|\eta\|_\infty \leq 1 \} = 2$.
	\item (Busemann-Hausdorff) $\mu_X^{\rm bh}(\B_X) = \pi$.
	\item (Holmes-Thompson) $\mu_X^{\rm ht}(e_1 \wedge e_2) = \tfrac{1}{\pi}\operatorname{Area}(\B_{X^\ast} ; e_1^\ast \wedge e_2^\ast) = \tfrac{4}{\pi}$.
\end{itemize}
Thus the different Finsler areas of the cone are
\[
\bM_{\rm m}(C) = \tfrac{\pi^2}{2} < 2\pi = \bM_{\rm ht}(C) < \bM_{\rm bh}(C) = \tfrac{\pi^3}{4} < \pi^2 = \bM_{\rm m\ast}(C) .
\]
In particular, $\bM_{\rm m}(C) < 2\pi$. So Gromov's filling conjecture fails for the Gromov-mass area. Moreover, with respect to the Holmes-Thompson area, the hemisphere cannot be the unique minimal filling of $\bS^1$.

\subsection{Lower bounds on the filling area}

Let $T \in \cR_2(E(\bS^1))$ with $\partial T = \curr{\bS^1}$, and fix $\alpha \in \R$. Then
\begin{align*}
T(d\pi_{\alpha} \wedge d\pi_{\alpha + \frac{\pi}{2}}) & = \curr{\bS^1}(\pi_{\alpha}\, d\pi_{\alpha + \frac{\pi}{2}}) = (\pi_\alpha,\pi_{\alpha+\frac{\pi}{2}})_\#\curr{\bS^1}(x\,dy) = \tfrac{\pi^2}{2} .
\end{align*}
In the last line we used that $(\pi_\alpha, \pi_{\alpha+\frac{\pi}{2}})_\#\curr{\bS^1}$ is a counterclockwise parametrization of the square in $\R^2$ with vertices $(0,\frac{\pi}{2})$, $(\frac{\pi}{2},0)$, $(\pi,\frac{\pi}{2})$, $(\frac{\pi}{2},\pi)$. The enclosed area is $\tfrac{\pi^2}{2}$, and by Lemma~\ref{lem:masslem}(2) this implies that
\begin{equation}
\label{eq:lowerbound}
\bM_{\rm m\ast}(T) = \bM(T) \geq \tfrac{\pi^2}{2} \approx 4.9348 .
\end{equation}
Consequently, the Gromov-mass$\ast$ filling area of $\curr{\bS^1}$—and therefore its inscribed Riemannian filling area, is bounded below by this value.

Numerical optimizations suggest that $\tilde\omega$, as defined in the introduction, is not a calibration. However, it appears close enough that it might still yield useful lower bounds on the filling area of a circle. In this context, it is useful to obtain upper bounds on the $L^1$ norm
\[
\|p(f)\|_1 \defl \int_0^\pi \int_\alpha^\pi p_{\alpha,\beta}(f) \,d\beta\,d\alpha .
\]
It is unclear whether a uniform upper bound exists for all $f \in E(\bS^1) \setminus \bS^1$, though numerical optimizations indicate that the following question may have an affirmative answer.

\begin{qu}
	\label{qu:final}
Is it true that
\[
\|p(f)\|_1 \leq \tfrac{\pi^2}{2},
\]
with equality if and only if $f \in \bS^2_+\setminus \bS^1$?
\end{qu}

If the above question has a positive answer, then
\[
\|\tilde\omega\|_{\rm m} \leq \tfrac{\pi}{2}
\]
follows from the rather crude estimate
\[
\tilde\omega_f(v\wedge w) = \frac{1}{\pi}\int_0^\pi \int_\alpha^\pi p_{\alpha,\beta}(f) (v_\alpha w_\beta - v_\beta w_\alpha) \,d\beta\,d\alpha \leq \tfrac{1}{\pi}\|p(f)\|_1\|v\|_\infty\|w\|_\infty
\]
for all $v,w \in L^\infty([0,\pi))$. By Proposition~\ref{prop:massbound}, this implies
\begin{align*}
\bM_{\rm m}(T)\|\tilde\omega\|_{\rm m} \geq T(\tilde\omega) & = \curr{\bS^2_+}(\tilde\omega) = 2\pi .
\end{align*}
A positive answer to Question~\ref{qu:final} would imply $\bM_{\rm m\ast}(T) \geq \bM_{\rm m}(T) \geq 4$. The inequality between these two masses is justified by \cite[Proposition~3.14]{APT}. Although this lower bound with respect to the Gromov-mass$\ast$ is weaker than the previously obtained $\frac{\pi^2}{2}$, further improvements along these lines appear plausible.



\end{document}